\newif\ifArXiV

\ArXiVtrue	

\ifArXiV
\documentclass{article}
\else
\documentclass[authoryear,preprint]{elsarticle}
\fi

\usepackage{fullpage}

\usepackage{lmodern}
\usepackage[english]{babel}

\usepackage{amsmath,amssymb,amsthm,graphicx}
\usepackage{comment}
\usepackage{enumitem}
\usepackage{todonotes}
\usepackage{multirow}
\usepackage{mathtools}
\usepackage{tikz}
\usepackage{hyperref}
\usepackage{subcaption}
\usepackage{booktabs}
\usepackage{xspace}
\usepackage{lscape}
\usepackage{multicol}
\usepackage{rotating} 
\usepackage[normalem]{ulem} 
\usepackage{placeins} 

\newtheorem{theorem}{Theorem}

\newtheorem{lemma}[theorem]{Lemma}
\newtheorem{corollary}[theorem]{Corollary}

\newtheorem{observation}[theorem]{Observation}

\theoremstyle{remark}

\newtheorem{examplex}[theorem]{Example}

\usepackage[linesnumbered, noend, ruled]{algorithm2e}
\let\oldnl\nl
\newcommand{\nonl}{\renewcommand{\nl}{\let\nl\oldnl}}

\newcommand{\R}{\mathbb{R}}

\DeclareMathOperator{\st}{s.t.}
\newcommand{\obj}{z}
\newcommand{\loc}{J}
\newcommand{\cus}{I}
\newcommand{\pts}{N}
\newcommand{\feasSol}{P}

\newcommand{\mytag}[1]{(\hypertarget{#1}{\mathrm{#1}})}
\newcommand{\myref}[1]{\textnormal{(\hyperlink{#1}{#1})}}

\newcommand{\APCP}{d-$\alpha$-$p$CP\xspace}
\newcommand{\GAPCP}{$\alpha$-$p$CP\xspace}
\newcommand{\PCP}{d-$p$CP\xspace}
\newcommand{\GPCP}{$p$CP\xspace}
\newcommand{\BC}{B\&C\xspace}
\newcommand{\newOpt}{^\circ}
\newcommand{\yNewOpt }{y^{\circ}}
\newcommand{\pNewOpt }{p^{\circ}}
\newcommand{\yNewOpti}{y^{\circ,i}}

\newcommand{\blue}[1]{{\color{black}#1}}
\newcommand{\elli}[1]{{\color{black}#1}}
\newcommand{\green}[1]{{\color{black}#1}}
\newcommand{\orange}[1]{{\color{black}#1}}

\ifArXiV
\usepackage[affil-it]{authblk}
\usepackage[authoryear]{natbib}
\newenvironment{frontmatter}{}{}
\newenvironment{keyword}{\small \textbf{Keywords:}}{}
\let\address\affil

\fi

\begin{document}
		\begin{frontmatter}
			\title{
				Exact solution approaches for the discrete
				\texorpdfstring{$\alpha$}{alpha}-neighbor
				\texorpdfstring{$p$}{p}-center problem}
			%
			%
			
			\ifArXiV
			\author[1]{Elisabeth Gaar\thanks{elisabeth.gaar@jku.at}}
			\author[1,2]{Markus Sinnl\thanks{markus.sinnl@jku.at}}
			\affil[1]{Institute of Production and Logistics Management, 
			Johannes Kepler University Linz, Linz, Austria}
			\affil[2]{JKU Business School, Johannes Kepler University 
			Linz, Linz, Austria}			
			\date{}
			\maketitle
			
			\fi

\begin{abstract}
The discrete $\alpha$-neighbor $p$-center problem (\APCP) is an emerging 
variant of 
the  
classical $p$-center problem which recently got attention in 
literature. In this problem, we are given a discrete set of points and we need 
to locate 
$p$ facilities on these points in such a way that the maximum distance 
between each point where no facility is located and its $\alpha$-closest 
facility is minimized. The 
only existing algorithms in literature for solving the \APCP are 
approximation algorithms and two recently proposed heuristics.

In this work, we present two integer programming formulations for the \APCP,  
together with lifting of inequalities, valid inequalities, 
inequalities that do not 
change the optimal objective function value and variable fixing procedures. We 
provide 
theoretical results on the strength of the formulations and convergence results 
for the lower bounds obtained after applying the lifting procedures or the 
variable fixing procedures in an iterative fashion. Based on our 
formulations 
and theoretical results, we develop branch-and-cut (\BC) algorithms, which are 
further enhanced with a starting heuristic and a primal 
heuristic. 

We evaluate the effectiveness of our \BC algorithms using instances from 
literature. Our algorithms are able to solve 116 out of 194 instances from  
literature to proven optimality, with a runtime of under a minute for most of 
them. By doing so, we also provide improved solution values for 116 instances. 

\end{abstract}
\end{frontmatter}


\section{Introduction}

The \emph{$\alpha$-neighbor $p$-center problem} (\GAPCP), proposed by 
\citet{krumke1995generalization}, 
is an emerging variant 
of the classical $p$-center problem (\GPCP) \citep{hakimi1965} which recently 
got 
attention in literature 
\citep{chen2013optimal,callaghan2019optimal,sanchez2022grasp}. In this problem, 
we are given a set of points and we need to locate $p$ facilities. The goal is 
to locate the facilities in such a way that the maximum distance between each 
point and its $\alpha$-closest facility is minimized. We note that both 
a continuous and discrete version of the \GAPCP exist. In the continuous 
version, the facilities can be located anywhere on the plane, while in the 
discrete version the given points are also the potential facility locations.  
In the discrete version all the points where a facility gets located are not 
considered in the objective function. 
The \GAPCP 
can be seen as a robust variant of the \GPCP, where customers do not need to go 
to their closest facility, but also have additional $\alpha-1$ facilities 
nearby. Thus, the \GAPCP can be \blue{a} useful modeling approach for 
applications 
which are traditionally modeled as \GPCP, such as emergency service locations 
and relief actions in humanitarian crisis 
\citep{calik2013double,lu2013robust,jia2007modeling}, where robust solutions 
are highly relevant.

A formal definition of the discrete \GAPCP (\APCP) is as follows 
\citep{krumke1995generalization,sanchez2022grasp, mousavi2023exploiting}: We 
are given a set of points 
$\pts$, a positive integer $p < |\pts|$, and a positive 
integer $\alpha \leq p$. For each pair of points $i,j\in 
\pts$ we are given a distance $d_{ij}\geq 0$. A feasible solution consists of a 
subset $\feasSol \subseteq \pts$ of $|\feasSol|=p$ facilities, indicating which 
facilities are 
opened. Given a feasible solution
$\feasSol$, i.e., a set of open facilities $\feasSol$, the set of demand points 
is defined as 
$\pts \setminus \feasSol$, i.e., the set of demand points depends on the 
chosen feasible solution and consists of all points that 
are not opened. The 
$\alpha$-distance $d_\alpha(\feasSol,i)$ for a feasible solution $\feasSol$ and 
a given 
demand point $i \in 
\pts \setminus \feasSol$ 
is defined as
\begin{equation*}
 d_\alpha(\feasSol,i) =\min_{A\subseteq \feasSol, |A|=\alpha}\max_{j \in 
 A}\{d_{ij}\}, 
\end{equation*}
so the $\alpha$-distance $d_\alpha(\feasSol,i)$ gives the distance 
of 
$i$ to 
the 
$\alpha$-nearest open facility for the feasible solution $\feasSol$. The 
objective 
function 
value $f_\alpha(\feasSol)$ 
of a feasible solution $\feasSol$ is defined as 
\begin{equation*}
f_\alpha(\feasSol)=\max_{i \in \pts \setminus \feasSol}  d_\alpha(\feasSol,i). 
\end{equation*}
Using these definitions, the \APCP can be formulated as
\begin{equation*}
\min_{\feasSol\subseteq \pts,|\feasSol|=p} f_\alpha(\feasSol).
\end{equation*}

We note that the discrete \GPCP (\PCP, also known as \emph{vertex} \GPCP) is 
obtained by setting $\alpha=1$, if we assume that the 
distances $d_{ii}=0$ for all $i \in \pts$, i.e., if we assume that each demand 
point $i$ is covered 
if the same facility $i$ is opened. Moreover, 
instead of 
assuming that distances between all pairs of points $i,j$ are given, the 
problem can also be defined on a (non-complete) graph and the distances are 
defined as the shortest-path distances on this graph.  With respect to this, 
\cite{kariv1979algorithmic} show that the \PCP is NP-hard in general, 
but there 
are some classes of graphs such as trees, where the problem can be solved in 
polynomial time \citep{jeger1985algorithms}.

In this work, we present exact solution approaches for solving 
the \APCP. So far, solution approaches for the \GAPCP focused mostly on 
the continuous version of the problem. For this version, an iterative exact 
algorithm based on the connection to a version of the set cover problem is 
proposed in \cite{chen2013optimal}. We note that for the classical \GPCP such 
set cover-based approaches are well established (for both the continuous and 
discrete version of the problem), going back to the seminal work of 
\citet{minieka1970}. Recent set-cover based approaches for the classical 
\GPCP include \citet{chen2009,contardo2019scalable}. 

In \cite{callaghan2019optimal} such a set cover-based 
approach is used for the continuous version of both the \GPCP and the \GAPCP. 
For the 
\APCP the only existing solution approaches in 
literature are 
approximation algorithms 
\citep{chaudhuri1998p,khuller2000fault,krumke1995generalization} and heuristics 
\citep{sanchez2022grasp, 
mousavi2023exploiting}. More details on these approaches and on the 
\GPCP and other related problems are given in 
Section~\ref{sec:litreview}.

\subsection{Contribution and outline}

In this work, we present two different integer programming formulations for 
the \APCP. We also present valid inequalities, (iterative) lifting procedures 
for some of the inequalities, inequalities that do not change the optimal 
objective function value and 
(iterative) variable fixing procedures. We denote the inequalities that do not 
change the optimal objective function value as \emph{optimality-preserving} 
inequalities. The  
lifting procedures are based on lower bounds to the problem and can be viewed 
as extension of previous results for the \PCP in 
\citet{gaar2022scaleable}. We 
show that the lower bounds converge to a certain fractional set cover solution 
when 
applying the lifting procedure or the variable fixing procedure in an 
iterative fashion. We also show that we can obtain the optimal objective 
function value of the 
semi-relaxation (in this semi-relaxation, one set of binary variables of our 
formulation is kept binary and the other set of binary variables is relaxed) of 
our second formulation in 
polynomial time using iterative 
variable fixing. This can be seen as an extension 
of a result obtained by \citet{elloumi2004} for the \PCP and a 
fault-tolerant 
version of the \GPCP. Moreover, we provide 
polyhedral comparisons between the formulations. 

Based on these formulations and our theoretical results, we develop 
branch-and-cut (\BC) algorithms to solve the \APCP. These algorithms 
also contain a starting heuristic and a primal heuristic. We evaluate the  
effectiveness of our \BC algorithms using instances also used in 
\citet{sanchez2022grasp} and \citet{mousavi2023exploiting}. Our algorithms are 
able to solve 116 out of 194 instances from literature to proven optimality. We 
also provide improved solution values for 116 out of these 194 instances. Note 
that these instances are not all the same as the instances for 
which we manage to prove \blue{optimality}, as for some instances, the 
heuristics 
from literature already found the optimal solution (but of course could not 
prove optimality, as they are heuristics).

The paper is structured as follows:
In the remainder of this section, we discuss previous and related work in more 
detail. 
Section \ref{sec:formulationD} presents our first integer programming 
formulation 
together with valid inequalities, lifted versions of inequalities, 
optimality-preserving inequalities and variable fixings. 
Section~\ref{sec:formulationE} contains the same for our second formulation. In 
Section 
\ref{sec:polyhedral}, we provide a polyhedral comparison of the formulations 
and convergence results for the lower bounds after applying the lifting 
procedure or the variable fixing procedure in an iterative fashion. In 
Section \ref{sec:implementation} we describe the implementation details of our 
\BC algorithm, including the starting heuristic and the primal heuristic and 
separation routines. In Section~\ref{sec:results} the computational study is 
presented. Finally, Section \ref{sec:conclusion} concludes the paper.

\subsection{Previous and related work}
\label{sec:litreview}

The \GPCP is a fundamental problem in location science, dating 
back to 1965 \citep{hakimi1965}, which has spawned many variations over the 
years, see, e.g., the book-chapter by \cite{calik2019p}. 

The seminal work of \citet{minieka1970} presented the first exact approach for 
the \GPCP and also created a blueprint of a solution algorithm which over the 
years many other algorithms for either the \GPCP or also variants of it 
including the continuous \GAPCP, used as a starting point. \citet{minieka1970} 
showed that the question whether there exists a feasible solution to the \GPCP 
with a given objective function value can be posed as a certain set 
cover problem. As a consequence the \GPCP can be solved by iteratively solving 
such set cover problems. Over the years, many authors 
\citep{garfinkel1977,ilhan2001,ilhan2002,alKhedhairi2005,caruso2003,chen2009,contardo2019scalable}
 have expanded on this 
idea to present algorithms to solve the \GPCP.

Aside from these set cover-based approaches, there also exist 
several integer programming formulations for solving the \PCP to proven 
optimality. The classical textbook 
formulation of the 
problem 
(see e.g., \citep{daskin2013network}) uses facility opening variables and 
assignment variables and is known to have a bad linear programming relaxation 
(see, e.g., \citep{snyder2011fundamentals}). In \citet{elloumi2004} an 
alternative integer programming formulation is presented and the authors show 
that there are instances where the linear relaxation bounds are provably better 
than the bounds obtained by the classical textbook formulation. In 
\citet{ales2018} a modification of this formulation is presented. Regarding our 
second formulation, which we present in Section \ref{sec:formulationE}, we note 
that there exists a variant of the 
\APCP, in which 
every point $i$ must be covered $\alpha$-times, even if there is a facility 
opened at $i$. This variant is sometimes called \emph{fault-tolerant \GPCP} 
(see, e.g., Section 6 of \citet{elloumi2004}), although this name 
	has 
also been used for other 
problems in literature, including \APCP. In Section 6 of 
\citet{elloumi2004} a 
formulation for the fault-tolerant \GPCP extending their 
formulation for the \PCP is sketched. Our second formulation,  
can be seen as an adaption 
of this formulation, taking also into account the modification proposed in 
\citet{ales2018}. In \citet{elloumi2004} it is also shown that a so-called 
semi-relaxation of their formulation, where one of the two sets of binary 
variables is relaxed, can be solved in polynomial time. They also briefly 
discuss such a result for their formulation of the fault-tolerant \GPCP. We 
\blue{prove} a similar result for our second formulation for the \APCP in 
Section 
\ref{sec:comparison:e}. In \citet{calik2013double} another formulation for the  
\PCP is presented and the authors show that the linear programming 
relaxation 
of it has the same strength as the relaxation of the formulation of 
\citet{elloumi2004}.

In \citet{gaar2022scaleable} the classical textbook 
formulation was used as starting point for a projection-based approach, which 
projected out the assignment variables to obtain a new integer programming 
formulation for the 
 \PCP. Moreover, an iterative lifting scheme for the inequalities in 
the new formulation was presented. This lifting scheme is based on the lower 
bound obtained from solving the linear programming relaxation, in which then 
the lifted inequalities are included and everything is resolved in an iterative 
fashion. \citet{gaar2022scaleable} showed that this procedure converges and the 
lower bound at convergence is the same lower bound as the one of the 
semi-relaxation considered in \citet{elloumi2004}. Furthermore, 
\citet{gaar2022scaleable} also showed that the solution at convergence solves a 
certain fractional set cover problem. Our first formulation for the \APCP, 
which we present in Section \ref{sec:formulationD}, is 
based on the classical textbook formulation for the \PCP and is also suitable 
for the ideas of \citet{gaar2022scaleable} regarding lifting.

For \APCP the only existing algorithms with computational 
results
are the GRASP proposed by \citet{sanchez2022grasp} and the local search by 
\citet{mousavi2023exploiting}. Aside from these heuristics, there are also 
works on approximation algorithms 
\citep{chaudhuri1998p,khuller2000fault,krumke1995generalization} which do not 
contain computations. The best possible approximation factor of two is obtained 
by the algorithms presented in \citet{chaudhuri1998p, khuller2000fault} under 
the condition that the distances fulfill the triangle \elli{inequalities}. We 
note that 
in 
principle set cover-based approaches such as the one of \citet{chen2013optimal} 
also work for the \APCP, but \citep{chen2013optimal} focuses on the 
continuous \GAPCP and presents no computations for the \APCP.

\section{Our first formulation \label{sec:formulationD}}

In this section we present our first integer programming formulation for 
the 
\APCP. First, we describe the formulation in 
Section~\ref{sec:formulationD:formulation}. Then we derive valid 
inequalities, valid inequalities that are based on lower bounds, and 
optimality-preserving  inequalities in Section~\ref{sec:formulationD:cuts}. 
Next, we detail conditions which allow to fix some of the variables in the 
linear relaxation in 
Section~\ref{sec:formulationD:fixing}.
Finally, we provide some insight on what happens if we relax one set of binary 
variables of our formulation in 
Section~\ref{sec:formulationD:integrality}.

\subsection{Formulation}\label{sec:formulationD:formulation}

Our first integer programming formulation of the \APCP can be viewed as 
extension 
of a 
classical formulation of the \PCP (see, e.g., \citep{daskin2000} and  
\citep{gaar2022scaleable}). We refer to this classical formulation of the 
 \PCP as $\mytag{PC1}$ following the notation of 
\citet{gaar2022scaleable}. \blue{The formulation $\mytag{PC1}$ as well as any 
other formulations of the \PCP which are mentioned in the remainder of this 
work can be found in Appendix~\ref{sec:app1}.}

Let the binary 
variables $y_j$ for all $j\in \pts$ indicate whether a facility 
is 
opened at point $j$. Let the binary variables $x_{ij}$ for all $i,j \in \pts$ 
with $i\neq j$ indicate whether the 
point $i$ 
is assigned to the open facility $j$. Let the continuous variables $z$ measure 
the 
distance in the objective function. Then the \APCP can be 
formulated as 
\begin{subequations}
	\begin{alignat}{3}
	\mytag{APC1} \qquad
	& \min & \obj \phantom{iiiii} \label{pc1:z}  \\ 
	& \st~  & \sum_{j \in \pts} y_j &= p \label{pc1:sumy} \\       
	&& \sum_{j \in \pts\setminus \{i\}} x_{ij} & =  \alpha(1-y_i) \qquad && 
	\forall i \in \pts \label{pc1:sumx} \\
	&& x_{ij} &\leq y_j && \forall i,j \in \pts, i \neq j \label{pc1:xy}\\   
	&& d_{ij} x_{ij} & \leq  \obj && \forall i, j \in \pts, i \neq j 
	\label{pc1:dx}\\
	&& x_{ij} &\in  \{0,1\} \qquad&& \forall i , j \in \pts\label{pc1:xbin}, i 
	\neq j  	
	\\
	&& y_{j} &\in  \{0,1\} && \forall j \in \pts \label{pc1:ybin}	
	\\
	&& \obj & \in \mathbb{R}.	\label{pc1:zVar}
	\end{alignat}
\end{subequations}

The constraints \eqref{pc1:sumy} ensure that exactly $p$ facilities are opened. 
The constraints \eqref{pc1:sumx} make sure that for each point $i \in \pts$, 
the 
point is either used for opening a facility, or it is assigned to   
$\alpha$ other open facilities. The constraints \eqref{pc1:xy} ensure that if a 
point 
$i$ 
is assigned to a facility at point $j$, then the facility at point $j$ is 
opened. 
The constraints \eqref{pc1:dx} ensure that $z$ takes at least the value of the 
distance from $i$ to $j$ if $i$ is assigned to $j$. Thus, $z$ will take at 
least the maximum distance for assigning $i$ to $\alpha$ facilities, since 
constraints \eqref{pc1:sumx} ensure the assignment of $i$ to $\alpha$ 
facilities in case it is not opened. The 
objective function \eqref{pc1:z} minimizes~$z$, i.e., it minimizes the maximum 
assignment distance. 
The formulation \myref{APC1} has $O(|\pts|^2)$ variables and $O(|\pts|^2)$ 
constraints.

Note that in the formulation \myref{PC1} for the classical formulation of the 
 \PCP, the constraint 
\eqref{pc1:dx} is included in an aggregated fashion as $ \sum_{j \in 
\pts} 
d_{ij} x_{ij}  \leq  \obj$ for all $i \in 
\pts$. Furthermore, in the classical  \PCP also open facilities are 
included in 
the demand points. Thus, in \myref{PC1} the 
variables $x_{ij}$ are required also for $i = j$, and the right hand-side is 
$\alpha$ and not 
$\alpha(1-y_i)$ in~\eqref{pc1:sumx}. 

\subsection{Strengthening inequalities}\label{sec:formulationD:cuts}
Due to the fact that \myref{PC1} is typically considered to have bad linear 
programming bounds (see, e.g., \citep{snyder2011fundamentals}) for the  
\PCP, it 
could be expected that also \myref{APC1} has a linear relaxation that provides 
a poor bound. In fact, we confirmed this in preliminary computations, see 
also Section \ref{sec:ingredients}. In Section \ref{sec:comparison:d} we 
present some theoretical results on the effect of adding the inequalities 
described in this section to \myref{APC1}.

\subsubsection{Valid inequalities}

The next theorem presents two sets of valid inequalities for \myref{APC1}.

\begin{theorem}\label{thm:validInequDaskin}
The inequalities
\begin{subequations}
\begin{alignat}{3}
	\sum_{j \in \pts\setminus \{i\}} d_{ij} x_{ij} & \leq  \alpha \obj \qquad 
	&&  
	\forall i \in \pts
	\label{cut:sumdijxij_alphaz}\\
		y_i + x_{ij} & \leq  1 &&  
	\forall i,j \in \pts, i \neq j
	\label{cut:yixij}
\end{alignat}
\end{subequations}
are valid inequalities for the formulation \myref{APC1} for the \APCP, i.e., 
when adding \eqref{cut:sumdijxij_alphaz} and  
\eqref{cut:yixij} to \myref{APC1}, the set of feasible solutions 
does not change.
\end{theorem}
\begin{proof}
	Clearly \eqref{cut:sumdijxij_alphaz} holds for any feasible 
	solution for \myref{APC1}, as in this case 
	$\sum_{j \in \pts\setminus \{i\}} d_{ij} x_{ij}$ is either zero (in case 
	$i$ is opened) or the sum of the 
	distances of 
	the closest, second-closest, \dots, $\alpha$-closest facility to point $i$, 
	which is at most $\alpha$ times the distances of the $\alpha$-closest 
	facility measured as $z$.
	
	Furthermore, it is obvious that \eqref{cut:yixij} holds for any  
	feasible solution for \myref{APC1}, as $i$ can not be assigned to any point 
	$j 
	\in \pts\setminus \{i\}$ if $i$ is opened.
\end{proof}

\subsubsection{Valid inequalities based on lower bounds}
Given a lower bound on the optimal objective function value of the \APCP, the 
inequalities 
\eqref{pc1:dx} and \eqref{cut:sumdijxij_alphaz} can be lifted, as we show next. 
The lifting is based on a similar idea recently proposed in 
\cite{gaar2022scaleable} for the  \PCP.

\begin{theorem}\label{thm:liftingValid}
	Let $LB$ \blue{be} a lower bound on the optimal objective function value 
	of the \APCP. 
	Then
	\begin{subequations}
		\begin{alignat}{3}
		LB y_i +\max\{LB,d_{ij}\} x_{ij} &\leq z &&  \forall i,j \in \pts, i 
		\neq j 
		\label{pc1:dxlifted}\\
		\max\{LB,d_{ij}\} x_{ij} &\leq z &&  \forall i,j \in \pts, i \neq j 
		\label{pc1:dxliftedRelax} \\
		\alpha LB y_i +\sum_{j \in \pts \setminus \{i\}} \max\{LB,d_{ij}\} 
		x_{ij} & 
		\leq  \alpha \obj \qquad &&  \forall i \in \pts
		\label{cut:sumdijxij_alphazlifted}
		\end{alignat}
	\end{subequations}
	are valid inequalities for the formulation \myref{APC1} for the \APCP, 
	i.e., when adding 
	\eqref{cut:sumdijxij_alphazlifted}, \eqref{pc1:dxlifted} and 
	\eqref{pc1:dxliftedRelax} to 
	\myref{APC1}, the set of feasible solutions does not change.
\end{theorem}
\begin{proof}	
	For the inequalities \eqref{pc1:dxlifted}, we note that due to constraints 
	\eqref{pc1:sumx} at most one of $y_i$ and $x_{ij}$ can take the value one 
	in 
	any 
	feasible solution of \myref{APC1}. Thus, the left hand-side can be at most 
	$\max\{LB,d_{ij}\}$, which is clearly a valid lower bound for $z$.
	
	Clearly the inequalities~\eqref{pc1:dxliftedRelax} are just a relaxation of 
	\eqref{pc1:dxlifted} and therefore also valid.
	The validity of the inequalities \eqref{cut:sumdijxij_alphazlifted} follows 
	from combining the arguments from the proof of the validity of inequalities 
	\eqref{cut:sumdijxij_alphaz} with the proof for the validity of the 
	inequalities \eqref{pc1:dxlifted}.
\end{proof}
Theorem~\ref{thm:liftingValid} allows us to add new valid inequalities to the 
linear relaxation of \myref{APC1}, as soon as we have a lower bound $LB$.
We present an iterative scheme exploiting this fact in 
	Section~\ref{sec:comparison:d}, where we also analyze the convergence 
	behavior of this scheme.

\subsubsection{Optimality-preserving inequalities}

Next we consider optimality-preserving inequalities. These inequalities may cut 
off some feasible solutions of \myref{APC1}, but do not change the optimal 
objective function value. In other words, there exists at least one optimal 
solution to \myref{APC1} which fulfills all these inequalities.

To present the inequalities, 
let 
$S_{ij} = 
\{j' \in \pts: (d_{ij'} < d_{ij}) \text{ or } (d_{ij'} = 
d_{ij} 
\text{ and } j' < j)\}$, i.e., $S_{ij}$ is the set of points $j'$ such 
that $j'$ is closer to $i$ than $j$, or such that $j'$ and $j$ are  
at the same distance to $i$ and $j'$ has a smaller index than~$j$. Thus, for 
any point $i$, the 
sets $S_{ij}$ induce an ordering of all points according to their distance to 
$i$ and their index. We denote this order with $\sigma_i$. 
\begin{theorem}\label{thm:optCutDaskin}
	The inequalities
	\begin{subequations}
		\begin{alignat}{3}
		\sum_{j \in \pts_\alpha} y_{j} + \sum_{j \in \pts \setminus  
			\bigcup_{j' \in \pts_\alpha} (S_{ij'} \cup \{i,j'\} ) } x_{ij} 
		& \leq \alpha \qquad  && \forall i \in \pts, \forall \pts_\alpha 
		\subseteq 
		\pts, |\pts_\alpha| = \alpha \label{cut:assignedToLaterSet}\\
		\intertext{and, if $UB$ is the objective function value of a feasible 
		solution of the \APCP, the inequalities }
		\sum_{j \in \pts \setminus \{i\}: d_{ij}\leq UB} y_{j} &\geq 
		\alpha(1-y_i) \qquad &&\forall 
		i \in \pts \label{eq:optcuts}
		\end{alignat}
	\end{subequations}
are optimality-preserving inequalities for the formulation \myref{APC1} for the 
\APCP, i.e., when adding \eqref{cut:assignedToLaterSet} and \eqref{eq:optcuts} 
to \myref{APC1}, the optimal objective function value does not change.
\end{theorem}
\begin{proof}
Note that the set $\pts 
\setminus 
\bigcup_{j' \in \pts_\alpha} (S_{ij'} \cup \{i,j'\})$  
that appears in~\eqref{cut:assignedToLaterSet} can 
alternatively be 
described as the set  
$\{j \in \pts \setminus \{i\}: (d_{ij} > \max_{j' \in \pts_\alpha} \{d_{ij'}\}) 
\text{ or } (d_{ij} =  \max_{j' \in \pts_\alpha} \{d_{ij'}\} \text{ and } j 
> \max\{j' \in \pts_\alpha: d_{ij} = d_{ij'} \} )\}$ and is the set of 
facilities $j$ that 
are further away to 
$i$ than the furthest facility in $\pts_\alpha$ according to $\sigma_i$. 

The inequality~\eqref{cut:assignedToLaterSet} ensures that
if a certain number $\beta$ of facilities, that are at most as far away from 
$i$ 
than the furthest facility in $\pts_\alpha$, \blue{are} opened (and thus, $i$ 
can be 
assigned to these $\beta$ facilities), the point $i$ is assigned at most 
$\alpha - 
\beta$ times to facilities that are further away from $i$ than the furthest 
facility in $\pts_\alpha$. 
Clearly this is fulfilled for any optimal solution of \myref{APC1}, where every 
facility $i$ is 
assigned to those $\alpha$ opened facilities, that are the $\alpha$ closest 
facilities to $i$ according to $\sigma_i$. Thus, 
adding~\eqref{cut:assignedToLaterSet} to \myref{APC1} does not change the 
optimal objective function value.

Next consider the inequalities~\eqref{eq:optcuts}.
If $y_i$ is one in an optimal solution of \myref{APC1}, the 
inequality~\eqref{eq:optcuts} is clearly satisfied and thus it does not cut 
off any optimal solution. Now suppose $y_i$ is zero, so location 
$i$ is not opened. As we know that a feasible solution with objective function 
value 
$UB$ exists, it follows that $i$ must be assigned to $\alpha$ facilities at
distance at most $UB$ to location $i$ and these $\alpha$ facilities must be 
opened. Therefore, also in this case~\eqref{eq:optcuts} is fulfilled.
\end{proof}

Note that the inequalities~\eqref{cut:assignedToLaterSet} from 
Theorem~\ref{thm:optCutDaskin} force an assignment of any location $i$ to those 
$\alpha$ opened facilities, that are the $\alpha$ closest opened
facilities to $i$ according to $\sigma_i$. They do so,  
even when also other assignments would not change the objective function value. 
Thus, in a sense~\eqref{cut:assignedToLaterSet} are symmetry breaking 
constraints that forbid certain similar solutions.

\subsection{Variable fixing}\label{sec:formulationD:fixing}

Next, we present a variable fixing condition which can be utilized whenever a 
feasible 
solution to 
the \APCP is known. This fixing of variables cuts off feasible solutions, but 
it does not cut off any optimal 
solution, i.e., it is optimality-preserving.

\begin{theorem}\label{thm:optcuts}
	Let $UB$ be the objective function value of a feasible solution of 
the \APCP. Then when adding the constraints
	\begin{subequations}
		\begin{alignat}{3}
		x_{ij} & = 0 \qquad && \forall i,j \in \pts, i \neq j, d_{ij} > UB 
		\label{eq:optcuts_xij}
		\end{alignat}
	\end{subequations}
	to \myref{APC1}, no optimal solution is cut off.
\end{theorem}

\begin{proof}
	Clearly in an optimal solution no point $i$ can be assigned to a point $j$ 
	that 
	is further away than $UB$, thus~\eqref{eq:optcuts_xij} is satisfied for any 
	optimal solution.
\end{proof}	


\subsection{Relaxing the assignment 
variables}\label{sec:formulationD:integrality}

We now turn our attention to an interesting aspect of \myref{APC1}. 
The 
classical formulation \myref{PC1} of the  \PCP has the following 
nice property:
When relaxing the $x$-variables in \myref{PC1}, i.e., 
replacing $x_{ij} \in \{0,1\}$ with $0 \leq x_{ij}$ 
for all $i,j \in \pts$, then the optimal objective function value does not 
change. Hence, it is not necessary to force the $x$-variables to be 
binary in order to obtain the optimal objective function value of \myref{PC1}. 
This is for example exploited by \citet{gaar2022scaleable}. 
Interestingly, this is not the case anymore for the \APCP.
To investigate this in detail, let $\mytag{APC1-Rx}$ be the formulation 
\myref{APC1} with relaxed $x$-variables, i.e., \myref{APC1-Rx} is
\myref{APC1} without~\eqref{pc1:xbin} and with the constraints $0 \leq x_{ij}$ 
for all $i,j \in \pts$ with $i\neq j$. 

We first consider an example to get some insight. The example is illustrated in 
Figure \ref{fig:ex:DaskinBasic}.

\tikzstyle{vertex}=[circle,draw=black,thick,fill=black!0,minimum 
size=20pt,inner 
sep=0pt]
\tikzstyle{edge} = [draw,thick]
\tikzstyle{weight} = [font=\small]

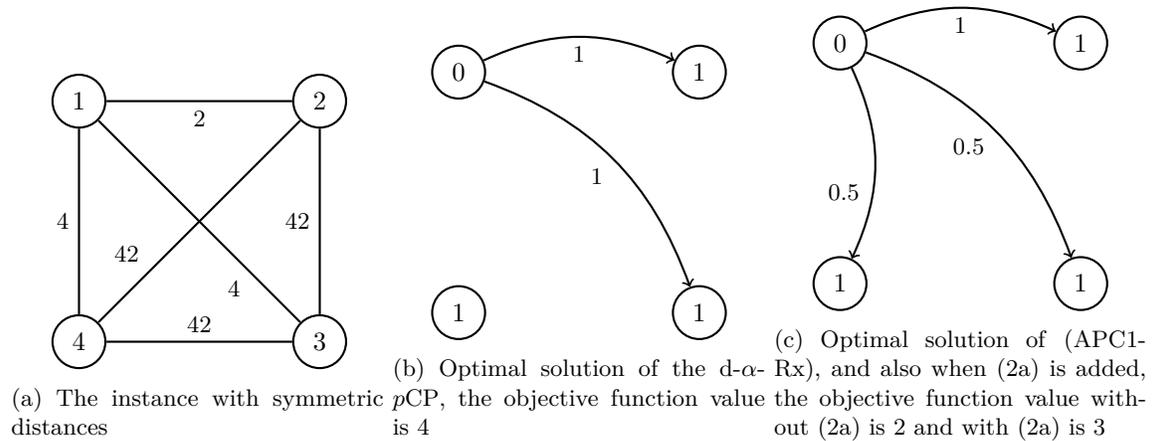
\begin{figure}[tbh]
	\begin{center}

	\begin{subfigure}{0.3\textwidth}
		\centering
		\begin{tikzpicture}[scale=1.6, auto,swap]
		\foreach \pos/\name in {{(1,1)/1}, {(1,-1)/4}, {(3,1)/2}, 
			{(3,-1)/3}}
		\node[vertex] (\name) at \pos {$\name$};
		\foreach \source/ \dest /\weight in {
			1/2/2,1/4/4,2/3/42,3/4/42}
		\path[edge] (\source) -- node[weight] {$\weight$} (\dest);
		\foreach \source/ \dest /\weight in {
			1/3/4,2/4/42}
		\path[edge] (\source) -- node[near end,weight] {$\weight$} (\dest);
		\end{tikzpicture}
		\caption{The instance with symmetric distances\label{ex:daskinLRxa}}
	\end{subfigure}
	\begin{subfigure}{0.3\textwidth}
		\centering
	\begin{tikzpicture}[scale=1.6, auto,swap]
	\foreach \pos/\name/\value in {{(1,1)/1/0}, {(1,-1)/4/1}, {(3,1)/2/1}, 
		{(3,-1)/3/1}}
	\node[vertex] (\name) at \pos {$\value$};
	\foreach \source/ \dest /\weight in {
		1/2/1,1/3/1}
	\path[edge] (\source) edge[bend left=25,draw,thick,->] node[weight] 
	{$\weight$} (\dest);
	\end{tikzpicture}
	\caption{Optimal solution of the \APCP, the objective function value is $4$ 
	\label{ex:daskinLRxb}}
\end{subfigure}
\begin{subfigure}{0.3\textwidth}
	\centering
	\begin{tikzpicture}[scale=1.6, auto,swap]
	\foreach \pos/\name/\value in {{(1,1)/1/0}, {(1,-1)/4/1}, {(3,1)/2/1}, 
		{(3,-1)/3/1}}
	\node[vertex] (\name) at \pos {$\value$};
	\foreach \source/ \dest /\weight in {
		1/2/1,1/3/0.5}
	\path (\source) edge[bend left=25,draw,thick,->] node[weight] {$\weight$} 
	(\dest);
	\foreach \source/ \dest /\weight in {
		1/4/0.5}
	\path (\source) edge[bend left=25,draw,thick,->] node[near end,weight] 
	{$\weight$} (\dest);
	\end{tikzpicture}
	\caption{Optimal solution of \myref{APC1-Rx}, and also when 
	\eqref{cut:sumdijxij_alphaz} is 
	added, the objective function value without \eqref{cut:sumdijxij_alphaz} is 
	$2$ and with 
	\eqref{cut:sumdijxij_alphaz} is $3$ \label{ex:daskinLRxc}}
\end{subfigure}
\caption{Illustration of Example 
\ref{ex:DaskinBasic}, \blue{in which} $p=3$ and 
$\alpha=2$. \label{fig:ex:DaskinBasic}
The value in the nodes in 
Figure~\ref{ex:daskinLRxa} is the index of the node and 
the values near the arcs are 
the distances.
The values in the nodes in 
Figures \ref{ex:daskinLRxb} and \ref{ex:daskinLRxc} 
are the values of the $y$-variables in the optimal 
solution, and the values near the arcs are the values of 
the $x$-variables 
in the optimal solution. If an 
arc is not drawn in a solution, this means the 
corresponding $x$-variable takes value zero.}
\end{center}
\end{figure}

\begin{examplex}
	\label{ex:DaskinBasic}
	Let $\pts=\{1,2,3,4\}$, $p = 3$, $\alpha = 2$, 
	$d_{1,2} = 2$, 
	$d_{1,3} = d_{1,4} = 4$,
	$d_{2,3} = d_{2,4} = d_{3,4} = 42$ and $d_{ij} = d_{ji}$ for all $i,j\in 
	\pts$ with $i \neq j$.
	
	In this example, it is easy to see that one optimal solution 
	$(x^*,y^*,z^*)$ for the 
	formulation \myref{APC1} of the \APCP is given as $y^*_1 = 
	0$, $y^*_2 = y^*_3 = y^*_4 = 1$, $x^*_{1,2} = x^*_{1,3} = 1$, all other 
	values of 
	$x^*_{ij}$ are equal to $0$, and $z^*=4$. Thus, the optimal objective 
	function 
	value of \myref{APC1} is $z^*=4$.
	
	Next consider the solution $(x',y',z')$, where $y' = y^*$, $x'_{1,2} =1$, 
	$x'_{1,3} = x'_{1,4} = 0.5$, all other values of 
	$x'_{ij}$ are equal to $0$, and $z'=2$. 
	Clearly, $(x',y',z')$ is 
	feasible for \myref{APC1-Rx} and therefore 
	the optimal 
	objective function 
	value of \myref{APC1-Rx} is at most $2$. Indeed, the optimal objective 
	function 
	value of \myref{APC1-Rx} is $2$ and thus not equal to the 
	optimal objective 
	function value of \myref{APC1}.
	
	It is easy to see, that the solution $(x',y',z')$ is not feasible 
	anymore for 
	\myref{APC1-Rx} when 
	the 
	inequalities~\eqref{cut:sumdijxij_alphaz} are added, as $(x',y',z')$ 
	does not fulfill~\eqref{cut:sumdijxij_alphaz} for $i=1$.
	However, the solution $(x'',y'',z'')$ with $x'' = 
	x'$, $y'' = y'$ and $z'' = 3$ is feasible. So the 
	optimal objective function value of \myref{APC1-Rx} 
	with~\eqref{cut:sumdijxij_alphaz} is at most $3$, and indeed it is exactly 
	$3$. Thus, it is again not  equal 
	to the optimal objective 
	function value of \myref{APC1}.
	
	Finally, $(x'',y'',z'')$ is not feasible for \myref{APC1-Rx} 
	with~\eqref{cut:assignedToLaterSet}, as  
	$(x'',y'',z'')$ does not fulfill the inequality $y''_2 + y''_3 + 
	x''_{1,4} \leq 
	\alpha$, 
	which is~\eqref{cut:assignedToLaterSet} for $i=1$ and $\pts_\alpha = 
	\{2,3\}$. Indeed, the optimal objective function value of \myref{APC1-Rx} 
	with~\eqref{cut:assignedToLaterSet} 
	coincides with the optimal objective function value of \myref{APC1}.
\end{examplex}

Example~\ref{ex:DaskinBasic} \blue{shows that} 
\myref{APC1-Rx} does not necessarily 
give the same \elli{optimal} objective function \elli{value} as \myref{APC1}, 
but \elli{there exist instances where} \blue{after} adding 
\eqref{cut:assignedToLaterSet}, the optimal 
objective function values coincide. The next result shows that this behavior is 
not a 
coincidence. 

\begin{theorem}\label{thm:relaxingxAPC1}
	\myref{APC1-Rx} 
	with~\eqref{cut:assignedToLaterSet} has 
	the same optimal objective function value as \myref{APC1}. 
\end{theorem}
\begin{proof}
	Let $(x^*,y^*,z^*)$ be an optimal solution of \myref{APC1-Rx} 
	with~\eqref{cut:assignedToLaterSet}.
	Because of Theorem~\ref{thm:optCutDaskin}, the optimal objective function 
	value
	of \myref{APC1} is at least $z^*$, so it is enough to show that $z^*$ is at 
	least the optimal objective function 
	of \myref{APC1}. 
	
	To do so, we construct a solution $(x\newOpt,y\newOpt,z\newOpt)$ 
	that is feasible for \myref{APC1} with $z^* \geq z\newOpt$.
	Towards this end consider some $i \in \pts$. 
	If $y^*_i = 1$, then let $\pts^i_\alpha = \emptyset$.
	Otherwise, so if $y^*_i = 0$, 
	let $j_{i,k}$ be such that 
	$\sum_{j \in S_{ij_{ik}}\setminus \{i\}} x^*_{ij} \leq k-1$ and such that
	$\sum_{j \in (S_{ij_{ik}} \cup \{j_{ik}\})\setminus \{i\}} x^*_{ij} > 
	k-1$ for all $k \in \{1, 2, \dots, \alpha\}$. Clearly such $j_{i,k}$ exist 
	because of~\eqref{pc1:sumx} and $y^*_i = 0$. 
	Let $\pts^i_\alpha = \{j_{i,k}: k \in \{1, \dots, \alpha\}\}$. 
	Due to the fact that $x^*_{ij} \leq 1$ for all $j \in \pts\setminus \{i\}$, 
	all $j_{i,k}$ are distinct for different values of $k$ by construction, so 
	$|\pts^i_\alpha| = \alpha$. Furthermore, 
	\begin{align}\label{yj1forNia}	
	y^*_{j} = 1 \quad \forall  
	j \in \pts^i_\alpha,
	\end{align}
		 because for such $j$ by construction $x^*_{ij} > 
	0$ and 
	$y_{j}^* \geq x^*_{ij}$ because of~\eqref{pc1:xy}.
	As a consequence, \eqref{cut:assignedToLaterSet} for $\pts_\alpha = 
	\pts^i_\alpha$ implies that 
	$\sum_{j \in \pts \setminus  
		\bigcup_{j' \in \pts^i_\alpha} (S_{ij'} \cup \{i,j'\} ) } x^*_{ij} \leq 
		0$ and hence $x^*_{ij} = 0$ for all $j \in \pts \setminus  
		\bigcup_{j' \in \pts^i_\alpha} (S_{ij'} \cup \{i,j'\} )$.
		This, together with~\eqref{pc1:sumx} and the fact that $j_{i,\alpha}$ 
		is the facility in $\pts^i_\alpha$ furthest away from $i$ according to 
		$\sigma_i$, implies 
		that 
		$\sum_{j \in (S_{ij_{i,\alpha}} \cup \{j_{i,\alpha}\})\setminus
		\{i\}} x^*_{ij} 
			= \alpha$. Due to the definition of $j_{i,\alpha}$ this implies 
			that $x^*_{ij_{i,\alpha}} = 1$. Thus 
			\begin{align}\label{zgeqdij}
			z^* \geq d_{ij_{i,\alpha}} 
			= \max_{j \in \pts^i_\alpha} \{d_{ij}\}
			\end{align}
			because 
			of~\eqref{pc1:dx}.
			
	Finally, let $y\newOpt = y^*$, $z\newOpt = z^*$ and let $x\newOpt_{ij} = 1$ 
	if $j \in \pts^i_\alpha$ and $x\newOpt_{ij} = 0$ otherwise. Clearly,  
	$x\newOpt$ and $y\newOpt$ are binary and $y\newOpt$ 
	satisfies~\eqref{pc1:sumy}. Furthermore, by construction of 
	$\pts^i_\alpha$, also~\eqref{pc1:sumx} and, in particular because 
	of~\eqref{yj1forNia}, \eqref{pc1:xy} are satisfied. Furthermore, 
	\eqref{pc1:dx} is fulfilled because of~\eqref{zgeqdij}. 
	As a consequence, $(x\newOpt,y\newOpt,z\newOpt)$ 
	is feasible for \myref{APC1} with $z^* = z\newOpt$, which finishes the 
	proof.
\end{proof}
As a consequence of Theorem~\ref{thm:relaxingxAPC1}, relaxing the $x$-variables 
in \myref{APC1} without changing the optimal objective function is possible, 
\blue{whenever} 
the inequalities~\eqref{cut:assignedToLaterSet} are added. 
Example~\ref{ex:DaskinBasic} shows that sometimes these inequalities are 
indeed necessary to preserve the optimal objective function value.

Note that also additionally including \eqref{cut:sumdijxij_alphaz} and 
\eqref{cut:yixij} into \myref{APC1-Rx} 
with~\eqref{cut:assignedToLaterSet} does not change the optimal objective 
function value, as these inequalities are valid for \myref{APC1}.

\section{Our second formulation \label{sec:formulationE}}

In this section we detail our second integer programming formulation of the 
\APCP. First, we present the formulation in 
Section~\ref{sec:formulationE:formulation}. Then we derive a set of valid 
inequalities in Section~\ref{sec:formulationE:cuts}. 
Finally, we present conditions which allow to fix some of the variables in the 
linear relaxation in 
Section~\ref{sec:formulationE:fixing}.

\subsection{Formulation}\label{sec:formulationE:formulation}
Our second formulation can be viewed as an extension of the formulation for the 
 \PCP proposed by \cite{ales2018}, which in turn is a 
refinement 
of a formulation of \cite{elloumi2004} with less constraints and the same 
linear relaxation bound. We denote the formulation 
of the  \PCP by \cite{elloumi2004} as $\mytag{PCE}$ in the same fashion 
as  
\cite{gaar2022scaleable}. Moreover, we denote the formulation of the  
\PCP 
by \cite{ales2018} as $\mytag{PCA}$. 
\elli{Both \myref{PCE} and \myref{PCA} can be found in Appendix~\ref{sec:app1}.}

Let $D=\{d_{ij}: i,j \in 
\pts, i \neq j\}$ denote the set of all possible distances
and let $d_1$, $\ldots$, $d_{K}$ be the values in~$D$, i.e., $D=\{d_1, \ldots, 
d_{K}\}$. 
It is easy to see that the optimal objective function value of the \APCP 
is in $D$ 
and there are at most $(|\pts|-1)|\pts|$ potential optimal values. 
Furthermore, let $D_i = \left(\bigcup_{j \in \pts\setminus \{i\}} 
\{d_{ij}\} \right)\setminus \{d_1\}$, so $D_i$ is the set of all distances that 
are relevant for point $i$, except for the smallest overall distance.

In this formulation, we have a binary variable $u_k$ for each $k=2,\ldots, K$. 
This 
variable indicates whether the optimal objective function value of 
the \APCP is 
greater than or equal to 
$d_k$, i.e., $u_k$ is one if and only if the optimal objective function value 
of the \APCP is at least $d_k$. 
Aside from 
the $u$-variables, we also have the binary variables $y_j$ for all $j\in \pts$ 
to 
indicate whether a facility is opened at point $j$ similar to the previous 
formulation. The formulation is denoted as \myref{APC2} and reads as
\begin{subequations}
	\begin{alignat}{3}
	\mytag{APC2} \qquad
	& \min & d_1 + \sum_{k=2}^{K} (d_k &- d_{k-1})u_k  \label{pc3:ojb}  \\ 
	& \st~ &  \sum_{j \in \pts} y_j &= p \label{pc3:sumy} \\       
	&& u_{k-1} &\geq u_{k} && \forall k \in \{3, \dots, K\} \label{pc3:ulink} 
	\\  
	&& \alpha u_k + \sum_{j \in \pts \setminus \{i\} :  d_{ij} < d_k} 
	y_{j} &\geq  \alpha(1- 
	y_i) \qquad && \forall i \in \pts, 
	\forall d_k \in D_i \label{pc3:sumyu} \\
	&& u_{k} &\in  \{0,1\} \qquad&& \forall  k \in \{2, \dots, K\} 
	\label{pcE:ubin}	\\
	&& y_{j} &\in  \{0,1\} && \forall j \in \pts.
	\label{pcE:ybin}
	\end{alignat}
\end{subequations}

The constraints \eqref{pc3:sumy} ensure that exactly $p$ facilities are opened. 
The constraints \eqref{pc3:ulink} make sure that if the variable $u_k$ is one, 
indicating that the optimal objective function value is at least $d_k$, then 
also all 
variables with smaller index are one. These constraints ensure that the 
objective function \eqref{pc3:ojb} measures the objective function value 
correctly: \blue{in}~\eqref{pc3:ojb}, the coefficient of $u_k$ 
is always the 
distance-increment from $d_{k-1}$ to $d_k$. Thus, we need that all 
$u_{k'}$ with $k'\leq k$ are set to one in order to get a value of $d_k$ in 
the objective function. Finally, constraints \eqref{pc3:sumyu} model that for 
each $i \in N$, the $u$-variables are set in such a way that $u_k$ is one, if 
$i$ is not opened and 
the $\alpha$-nearest open facility to $i$ has distance at least $d_k$: In case 
a facility is opened at point $i$, 
i.e., $y_i$ is one, the constraints are trivially fulfilled. In case no 
facility is opened at point~$i$, i.e., $y_i$ is zero, the constraints force 
$u_k$ to be one, or that at least $\alpha$ facilities closer than distance 
$d_k$ to $i$ are opened. The formulation \myref{APC2} has $O(|\pts|^2)$ 
variables and $O(|\pts|^2)$ 
constraints.

In comparison to the formulation \myref{PCA} for the  \PCP, we have 
several modifications in \myref{APC2} for the \APCP. First, we have 
the right hand-side $1-y_i$ instead of just $1$ and the sum over all $j \in 
\pts \setminus \{i\}$ instead of over all $j \in \pts$ in \eqref{pc3:sumyu} as 
a consequence of the fact that in the \APCP opened facilities do not serve as 
demand 
points. Furthermore, we have a coefficient $\alpha$ for $u_k$ and $1-y_i$ in 
\eqref{pc3:sumyu}. Finally, we do not include $K$ into the set $D_i$, 
independent from whether there is a facility $j$ with distance $d_{ij} = d_K$ 
or not. 
This does not influence the correctness of the model, as in the case that there 
is no facility $j$ with $d_{ij} = d_K$ for some $i$, then for \eqref{pc3:sumyu} 
for~$i$ and $k=K$, the sum $\sum_{j \in \pts \setminus \{i\} :  d_{ij} < d_k} 
y_{j}$ is equal to $p-y_i$. 
\elli{This implies that the constraint becomes $\alpha u_K \geq \alpha(1-y_i) - 
(p-y_i)$, which is always satisfied because $1 \leq \alpha \leq p$ holds.} 
Therefore the constraint does not impose a 
restriction on $u_K$, \elli{and $K$ can be omitted when defining the set $D_i$.}

\subsection{Strengthening inequalities}\label{sec:formulationE:cuts}

We have the following valid inequalities.

\begin{theorem}\label{thm:validInequE}
	The inequalities
		\begin{alignat}{3}
		u_k + y_i & \geq  1 \qquad &&  
		\forall i \in \pts, d_k \in D_i, |\{j \in \pts\setminus \{i\}: d_{ij} < 
		d_k\}| < \alpha \label{cut:ukyj}
		\end{alignat}
	are valid inequalities for the formulation \myref{APC2} for the \APCP, 
	i.e., when adding \eqref{cut:ukyj} to \myref{APC2}, the 
	set of feasible solutions does not change.
\end{theorem}
\begin{proof}
	In any optimal solution of \myref{APC2}, if a point $i$ is such that it 
	does not 
	have $\alpha$ locations at distance smaller than $d_k$, then any feasible 
	solution either has objective function value at least $d_k$ (so $u_k = 1$) 
	or $i$ is opened (so $y_i = 1$). 
\end{proof}

We observe the following for the inequalities of Theorem~\ref{thm:validInequE}.

\begin{observation}\label{obs:elConstraintsReplacement}
	For any $i$ and $d_k \in D_i$ such that 
	$|\{j \in \pts\setminus \{i\}: d_{ij} < 	d_k\}| < \alpha$, 
	the inequalities 
	\eqref{pc3:sumyu} are dominated by 
	the inequalities~\eqref{cut:ukyj}, because the former are the latter 
	multiplied by $\alpha$ with additional non-negative terms in the sum 
	on the left hand-side.
	Thus, it is not necessary to include \eqref{pc3:sumyu} for any such~$i$ and 
	$d_k$, if \eqref{cut:ukyj} is included.
\end{observation}

\subsection{Variable fixing}\label{sec:formulationE:fixing}


Next we present some conditions which allow the fixing of variables. In 
contrast to \myref{APC1}, for which we only have a condition based on an upper 
bound on the optimal objective function value, for \myref{APC2} we also have a 
condition 
which can be utilized with any known lower bound on the optimal objective 
function value of the \APCP.

\begin{theorem}\label{thm:feascutsElloumi}
	Let $LB$ \blue{be} a lower bound on the optimal objective function value 
	of the \APCP. Then 
	\begin{alignat}{3}
	u_k & = 1 \qquad && \forall k \in \{2, \dots, K\}, d_k \leq LB 
	\label{eq:feascuts_uk_LB}
	\end{alignat}
	are valid equalities for the formulation \myref{APC2} for the \APCP, i.e., 
	when adding \eqref{eq:feascuts_uk_LB} to 
	\myref{APC2}, the set of feasible solutions does not change.
\end{theorem}
\begin{proof}
	Consider any optimal solution for \myref{APC2}.
	If $LB$ is a lower bound on the optimal objective function value 
	of the \APCP, then this optimal value is at least $d_k$ for any $k$ 
	such that $d_k \leq LB$. Therefore, $u_k = 1$ in this case.
\end{proof}

In Section \ref{sec:comparison:e} we present an iterative scheme for variable 
fixing based on the optimal solution of the linear programming relaxation of 
\myref{APC2} which can be seen as extension of Theorem 
\ref{thm:feascutsElloumi}.

\begin{theorem}\label{thm:optcutsElloumi}
	Let $UB$ be the objective function value of a feasible solution of 
	the \APCP. Then when adding
		\begin{alignat}{3}
		u_k & = 0 \qquad && \forall k \in \{2, \dots, K\}, d_k > UB 
		\label{eq:optcuts_uk_UB}
		\end{alignat}
	to \myref{APC2}, no optimal solution is cut off.
\end{theorem}
\begin{proof}
	Consider any optimal solution for \myref{APC2}.
	If $UB$ is an upper bound on the optimal objective function 
	value of the \APCP, then this optimal value is at most $d_k$ for any 
	$k$ such that $d_k \geq UB$. As a consequence, this optimal value it not 
	greater or equal to any $d_k > UB$ and hence $u_k = 0$ in this case.
\end{proof}

\section{Polyhedral study \label{sec:polyhedral}}

In this section we provide a polyhedral study of our two integer 
programming formulations for the \APCP. 
We start by comparing the basic linear relaxations of the two formulations in 
Section~\ref{sec:comparison:basic}. Next, we detail how to obtain the best 
lower 
bound based on \myref{APC1}, which can be computed in polynomial time,
in Section~\ref{sec:comparison:d} and also give a combinatorial interpretation 
of this best bound. In Section~\ref{sec:comparison:e} we do the same for \myref{APC2}. Finally, we compare the two best lower bounds in 
Section~\ref{sec:comparison:comparede}.

\subsection{Comparison of basic linear relaxations}\label{sec:comparison:basic}

Whenever several integer programming formulations of a problem are available, 
it is an interesting question to compare the corresponding linear relaxations.
We note that for the  \PCP, 
\cite{ales2018} proved that the objective 
function values of the linear relaxations of 
\myref{PCE} and 
\myref{PCA} coincide. Furthermore, \cite{elloumi2004} showed that the objective 
function value of the linear relaxation of \myref{PCE} is always as least as 
good as the one of the linear relaxation of \myref{PC1}, and they demonstrated 
that the dominance 
might be strict by providing an instance where this is the case. Thus, in case 
of the  \PCP, both \myref{PCA} and \myref{PCE} dominate \myref{PC1}.

Let 
	$\mytag{APC1-R}$ be the linear relaxation of \myref{APC1}, i.e., \myref{APC1-R} 
	is
	\myref{APC1} without~\eqref{pc1:xbin} and~\eqref{pc1:ybin} and with the 
	constraints $0 \leq x_{ij}$ 
	for all $i,j \in \pts$ with $i\neq j$ and $0 \leq y_j \leq 1$ for all $j \in 
	\pts$.
	Let $\mytag{APC2-R}$ be 
	the linear relaxation of
	\myref{APC2}, i.e., \myref{APC2-R} is
	\myref{APC2} without~\eqref{pcE:ubin} and~\eqref{pcE:ybin} and with the 
	constraints $0 \leq u_k \leq 
	1$ 
	for all $k \in \{2, \dots, K\}$ and $0 \leq y_j \leq 
	1$ 
	for all $j \in \pts$.
To study \myref{APC1-R} and \myref{APC2-R}, we start by considering the 
following examples, which are illustrated in 
Figures \ref{fig:ex:daskinLRbetter} and \ref{fig:ex:elloumiLRbetter}.

\tikzstyle{vertex}=[circle,draw=black,thick,fill=black!0,minimum 
size=20pt,inner 
sep=0pt]
\tikzstyle{edge} = [bend left,draw,thick]
\tikzstyle{weight} = [font=\small]

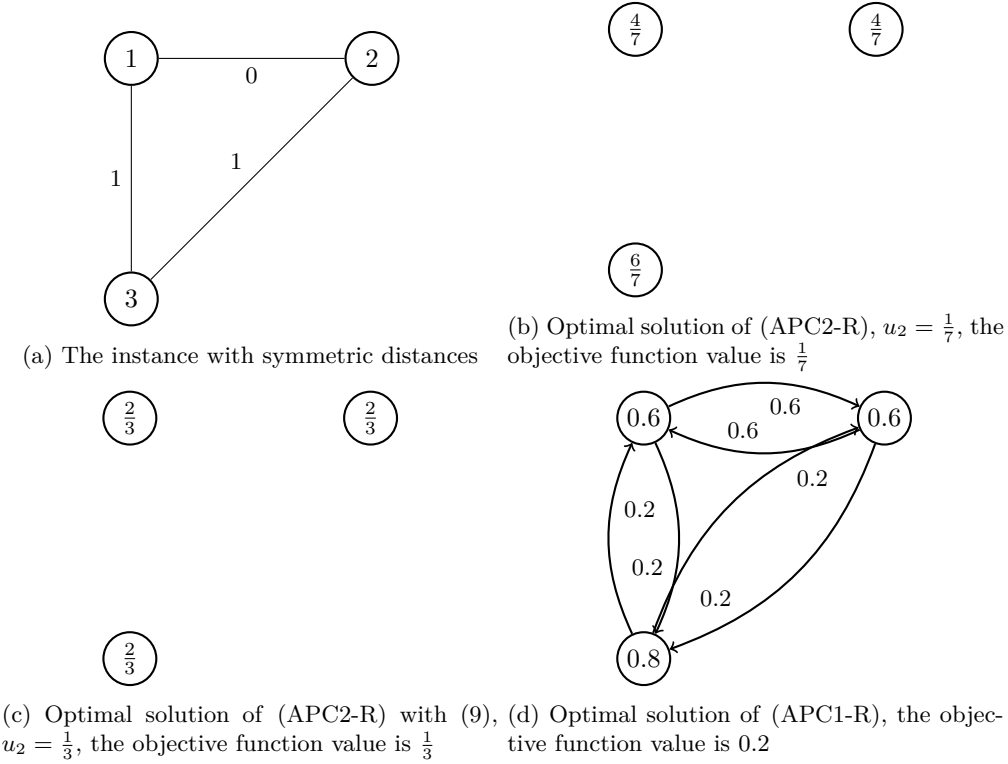
\begin{figure}[tbh]
	\begin{center}
	\begin{subfigure}{0.4\textwidth}
		\centering
		\begin{tikzpicture}[scale=1.6, auto,swap]
		\foreach \pos/\name in {{(1,1)/1}, {(1,-1)/3}, {(3,1)/2}}
		\node[vertex] (\name) at \pos {$\name$};
		\foreach \source/ \dest /\weight in {
			1/2/0,1/3/1,2/3/1}
		\path (\source) edge node[weight] 
		{$\weight$} 
		(\dest);
		\end{tikzpicture}
		\caption{The instance with symmetric 
		distances\label{ex:daskinLRbetter1}}
	\end{subfigure}
	\begin{subfigure}{0.4\textwidth}
		\centering
		\begin{tikzpicture}[scale=1.6, auto,swap]
		\foreach \pos/\name/\value in {{(1,1)/1/\frac{4}{7}}, 
		{(1,-1)/3/\elli{\frac{6}{7}}}, {(3,1)/2/\elli{\frac{4}{7}}}}
		\node[vertex] (\name) at \pos {$\value$};
		\end{tikzpicture}
		\caption{Optimal solution of \myref{APC2-R}, 
		$u_2=\frac{1}{7}$, the objective function value is \blue{$\frac{1}{7}$} 
		\label{ex:daskinLRbetter2}}
	\end{subfigure}

	\begin{subfigure}{0.4\textwidth}
		\centering
		\begin{tikzpicture}[scale=1.6, auto,swap]
		\foreach \pos/\name/\value in {{(1,1)/1/\frac{2}{3}}, 
		{(1,-1)/3/\frac{2}{3}}, {(3,1)/2/\frac{2}{3}}}
		\node[vertex] (\name) at \pos {$\value$};
		\end{tikzpicture}
		\caption{Optimal solution of \myref{APC2-R} with \eqref{cut:ukyj}, 
		$u_2=\frac{1}{3}$, the objective function value is $\frac{1}{3}$ 
		\label{ex:daskinLRbetter3}}
	\end{subfigure}
		\begin{subfigure}{0.4\textwidth}
		\centering
		\begin{tikzpicture}[scale=1.6, auto,swap]
		\foreach \pos/\name/\value in {{(1,1)/1/0.6}, 
			{(1,-1)/3/\elli{0.8}}, {(3,1)/2/\elli{0.6}}}
		\node[vertex] (\name) at \pos {$\value$};
		\foreach \source/ \dest /\weight in {
			1/2/0.6,1/3/0.2,2/3/0.2,2/1/0.6,3/1/0.2,3/2/0.2}
		\path (\source)  edge[bend left=25,draw,thick,->] node[near end, 
		weight] 
		{$\weight$} (\dest);
		\end{tikzpicture}
		\caption{Optimal solution of 
		\myref{APC1-R}, the objective function value is $0.2$ 
		\label{ex:daskinLRbetter4}}
	\end{subfigure}
	\caption{Illustration of Example 
		\ref{ex:daskinLRbetter}, \elli{in which} $p=2$ and 
		$\alpha=2$.\label{fig:ex:daskinLRbetter} The value in the nodes in 
		Figure \ref{ex:daskinLRbetter1} is the index of the node and 
		the values near the arcs are 
		the distances. The values in the nodes in 
		Figures \ref{ex:daskinLRbetter2}, \ref{ex:daskinLRbetter3} and 
		\ref{ex:daskinLRbetter4} 
		are the values of the $y$-variables in the optimal 
		solution, and the values near the arcs in Figure 
		\ref{ex:daskinLRbetter4} are the values of 
		the $x$-variables 
		in the optimal solution.}
		\end{center}
\end{figure}

\begin{examplex}\label{ex:daskinLRbetter}
	Let $\pts=\{1,2,3\}$, $p = 2$, $\alpha = 2$, 
	$d_{1,2} = 0$, 
	$d_{1,3} = d_{2,3} = 1$ 
	and $d_{ij} = d_{ji}$ for all $i,j\in 
	\pts$ with $i \neq j$.
	
	In the formulation \myref{APC2-R} we have $D = \{0,1\}$ and $D_{1} = D_{2} 
	= D_{3} = \{1\}$, so $K = 2$. An optimal solution of \myref{APC2-R} is  
	$y_1 = y_2 = \frac{4}{7}$, $y_3 = \frac{6}{7}$ and $u_2 = \frac{1}{7}$. 
	Thus, the optimal objective function value of \myref{APC2-R} is equal to 
	$\frac{1}{7}\approx 0.143$. 
	
	This solution is not feasible anymore when adding~\eqref{cut:ukyj} to 
	\myref{APC2-R}, as $u_2 + y_1 = \frac{5}{7} < 1$, which 
	\blue{violates}~\eqref{cut:ukyj} for $i=1$ and $d_k = 1$. 
	An optimal solution of \myref{APC2-R} with~\eqref{cut:ukyj} is given as 
	$y_1 = y_2 = y_3 = \frac{2}{3}$ and $u_2 = \frac{1}{3}$. Therefore, the 
	optimal objective function value of \myref{APC2-R} with~\eqref{cut:ukyj} is 
	equal to $\frac{1}{3}\approx 0.333$, which is larger than the optimal 
	objective function value of \myref{APC2-R}.

	An optimal solution for \myref{APC1-R} is $y_1 = y_2 = 0.6$, $y_3 = 0.8$, 
	$x_{1,2} = x_{2,1} = 0.6$, $x_{1,3} = x_{2,3} = x_{3,1} = x_{3,2} = 0.2$ 
	and $z=0.2$, so the optimal objective function value of \myref{APC1-R} is 
	equal to $0.2$. As a consequence, for this instance \myref{APC1-R} gives a 
	better bound for the \APCP than \myref{APC2-R}.
\end{examplex}

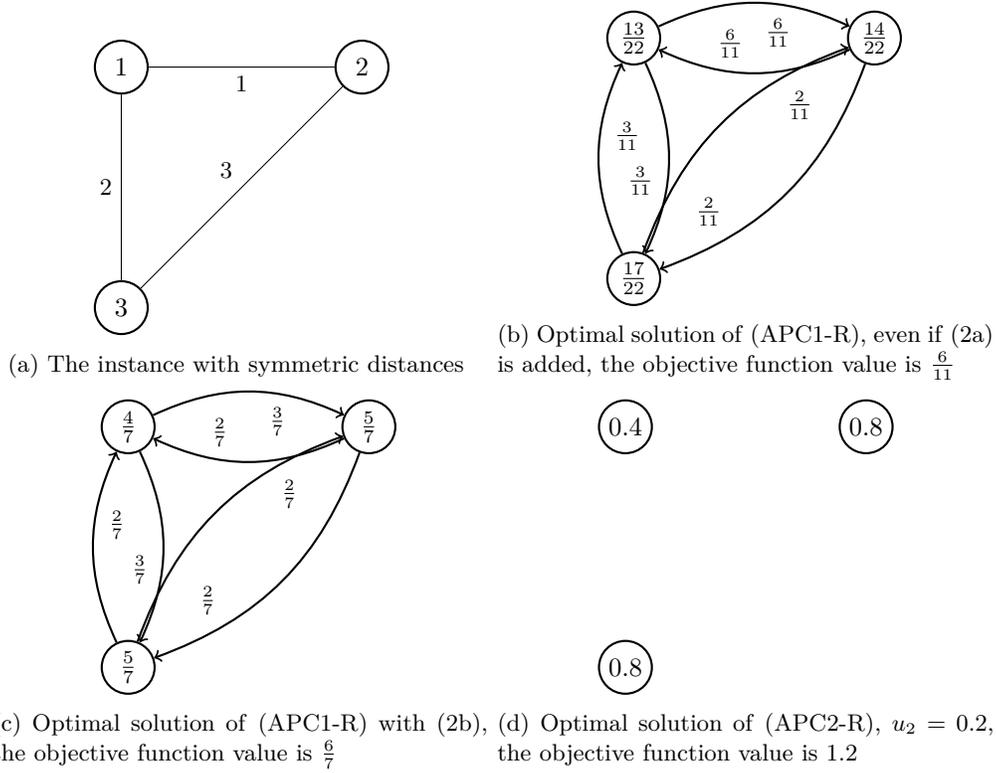
\begin{figure}[tbh]
	\begin{center}
	\begin{subfigure}{0.4\textwidth}
		\centering
		\begin{tikzpicture}[scale=1.6, auto,swap]
		\foreach \pos/\name in {{(1,1)/1}, {(1,-1)/3}, {(3,1)/2}}
		\node[vertex] (\name) at \pos {$\name$};
		\foreach \source/ \dest /\weight in {
			1/2/\elli{1},1/3/\elli{2},2/3/\elli{3}}
		\path (\source) edge node[weight] 
		{$\weight$} 
		(\dest);
		\end{tikzpicture}
		\caption{The instance with symmetric distances 
		\label{ex:elloumiLRbetter1}}
	\end{subfigure}
	\begin{subfigure}{0.4\textwidth}
	\centering
	\begin{tikzpicture}[scale=1.6, auto,swap]
	\foreach \pos/\name/\value in {{(1,1)/1/\frac{13}{22}}, 
		{(1,-1)/3/\blue{\frac{17}{22}}}, {(3,1)/2/\blue{\frac{14}{22}}}}
	\node[vertex] (\name) at \pos {$\value$};
	\foreach \source/ \dest /\weight in {
		2/1/\frac{6}{11},3/1/\frac{3}{11},3/2/\frac{2}{11},1/2/\frac{6}{11},1/3/\frac{3}{11},2/3/\frac{2}{11}}
	\path (\source)  edge[bend left=25,draw,thick,->] node[near end, 
	weight] 
	{$\weight$} (\dest);
	\end{tikzpicture}
	\caption{Optimal solution of 
		\myref{APC1-R}, even if \eqref{cut:sumdijxij_alphaz} is 
		added, the objective function value is $\frac{6}{11}$ 
		\label{ex:elloumiLRbetter2}}
\end{subfigure}

	\begin{subfigure}{0.4\textwidth}
		\centering
		\begin{tikzpicture}[scale=1.6, auto,swap]
		\foreach \pos/\name/\value in {{(1,1)/1/\frac{4}{7}}, 
		{(1,-1)/3/\frac{5}{7}}, {(3,1)/2/\frac{5}{7}}}
	\node[vertex] (\name) at \pos {$\value$};
	\foreach \source/ \dest /\weight in {
		1/2/\frac{3}{7},1/3/\frac{3}{7},2/3/\frac{2}{7},2/1/\frac{2}{7},3/1/\frac{2}{7},3/2/\frac{2}{7}}
	\path (\source)  edge[bend left=25,draw,thick,->] node[near end, 
	weight] 
	{$\weight$} (\dest);
		
		\end{tikzpicture}
		\caption{Optimal solution of \myref{APC1-R} with~\eqref{cut:yixij}, the 
		objective function value is $\frac{6}{7}$  
		\label{ex:elloumiLRbetter3}}
	\end{subfigure}
	\begin{subfigure}{0.4\textwidth}
		\centering
		\begin{tikzpicture}[scale=1.6, auto,swap]
		\foreach \pos/\name/\value in {{(1,1)/1/0.4}, 
			{(1,-1)/3/0.8}, {(3,1)/2/0.8}}
		\node[vertex] (\name) at \pos {$\value$};
		\end{tikzpicture}
			\caption{Optimal solution of \myref{APC2-R}, $u_2=0.2$, the 
			objective function value is $1.2$  
			\label{ex:elloumiLRbetter4}}
	\end{subfigure}
	\caption{Illustration of Example 
		\ref{ex:elloumiLRbetter}, \elli{in which} $p=2$ and 
		$\alpha=2$.\label{fig:ex:elloumiLRbetter} The value in the nodes in 
		Figure \ref{ex:elloumiLRbetter1} is the index of the node and 
		the values near the arcs are 
		the distances. The values in the nodes in 
		Figures \ref{ex:elloumiLRbetter2}, \ref{ex:elloumiLRbetter3} and 
		\ref{ex:elloumiLRbetter4} 
		are the values of the $y$-variables in the optimal 
		solution, and the values near the arcs in Figure 
		\ref{ex:elloumiLRbetter2} 
		and
		\ref{ex:elloumiLRbetter3} are the values of 
		the $x$-variables 
		in the optimal solution.}
			\end{center}
\end{figure}

\begin{examplex}\label{ex:elloumiLRbetter}
	Let $\pts=\{1,2,3\}$, $p = 2$, $\alpha = 2$, 
	$d_{1,2} = 1$, 
	$d_{1,3} = 2$, $d_{2,3} = 3$ 
	and $d_{ij} = d_{ji}$ for all $i,j\in 
	\pts$ with $i \neq j$.
	
	For \myref{APC1-R} an optimal solution is given as $x_{12} = x_{21} = 
	\frac{6}{11}$, $x_{13} = x_{31} = \frac{3}{11}$, $x_{23} = x_{32} = 
	\frac{2}{11}$, $y_1 = \frac{13}{22}$, $y_2 = \frac{14}{22}$, $y_3 = 
	\frac{17}{22}$ and the optimal objective function value of  \myref{APC1-R} 
	is $z = \frac{6}{11} \approx 0.545$. 
	This solution remains feasible when 
	adding~\eqref{cut:sumdijxij_alphaz}, so also for \myref{APC1-R} 
	with~\eqref{cut:sumdijxij_alphaz} the optimal 
	objective function value is $z = \frac{6}{11} \approx 0.545$.
	
	However, this solution is not feasible anymore when 
	adding~\eqref{cut:yixij} to 
	\myref{APC1-R}, as $y_1 + x_{12} = \frac{25}{22} > 1$, which 
	\elli{violates}~\eqref{cut:yixij} for $i=1$ and $j = 2$. 
	An optimal solution of \myref{APC2-R} with~\eqref{cut:yixij} is given as
	$x_{12} = x_{13} = \frac{3}{7}$, 
	$x_{21} = x_{23} = x_{31} = x_{32} = \frac{2}{7}$, 
	$y_1 = \frac{4}{7}$, $y_2 = y_3 = \frac{5}{7}$, 
	and $z = \frac{6}{7} \approx 0.857$. 
	This solution is feasible also when adding~\eqref{cut:assignedToLaterSet}.
	Therefore, the 
	optimal objective function value of \myref{APC1-R} 
	with~\eqref{cut:sumdijxij_alphaz} and~\eqref{cut:yixij}, and also of 
	\myref{APC1-R} 
	with~\eqref{cut:sumdijxij_alphaz},~\eqref{cut:yixij} and 
	\eqref{cut:assignedToLaterSet}
	is equal to $\frac{6}{7}\approx 0.857$.
	
	In the formulation \myref{APC2-R} we have $D = \{1,2,3\}$, $D_{1} = 
	\{2\}$, $D_{2} = \{3\}$,  
	$D_{3} = \{2,3\}$ and $K = 3$. An optimal solution of \myref{APC2-R} is  
	$y_1 = 0.4$, $y_2 = y_3 = 0.8$, $u_2 = 0.2$ and $u_3 = 0$. 
	Thus, the optimal objective function value of \myref{APC2-R} is equal to 
	$1.2$, which is larger than the optimal 
	objective function value of \myref{APC1-R}, even when adding the  
	inequalities~\eqref{cut:sumdijxij_alphaz},~\eqref{cut:yixij} and 
	\eqref{cut:assignedToLaterSet} to \myref{APC1-R}.
	As a consequence, for this instance \myref{APC2-R} gives a 
	better bound for the \APCP than \myref{APC1-R}.
\end{examplex}

Example~\ref{ex:daskinLRbetter} shows that the 
linear relaxation of \myref{APC1} might give strictly better bounds than the 
linear 
relaxation of \myref{APC2}.
In turn, Example~\ref{ex:elloumiLRbetter} shows that the 
linear relaxation of \myref{APC2} might give strictly better bounds than the 
linear 
relaxation of \myref{APC1}.
Thus, in the case of the \APCP the linear relaxations of \myref{APC1} and 
\myref{APC2} are not comparable.

Furthermore,
Example~\ref{ex:daskinLRbetter} also demonstrates the existence of an instance 
of the \APCP, where including~\eqref{cut:ukyj} into \myref{APC2-R} 
yields a strictly better bound than the one of \myref{APC2-R}.
Moreover, 
Example~\ref{ex:elloumiLRbetter} also shows that there exists an instance 
where adding~\eqref{cut:yixij} to \myref{APC1-R} improves the linear relaxation 
bound. 

\subsection{Best lower bound based on (APC1)}
\label{sec:comparison:d}

The aim of this section is to derive the best possible bound for \myref{APC1} 
when utilizing all inequalities derived in Section~\ref{sec:formulationD:cuts}.
To do so, we investigate
Theorem~\ref{thm:liftingValid} in more detail. In particular, it allows us to 
add new valid inequalities to the 
linear relaxation of \myref{APC1}, as soon as we have a lower bound $LB$.
Our  
hope is that including the new valid inequalities for the lower bound $LB$ will 
give us a new, even better lower bound, with which we can include new, stronger 
valid 
inequalities. This leads to an iterative approach to improve the lower bound on 
the optimal objective function value of the \APCP, 
which is analogous to the approach   
 \cite{gaar2022scaleable} have developed for the  \PCP. They 
proved that their approach for the  \PCP converges (i.e., including the 
valid 
inequalities for a given lower bound $LB$ does not give a better lower bound, 
but only $LB$ again) if and only if there is a fractional set cover solution 
with 
radius $LB$ that uses at most $p$ sets. 

In the following, we investigate a similar iterative approach for the \APCP by  
iteratively adding the inequalities from 
Theorem~\ref{thm:liftingValid}. 
Let $LB$ \elli{be} a lower bound on the optimal objective function 
value 
of the \APCP
and let 
\begin{subequations}
	\begin{alignat}{3}
	\mytag{APCLB} \qquad
	\mathcal{L}_\alpha(LB) = & \min & \obj \phantom{iiiii}  \\ 
	& \st~ &  \eqref{pc1:sumy}, &\eqref{pc1:sumx}, \eqref{pc1:xy} 
	\\
	&&\eqref{cut:yixij}, & 
	\eqref{pc1:dxlifted},
	\eqref{cut:sumdijxij_alphazlifted}\\
	&& 0 &\leq x_{ij} \qquad&& \forall i , j \in \pts, i 
	\neq j  	\label{apclb:x}
	\\
	&& 0 & \leq y_{j} \leq 1 && \forall j \in \pts \label{apclb:y}	
	\\
	&& z & \geq LB. \label{apclb:z}
	\end{alignat}
\end{subequations}
It follows from Theorem~\ref{thm:validInequDaskin},~\ref{thm:liftingValid} 
and~\ref{thm:optCutDaskin} 
that $\mathcal{L}_\alpha(LB)$ is again a lower bound on the optimal objective 
function value 
of the \APCP. Furthermore, it is easy to see that 
$\mathcal{L}_\alpha(LB) \geq LB$ holds.
We now want to establish a condition for the case that adding inequalities 
\blue{from Theorem \ref{thm:liftingValid}} for 
a lower bound $LB$ \blue{to \myref{APCLB}} does not improve the obtained bound 
$\mathcal{L}_\alpha(LB)$ 
anymore, i.e., a convergence-condition. It turns out that the following holds.
\begin{theorem}
	\label{thm:convergenceDaskin}
	Let $LB$ \elli{be} a lower bound on the optimal objective function value 
	of the \APCP. 
	
	Then $\mathcal{L}_\alpha(LB) = LB$ holds if and only if 
	there 
	is a feasible solution for 
	\begin{subequations} \label{eq:setcoverDaskin}
		\begin{alignat}{3}
		&\min &  \sum_{j \in \pts} y_j \phantom{iiiii}&\label{eq:afscoD} 
		\\       
		&\st~&\sum_{j \in \pts \setminus \{i\}:  d_{ij}  \leq LB}y_{j} 
		&\geq 
		\alpha(1-y_i) \qquad &&
		\forall i \in \pts \label{eq:afsc1aD} \\
		&&\sum_{j \in \pts \setminus (\pts_\beta \cup \{i\}):  d_{ij}  \leq 
			LB}y_{j} 
		&\geq 
		(\alpha-\beta)(1-y_i) \qquad &&
		\forall i \in \pts, \forall \beta \in \{1, \dots, \alpha\}, \forall 
		\pts_\beta 
		\subseteq 
		\pts\setminus \{i\}, |\pts_\beta| = \beta
		\label{eq:afsc1bD} \\
		&& 0 &\leq y_{j} \leq 1 \qquad && \forall j \in \pts \label{eq:afsc2D}
		\end{alignat}
	\end{subequations}
	with objective function value at most $p$.
\end{theorem}
\begin{proof}
	We prove each of the two sides of the equivalence in a separate part for 
	the sake of clarity.
	
	\textbf{Part 1:}
	Assume $LB$ is such that $\mathcal{L}_\alpha(LB) = LB$ holds. Let 
	$(x^*,y^*,z^*)$ be an optimal solution of \myref{APCLB} in this 
	case, so $\mathcal{L}_\alpha(LB) = z^* = LB$. We will finish this part of 
	the proof by showing that 
	$y^*$ is a 
	feasible solution for \eqref{eq:setcoverDaskin} with objective function 
	value at most $p$.
	
	It is easy to see that \eqref{eq:afsc2D} is satisfied because 
	of~\eqref{apclb:y} and that the objective function value~\eqref{eq:afscoD}  
	of $y^*$ is $p$ because of \eqref{pc1:sumy}.
	
	In order to show that $y^*$ fulfills~\eqref{eq:afsc1aD}, we can exploit 
	\eqref{cut:sumdijxij_alphazlifted} and get that
	\begin{alignat*}{3}
	\alpha LB y^*_i +\sum_{j \in \pts \setminus \{i\}} \max\{LB,d_{ij}\} 
	x^*_{ij} & 
	\leq  \alpha LB \qquad &&  \forall i \in \pts
	\end{alignat*}
	holds, which, when splitting the $x^*_{ij}$ according to their distance 
	$d_{ij}$, is equivalent to 
	\begin{alignat*}{3}
	\alpha LB y^*_i 
	+LB\sum_{j \in \pts \setminus \{i\}:  d_{ij}  \leq LB} x^*_{ij}
	+\sum_{j \in \pts \setminus \{i\}:  d_{ij}  > LB} d_{ij} x^*_{ij}
	 & 
	\leq  \alpha LB \qquad &&  \forall i \in \pts.
	\end{alignat*}
	Now we can use~\eqref{pc1:sumx} for the first sum with $x^*_{ij}$ and obtain
	\begin{alignat*}{3}
	\alpha LB y^*_i 
	+LB
	\left( \alpha(1-y^*_i) -  \sum_{j \in \pts \setminus \{i\}:  d_{ij}  > LB} 
	 x^*_{ij} \right)
	+\sum_{j \in \pts \setminus \{i\}:  d_{ij}  > LB} d_{ij} x^*_{ij}
	& 
	\leq  \alpha LB \qquad &&  \forall i \in \pts,
	\end{alignat*}
	which can be simplified to
		\begin{alignat*}{3}
	\sum_{j \in \pts \setminus \{i\}:  d_{ij}  > LB} (d_{ij} - LB) x^*_{ij}
	& 
	\leq  0 \qquad &&  \forall i \in \pts.
	\end{alignat*}
	On the left hand-side this is a sum of non-negative terms, because  
	$x^*_{ij} \geq 0$ due to~\eqref{apclb:x} and for each term in the sum 
	$(d_{ij} - LB) > 0$ holds. Thus, the only way that this can be satisfied 
	is 
	that $x^*_{ij} = 0$ for all $j\in \pts \setminus \{i\}$ such that $d_{ij} > 
	LB$. 
	This, together with~\eqref{pc1:sumx} and~\eqref{pc1:xy} implies that
	\begin{alignat}{3}\label{eq:proofD1}
	\alpha(1- y^*_i) 
	= \sum_{j \in \pts \setminus \{i\}} x^*_{ij}
	= \sum_{j \in \pts \setminus \{i\}:  d_{ij}  \leq LB} x^*_{ij}
	&\leq \sum_{j \in \pts \setminus \{i\}:  d_{ij}  \leq LB} y^*_{j}
	\qquad &&  \forall i \in \pts,
	\end{alignat}
	so $y^*$ fulfills~\eqref{eq:afsc1aD}.
	
	What is left to show is that $y^*$ satisfies~\eqref{eq:afsc1bD}. Towards 
	this end, we can use~\eqref{pc1:xy}, \eqref{eq:proofD1} 
	and~\eqref{cut:yixij} to obtain
	\begin{alignat*}{3}
	\sum_{j \in \pts \setminus (\pts_\beta \cup \{i\}):  d_{ij}  \leq 
		LB}y^*_{j} 
	&\geq 
	\sum_{j \in \pts \setminus (\pts_\beta \cup \{i\}):  d_{ij}  \leq 
		LB}x^*_{ij}
    \geq 	
    \sum_{j \in \pts \setminus \{i\}:  d_{ij}  \leq 
    	LB}x^*_{ij}
    -
        \sum_{j \in \pts_\beta }x^*_{ij}\\
    &\geq
    \alpha(1- y^*_i)
    - |\pts_\beta|(1-y^*_i)\\
    &= (\alpha - \beta)(1-y^*_i) \qquad \qquad \qquad 
	\forall i \in \pts, \forall \beta \in \{1, \dots, \alpha\}, \forall 
	\pts_\beta 
	\subseteq 
	\pts\setminus \{i\}, |\pts_\beta| = \beta,
	\end{alignat*}	
	so $y^*$ fulfills~\eqref{eq:afsc1bD}. Thus, $y^*$ is a 
	feasible solution for \eqref{eq:setcoverDaskin} with objective function 
	value at most $p$.
	
	\textbf{Part 2:}
	Assume $LB$ is such that there is a feasible solution $\yNewOpt$ 
	for~\eqref{eq:setcoverDaskin} with objective function value \blue{of} at 
	most $p$. We 
	will finish this part of the proof in four steps. 
	In the first step we utilize $\yNewOpt$ to construct $y^*$, which is 
	feasible 
	for~\eqref{eq:setcoverDaskin} and has \blue{an} objective function value 
	$p$. 
	In the second step we use $y^*$ to construct $\yNewOpti$ for 
	each $i \in \pts$ and show that $\yNewOpti$ has a particular property. 
	In the third step we use $\yNewOpti$ to construct $x^*$.
	In the fourth step we show that  
	$(x^*,y^*,z^*)$ with $z^* = LB$ is a feasible solution for 
	\myref{APCLB}, which implies 
	that $\mathcal{L}_\alpha(LB) = LB$ holds.
	
	We start with the first step by constructing $y^*$. Let $\pNewOpt$ be the 
	objective function value~\eqref{eq:afscoD} of $\yNewOpt$, so $\pNewOpt = 
	\sum_{j \in \pts} \yNewOpt_j$. It follows that $\pNewOpt \leq p$, as 
	$\yNewOpt$ has objective function value 
	at most $p$.
	We now construct $y^*$ as $y^*_j = \yNewOpt_j + 
	(1-\yNewOpt_j)\frac{p-\pNewOpt 
	}{|\pts| - \pNewOpt}$ for all $j \in \pts$. 
	We have that $0 \leq \frac{p-\pNewOpt}{|\pts| - \pNewOpt}< 1$ because 
	$\pNewOpt \leq p$ and  $p < |\pts|$. As a consequence, 
	 $y^*_j$ fulfills $0 
	\leq \yNewOpt_j \leq y^*_j \leq 1$ for all $j \in \pts$ because $\yNewOpt$ 
 	fulfills~\eqref{eq:afsc2D}. 
 	Furthermore, it holds that
 	\begin{align*}
 	\sum_{j \in \pts} y^*_j 
 	=\sum_{j \in \pts} \left( \yNewOpt_j + 
 	(1-\yNewOpt_j)\frac{p-\pNewOpt 
 	}{|\pts| - \pNewOpt} \right)
 	= \sum_{j \in \pts}  \yNewOpt_j
 	+ \frac{p-\pNewOpt 
 	}{|\pts| - \pNewOpt} \sum_{j \in \pts} (1-\yNewOpt_j)
 	= \pNewOpt + \frac{p-\pNewOpt 
 	}{|\pts| - \pNewOpt} (|\pts| - \pNewOpt)
 	= p.
 	\end{align*} 
 	Thus, $y^*$ is feasible 
	for~\eqref{eq:setcoverDaskin} and has objective function value $p$.
	
	We proceed with the second step by constructing $\yNewOpti$. 
	For all $i \in \pts$ we define $\yNewOpti$ in such a way that 
	$\yNewOpti_j = 
	\min\{y^*_j,1-y^*_i\}$ for all $j \in \pts$, so in particular $\yNewOpti$ 
	is the component-wise minimum of $y^*$ and $(1-y^*_i)$ and $\yNewOpti \leq 
	y^*$ holds. We will now show 
	that $\yNewOpti$ fulfills
		\begin{alignat}{3}\label{eq:propertyy}
\sum_{j \in \pts \setminus \{i\}:  d_{ij}  \leq LB} \yNewOpti_j
& 
\geq  \alpha(1-y^*_i) \qquad &&  \forall i \in \pts.
\end{alignat}
To do so, let $N\newOpt_i = \{j \in \pts\setminus\{i\}:  d_{ij}  \leq LB \text{ 
and } \yNewOpti_j < y^*_j  \}$, so $N\newOpt_i$ is the set of indices $j$ that 
appear in the sum on the left hand-side of~\eqref{eq:propertyy} and fulfill 
$y^*_j > 1-y^*_i = \yNewOpti_j$. 

If $|N\newOpt_i| = 0$, then $\yNewOpti_j = y^*_j$ for each term in the sum 
of~\eqref{eq:propertyy} and thus~\eqref{eq:propertyy} is satisfied because 
$y^*$ fulfills~\eqref{eq:afsc1aD}.

If $|N\newOpt_i| \geq \alpha$, then $\yNewOpti_j = 1-y^*_i$ for all $j \in  
N\newOpt_i$ implies that
\begin{alignat*}{3}
\sum_{j \in \pts \setminus \{i\}:  d_{ij}  \leq LB} \yNewOpti_j
&\geq 
\sum_{j \in N\newOpt_i} \yNewOpti_j
= \sum_{j \in N\newOpt_i} (1-y^*_i)
= |N\newOpt_i|(1-y^*_i)
\geq \alpha(1-y^*_i),
\end{alignat*}
so also in this case~\eqref{eq:propertyy} is fulfilled.

If $0 < |N\newOpt_i| < \alpha$, then\blue{~\eqref{eq:afsc1bD}} for $\beta = 
|N\newOpt_i|$ and $\pts_\beta = N\newOpt_i$ together with the fact that 
$\yNewOpti_j = 1-y^*_i$ for all $j \in \pts_\beta = N\newOpt_i$ shows that
\begin{alignat*}{3}
\sum_{j \in \pts \setminus \{i\}:  d_{ij}  \leq LB} \yNewOpti_j
&= 
\sum_{j \in \pts \setminus (\pts_\beta \cup \{i\}):  d_{ij}  \leq LB} 
\yNewOpti_j
+ 
\sum_{j \in \pts_\beta} \yNewOpti_j\\
& 
\geq  (\alpha-\beta)(1-y^*_i) +  \beta(1-y^*_i)
=  \alpha(1-y^*_i),
\end{alignat*}
so also in this case~\eqref{eq:propertyy} is fulfilled.
As a consequence, $\yNewOpti$ satisfies~\eqref{eq:propertyy} in all cases, so 
for all $i \in \pts$.  

	We continue with the third step, i.e., we now construct $x^*$.
	To do so, we first fix a point $i\in \pts$. 
	Then let
	$j_i$ be such that 
	$\sum_{j \in S_{ij_{i}}\setminus \{i\}} \yNewOpti_{j} < \alpha(1-y^*_i)$ 
	and such that
	$\sum_{j \in (S_{ij_{i}} \cup \{j_{i}\})\setminus \{i\}} \yNewOpti_{j} 
	\geq 
	\alpha(1-y^*_i)$. Clearly such a $j_{i}$ exists and $d_{ij_i} \leq LB$ 
	because of~\eqref{eq:propertyy}. Then we set 
	$x^*_{ij} = \yNewOpti_{j}$ if $j \in  S_{ij_{i}}\setminus \{i\}$, we 
	set 
	$x^*_{ij} = \alpha(1-y^*_i) - \sum_{j' \in  S_{ij_{i}}\setminus \{i\}} 
	\yNewOpti_{j'}$ if 
	$j = j_i$ and we set 
	$x^*_{ij} = 0$ otherwise.
	Note that this construction implies that $x^*_{ij} = 0$ for all $j$ such 
	that $d_{ij} > LB$. 
	
	Finally, we are able to do the fourth step, i.e., we show that 
	$(x^*,y^*,z^*)$ with $z^* = LB$ is feasible for \myref{APCLB}.
	By construction, $x^*_{ij} \geq 0$, $x^*_{ij} \leq \yNewOpti_{j} \leq 
	y^*_j$ and $x^*_{ij} \leq \yNewOpti_{j} \leq 
	1-y^*_i$ for all $i,j \in \pts$ with $j\neq i$, 
	so $(x^*,y^*,z^*)$ 
	fulfills~\eqref{apclb:x},~\eqref{pc1:xy} and~\eqref{cut:yixij}.
	Also $\sum_{j \in \pts \setminus \{i\}} x^*_{ij} = \alpha(1-y^*_i)$ by 
	construction, so~\eqref{pc1:sumx} holds.
	Moreover, by construction $y^*$ is a feasible solution 
	of~\eqref{eq:setcoverDaskin} and has objective function value $p$, so it 
	fulfills~\eqref{apclb:y} and~\eqref{pc1:sumy}. 
	Furthermore $z^* = LB$, so clearly~\eqref{apclb:z} is satisfied.
	
	The inequality~\eqref{pc1:dxlifted} is fulfilled if $d_{ij} > LB$, because 
	then $x^*_{ij} = 0$ and thus $LBy^*_i \leq LB = z^*$ is satisfied as 
	we have already shown that~\eqref{apclb:y} holds. If $d_{ij} \leq LB$, then 
	the inequality is $LB(y^*_i + 
	x^*_{ij}) \leq LB = z^*$, which is fulfilled because we already know 
	that~\eqref{cut:yixij} is satisfied. Thus, in any case $(x^*,y^*,z^*)$ 
	fulfills~\eqref{pc1:dxlifted}.
	
	Finally, we consider~\eqref{cut:sumdijxij_alphazlifted}. We can utilize 
	$x^*_{ij} = 0$ whenever $d_{ij}> LB$ and the already 
	shown~\eqref{pc1:sumx}  to obtain 
	\begin{align*}
\alpha LB y^*_i +\sum_{j \in \pts \setminus \{i\}} \max\{LB,d_{ij}\} x^*_{ij} 
	&= 
\alpha LB y^*_i +LB\sum_{j \in \pts \setminus \{i\}: d_{ij} \leq LB} x^*_{ij}	
=
\alpha LB y^*_i +LB\sum_{j \in \pts \setminus \{i\}} x^*_{ij} \\
&= 
\alpha LB y^*_i + LB\alpha(1-y^*_i) = LB\alpha = \alpha z^*,
	\end{align*}
	so~\eqref{cut:sumdijxij_alphazlifted} holds for  $(x^*,y^*,z^*)$.
	Therefore, 
	$(x^*,y^*,z^*)$ with $z^* = LB$ is a feasible solution for 
	\myref{APCLB}, which implies 
	that $\mathcal{L}_\alpha(LB) = LB$ holds.
\end{proof}

When comparing Theorem~\ref{thm:convergenceDaskin} to the corresponding 
result for the  \PCP, it becomes obvious that \eqref{eq:setcoverDaskin} 
is closely related to a fractional set 
cover problem, where every set has to be covered $\alpha$ times.

\blue{We note that 
the right 
hand-side of~\eqref{eq:afsc1aD} is $\alpha(1-y_i)$, instead of 
$\alpha$, which 
would be the generalization of the result of \cite{gaar2022scaleable}. This 
is caused by 
the fact that $i$ does not need to be covered if it is opened in the 
\APCP, while in the \PCP each point needs to be covered}.
Moreover, the inequalities~\eqref{eq:afsc1bD} are completely new. They make 
sure 
that a set cover property is fulfilled not only for all points at most $LB$ 
away, but also for subsets of these points when removing at most $\alpha$ 
points. Note that~\eqref{eq:afsc1aD} can be interpreted as~\eqref{eq:afsc1bD} 
for $\beta = 0$.

Interestingly, we can pin point which of the inequalities of  
\myref{APCLB} are responsible for the existence 
of~\eqref{eq:afsc1bD}. 
To do so, let 
\begin{subequations} 
	\begin{alignat}{3}
	\mytag{AFSC}_\delta \qquad  &\min &  \sum_{j \in \pts} 
	y_j 
	\phantom{iiiii}&\label{eq:afsco} \\       
	&\st~&\sum_{j \in \pts \setminus \{i\}:  d_{ij}  \leq \delta}y_{j} 
	&\geq 
	\alpha(1-y_i) \qquad &&
	\forall i \in \pts \label{eq:afsc1a} \\
	&& 0 &\leq y_{j} \leq 1 \qquad && \forall j \in \pts \label{eq:afsc2}
	\end{alignat}
\end{subequations}
denote the fraction set cover problem for the \APCP for a given $\delta \in 
\R$. Note that $\myref{AFSC}_{LB}$  is a relaxation 
of~\eqref{eq:setcoverDaskin}.
Furthermore, let 
$LB$ \elli{be} a lower bound on the optimal objective function 
value 
of the \APCP
and let 
\begin{alignat*}{3}
\mytag{APCLB'} \qquad
\mathcal{L}'_\alpha(LB) = & \min && \obj \phantom{iiiii}  \\ 
& \st~ &&  \eqref{pc1:sumy}, \eqref{pc1:sumx}, \eqref{pc1:xy} 
\\
&&& 
\eqref{pc1:dxliftedRelax},
\eqref{cut:sumdijxij_alphazlifted}\\
&&&\eqref{apclb:x}, 
\eqref{apclb:y}, \eqref{apclb:z}.
\end{alignat*}
Note that when in \myref{APCLB} the constraint~\eqref{pc1:dxlifted} is relaxed 
to \eqref{pc1:dxliftedRelax} and \eqref{cut:yixij} is removed, then one obtains 
\myref{APCLB'}, so \myref{APCLB'} is a relaxation of \myref{APCLB}.
We are also able to give an interpretation of when the new lower bound  
$\mathcal{L}'_\alpha(LB)$ does not improve the previous lower bound $LB$ 
in the 
following theorem.
\begin{theorem}
	\label{thm:convergenceDaskinRelaxation}
	Let $LB$ \elli{be} a lower bound on the optimal objective function value 
	of the \APCP. 
	
	Then $\mathcal{L}'_\alpha(LB) = LB$ holds if and only if 
	there 
	is a feasible solution for $\myref{AFSC}_{LB}$ 
	with objective function value at most $p$.
\end{theorem}
\begin{proof}
	The proof of Theorem~\ref{thm:convergenceDaskinRelaxation} is a 
	straight-forward simplified 
	version of the proof of Theorem~\ref{thm:convergenceDaskin}, where 
	the construction of $\yNewOpti$ in the second part is replaced by using 
	$\yNewOpti = y^*$ for all $i \in \pts$.
	Thus, we omit 
	the proof for the sake of brevity.
\end{proof}

If we combine the knowledge of Theorem~\ref{thm:convergenceDaskin} 
and~\ref{thm:convergenceDaskinRelaxation}, then we can deduce that the 
inequalities \eqref{cut:yixij} and \eqref{pc1:dxlifted} \elli{in \myref{APCLB}} 
(instead of the weaker 
version \eqref{pc1:dxliftedRelax} \elli{ in \myref{APCLB'}}) are responsible 
for the existence 
of~\eqref{eq:afsc1bD} \elli{in~\eqref{eq:setcoverDaskin}}.  
\elli{Thus, the inequalities~\eqref{eq:afsc1bD}, which are not present in a 
straight-forward generalization of the results of \cite{gaar2022scaleable} for 
the \PCP to the \APCP, are caused by the inequalities \eqref{cut:yixij} and 
\eqref{pc1:dxlifted}.}

Furthermore, with the help of Theorem~\ref{thm:convergenceDaskin}  
and~\ref{thm:convergenceDaskinRelaxation} it is easy to see that whenever 
$\mathcal{L}_\alpha(LB) = LB$ and $\mathcal{L}'_\alpha(LB) = LB$ holds for some 
lower bound $LB$, then also $\mathcal{L}_\alpha(LB') = LB'$ and 
$\mathcal{L}'_\alpha(LB') = LB'$ holds for any $LB' > LB$, \elli{that is, if 
the lower bound $LB$ can not be improved by adding the valid inequalities from 
Theorem \ref{thm:liftingValid}, then also no larger lower bound can be improved 
this way.} Thus, it makes sense 
to define the largest possible lower bounds one can obtain with iteratively 
adding the valid inequalities \elli{from Theorem \ref{thm:liftingValid}}.
Let $LB_\alpha^{\#} = \min \{LB \in \mathbb{R}: \mathcal{L}_\alpha(LB) = 
LB\}$ and 
let ${LB_\alpha^{\#}}' = \min \{LB \in \mathbb{R}: 
\mathcal{L}'_\alpha(LB) 
= LB\}$. Our results imply the following relationship.
\begin{corollary}\label{cor:boundComparisonDaskin}
	It holds that $LB_\alpha^{\#} \geq {LB_\alpha^{\#}}'$. 
\end{corollary}
\begin{proof}
	This is a consequence of Theorem~\ref{thm:convergenceDaskin}  
	and~\ref{thm:convergenceDaskinRelaxation}.
\end{proof}

\elli{Next}, we point out that both $LB_\alpha^{\#}$ and ${LB_\alpha^{\#}}'$ 
can be 
computed efficiently.
\begin{theorem}
	 $LB_\alpha^{\#}$ and ${LB_\alpha^{\#}}'$ can be 
	computed in polynomial time.
\end{theorem}
\begin{proof}
A trivial lower bound $LB$ on the optimal objective function value of the \APCP 
is given by $d_1$, the smallest element of $D$. 
For any given lower bound $LB$, the 
computation of $\mathcal{L}_\alpha(LB)$ requires to solve a linear program with 
a polynomial number of variables and constraints, and thus can be done in 
polynomial time.
Furthermore, there are only a 
polynomial number of potential values for $LB_\alpha^{\#}$, as clearly 
$LB_\alpha^{\#} \in D$ holds, because only for values in $D$ the included 
variables in 
the sum in the left hand-side of~\eqref{eq:afsc1aD} and~\eqref{eq:afsc1bD} 
change. Thus, whenever we have obtained some new lower bound $LB_\alpha^{\#}$, 
we we know that also $\min_{d_k \in D}\{d_k \geq LB_\alpha^{\#}\}$ is a lower 
bound. Therefore, it is possible to compute $LB_\alpha^{\#}$ in polynomial time.

By same arguments also ${LB_\alpha^{\#}}' \in D$ and  
${LB_\alpha^{\#}}'$ can be computed in polynomial time.
\end{proof}

\elli{Thus, not only for the  \PCP, but also for the \APCP the iterative 
improvement 
of the lower bound leads to an ultimate lower bound $LB_\alpha^{\#}$, 
which can be computed in polynomial time}. 

Finally, we want to discuss another interesting aspect about $LB_\alpha^{\#}$ 
and 
${LB_\alpha^{\#}}'$. We have seen in Example~\ref{ex:DaskinBasic} that 
adding the optimality-preserving inequalities~\eqref{cut:assignedToLaterSet} to 
a relaxed version of \myref{APC1} improved the bound obtained from the 
relaxation. Thus, it is a natural question if the bounds $LB_\alpha^{\#}$ 
and 
${LB_\alpha^{\#}}'$ could be further improved by 
adding~\eqref{cut:assignedToLaterSet} 
to \myref{APCLB} and \myref{APCLB'}, respectively. It turns our that this is 
not the case.

\begin{theorem}\label{thm:convergenceWithOptimalityInequ}
	Let $LB$ \blue{be} a lower bound on the optimal objective function value 
	of the \APCP and let $UB$ be the objective function value of a feasible 
	solution of the \APCP.
	Let $\mytag{APCLB^\circ}$ 
	be \myref{APCLB} with~\eqref{cut:assignedToLaterSet} 
	and~\eqref{eq:optcuts}, and denote the 
	optimal objective function value with $\mathcal{L}^{\circ}_\alpha(LB)$. 
	Let $\mytag{APCLB^{\prime\circ}}$ 
be \myref{APCLB'} with~\eqref{cut:assignedToLaterSet} and~\eqref{eq:optcuts}, 
and denote the optimal 
objective function value with $\mathcal{L}^{\prime\circ}_\alpha(LB)$.	

Then
\begin{itemize}
	\itemsep0em
	\item[(a)] $\mathcal{L}^{\circ}_\alpha(LB) = LB$ if and only if 
	$\mathcal{L}_\alpha(LB) = LB$, and
	\item[(b)] $\mathcal{L}^{\prime\circ}_\alpha(LB) = LB$ if and only if 
	$\mathcal{L}^{\prime}_\alpha(LB) = LB$.
\end{itemize}
\end{theorem}
\begin{proof}
	To prove (a), it is enough to show that 
	$\mathcal{L}^{\circ}_\alpha(LB) = LB$ if and only if there is a feasible 
	solution for~\eqref{eq:setcoverDaskin} with objective function value at 
	most $p$, because of 
	Theorem~\ref{thm:convergenceDaskin}. 
	To do so, we can follow the proof of 	
	Theorem~\ref{thm:convergenceDaskin}. In particular, Part 1 can be used 
	without modifications. Also steps one, two and three of Part~2 can be used 
	without changes. Only in step four we have \blue{to} additionally show that 
	$(x^*,y^*,z^*)$ fulfills~\eqref{cut:assignedToLaterSet} 
	and~\eqref{eq:optcuts}.
	Clearly $(x^*,y^*,z^*)$ satisfies~\eqref{eq:optcuts} due 
	to~\eqref{eq:afsc1aD} and the fact that $LB \leq UB$. 
	 
	To show that also~\eqref{cut:assignedToLaterSet} is fulfilled we fix some 
	$i \in \pts$ and some $\pts_\alpha 
	\subseteq \pts$ with $|\pts_\alpha| = \alpha$. Let $j_{\alpha} \in 
	\pts_\alpha$ be the 
	maximum 
	entry of $\pts_\alpha$ according to $\sigma_i$, i.e., such that 
	$\pts_\alpha \subseteq (S_{ij_\alpha} \cup \{j_\alpha\})$. 
	Then~\eqref{cut:assignedToLaterSet} can be reformulated to
	\begin{alignat}{3}\label{proof:reformulationAssignedToLater}
	\sum_{j \in \pts_\alpha} y_{j} + \sum_{j \in \pts \setminus  
		(S_{ij_\alpha} \cup \{i,j_\alpha\} ) } x_{ij} 
	& \leq \alpha,
	\end{alignat}
	so it is enough to show that~\eqref{proof:reformulationAssignedToLater} 
	holds 
	for $(x^*,y^*,z^*)$. 	
	
	If $j_i \in (S_{ij_\alpha} \cup \{j_\alpha\})$, i.e., if $j_i$ is before 
	$j_\alpha$ according to the order $\sigma_i$ and thus $j_i$ is closer or at 
	the same distance to $i$
	than $j_\alpha$, then we can deduce that $x^*_{ij} = 0$ for all $j \in \pts 
	\setminus  
	(S_{ij_\alpha} \cup \{i,j_\alpha\})$ by construction, because all of these 
	$j$ are further away from $i$ than $j_i$ is. Thus, this implies that 
	\eqref{proof:reformulationAssignedToLater} is fulfilled in this case, 
	as $|\pts_\alpha| = 
	\alpha$ and $y^*_j \leq 1$ for all $j \in \pts$.
	
	If $j_i \not\in (S_{ij_\alpha} \cup \{j_\alpha\})$, i.e., if $j_i$ is 
	further away to $i$ than $j_\alpha$ is, 
	then $x^*_{ij} = \yNewOpti_j = 
	\min\{y^*_j,1-y^*_i\}$ holds for all $j \in \pts_\alpha$ by construction.
	Thus, we can define $\varepsilon_j$ such that $y^*_j = x^*_{ij} + 
	\varepsilon_j$ 
	for each $j \in \pts_\alpha$, because either $y^*_j = x^*_{ij}$ and  
	$\varepsilon_j = 0$, or $y^*_j > x^*_{ij} = 1-y^*_i$ and  
	$\varepsilon_j = y^*_j - (1- y^*_i)$. 
	In any case, $0 \leq \varepsilon_j$ and $\varepsilon_j \leq y^*_i$, as 
	$y^*_j 
	\leq 1$. 
	This, together with the already shown~\eqref{pc1:sumx}, implies that
	\begin{alignat*}{3}
	\sum_{j \in \pts_\alpha} y^*_{j} 
	+ \sum_{j \in \pts \setminus  (S_{ij_\alpha} \cup \{i,j_\alpha\} ) } 
	x^*_{ij} 
	&= 
	\sum_{j \in \pts_\alpha} (x^*_{ij} + \varepsilon_j) 
	+ \sum_{j \in \pts \setminus (S_{ij_\alpha} \cup \{i,j_\alpha\} ) } x^*_{ij}
	\\
	&\leq 
	\sum_{j \in \pts_\alpha} \varepsilon_j
	+
	\sum_{j \in \pts \setminus \{i\} } x^*_{ij}
	\leq
	\alpha y^*_i
	+ \alpha(1-y^*_i)
	= \alpha,
	\end{alignat*}	
	so~\eqref{proof:reformulationAssignedToLater} holds also in this case.
	Thus $(x^*,y^*,z^*)$ fulfills~\eqref{cut:assignedToLaterSet}, which 
	finishes the proof of (a).
	
	The proof of (b) can be done analogously with the help of 
	Theorem~\ref{thm:convergenceDaskinRelaxation} and is therefore skipped.
\end{proof}

Theorem~\ref{thm:convergenceWithOptimalityInequ} shows that 
adding the optimality-preserving inequalities~\eqref{cut:assignedToLaterSet} 
and~\eqref{eq:optcuts} to the iterative 
lifting does not improve 
the best lower bounds obtained $LB_\alpha^{\#}$ 
and 
${LB_\alpha^{\#}}'$.

\subsection{Best lower bound based on (APC2)}
\label{sec:comparison:e}

Next, we analyze \myref{APC2} for the \APCP in a similar
way \cite{elloumi2004} and \cite{ales2018} have done with \myref{PCE} and 
\myref{PCA} for the  \PCP. To do so, we introduce a 
\emph{semi-relaxation} 
$\mytag{APC2-Ry}$ of \myref{APC2} which is defined as \myref{APC2} with relaxed 
$y$-variables, i.e., \myref{APC2-Ry} is
\myref{APC2} without~\eqref{pcE:ybin} and with the constraints $0 \leq y_j \leq 
1$ 
for all $j \in \pts$ instead.
In the same fashion, let $\mytag{PCE-Ry}$ be 
the formulation 
\myref{PCE} without the constraints $y_j \in \{0,1\}$ and with the constraints 
$0 \leq y_j \leq 
1$ 
for all $j \in \pts$. In case of the  \PCP, the semi-relaxation 
\myref{PCE-Ry} of 
\cite{elloumi2004} has several interesting properties, which we now investigate 
in analogous form for the \APCP. 

\subsubsection{Computation in polynomial time}

First, for the  \PCP the optimal objective function value of the 
semi-relaxation 
\myref{PCE-Ry} can be 
computed in polynomial time as shown by \cite{elloumi2004}. 
Our next aim is to present a procedure for the \APCP to compute also the 
optimal value of 
the 
semi-relaxation \myref{APC2-Ry} in polynomial time. To do so, we first need the 
following result.

\begin{lemma}\label{thm:firstFractionalFixed}
	Let $k'$ be such that 
	\begin{alignat}{3}
	u_k & = 1 \qquad && \forall k \in \{2, \dots, K\}, k \leq k'
	\label{eq:condition}
	\end{alignat}
	are valid equalities for both \myref{APC2} and \myref{APC2-Ry}.
	Let $(y^*,u^*)$ be an optimal solution of \myref{APC2-R} 
	with~\eqref{eq:condition}.
	If $u^*$ is binary, let $k^*$ be the largest $k$ such that $u^*_k = 1$. 
	If $u^*$ is not binary, let 
	$k^*$ be the 
	smallest $k$ such that $u^*_k < 1$, i.e.,  
	$u^*_{k^*}$ is the first fractional 
	entry of $u^*$. Then the constraints 
	\begin{alignat}{3}
	u_k & = 1 \qquad && \forall k \in \{2, \dots, K\}, k \leq k^*
	\label{eq:lifting_uk}
	\end{alignat}
	are valid equalities for both \myref{APC2} and \myref{APC2-Ry}, i.e., when 
	adding~\eqref{eq:lifting_uk}  
	to \myref{APC2} and \myref{APC2-Ry}, the respective sets of feasible 
	solutions do not change.
\end{lemma}
\begin{proof}
%
	If $u^*$ is binary, $k^*$ is chosen in such a way that the optimal 
	objective function value of \myref{APC2-R} 
	with~\eqref{eq:condition} is $d_{k^*}$.
	If $u^*$ is not binary, then $k^*$ is chosen in such a way that the optimal 
	objective function value of \myref{APC2-R} 
	with~\eqref{eq:condition} is larger than $d_{k^*-1}$. 
	Thus, in any case, the optimal 
	objective function value of \myref{APC2-R} 
	with~\eqref{eq:condition} is larger than $d_{k^*-1}$.
	
	\orange{
	Assume that~\eqref{eq:lifting_uk} is not a valid equality for 
	\myref{APC2}. Then there is a feasible solution $(u^\circ,y^\circ)$ of 
	\myref{APC2} and there is a $k^\circ \leq k^*$ such that 
	$u^\circ_{k^\circ} = 0$. Then $u^\circ_k = 0$ for all $k \geq k^\circ$ 
	because of~\eqref{pc3:ulink} and $u^\circ_k \leq 1$ for all  $k < k^\circ$. 
	Thus, the objective function value of $(u^\circ,y^\circ)$ for 
	\myref{APC2-R} 
	with~\eqref{eq:condition}, which is equal to $d_1 + \sum_{k=2}^{K} (d_k - 
	d_{k-1})u^\circ_k$, is at most $d_{k^\circ-1}$ and therefore it is at most 
	$d_{k^*-1}$. 
	Furthermore, $(u^\circ,y^\circ)$ is feasible for \myref{APC2-R} 
	with~\eqref{eq:condition}, because \myref{APC2-R} with~\eqref{eq:condition} 
	is a relaxation of \myref{APC2}.
	Thus, the optimal objective function value of \myref{APC2-R} 
	with~\eqref{eq:condition} is at most $d_{k^*-1}$, a contradiction.
	Therefore, the assumption was wrong and~\eqref{eq:lifting_uk} is a valid 
	equality for \myref{APC2}.
	
	The fact that~\eqref{eq:lifting_uk} is a valid equality for 
	\myref{APC2-Ry} can be shown analogously.}	
\end{proof}


As a consequence, by applying 
Lemma~\ref{thm:firstFractionalFixed} in an iterative fashion, we 
can compute \myref{APC2-Ry} in polynomial time, as the next results shows.

\begin{theorem}\label{thm:semiRelaxPolyTime}
	An optimal solution of the semi-relaxation \myref{APC2-Ry} can be computed 
	in polynomial time.
\end{theorem}
\begin{proof}
	We can compute an optimal solution of \myref{APC2-Ry} as follows. First, we 
	set $k'  = 1$. Then we solve \myref{APC2-R} 
	with~\eqref{eq:condition}, which is equivalent to \myref{APC2-R} in the 
	case that $k' 
	= 1$ holds. Let $(y^*,u^*)$ be the obtained optimal
	solution. If $u^*$ is binary, it is an optimal 
	solution of \myref{APC2-Ry}. Otherwise, we can apply  
	Lemma~\ref{thm:firstFractionalFixed} to obtain $k^*$, update $k' = k^*$ and 
	solve 
	\myref{APC2-R} with~\eqref{eq:condition} again. We 
	repeat this, until we obtain a binary $u^*$. 
	
	Note that $k'$ increases at least by one in each iteration, and there are 
	$O(|\pts|^2)$ many potential values of $k'$. Furthermore, in each 
	iteration a linear program with a polynomial number of variables and 
	constraints 
	has to be solved. Thus, this procedure computes an optimal solution of 
	\myref{APC2-Ry} in polynomial time.	
\end{proof}

\subsubsection{Combinatorial interpretation}

A second interesting property of the semi-relaxation \myref{PCE-Ry} for the 
 \PCP is that  
\cite{gaar2022scaleable} proved that it is connected to the optimal solution of 
a 
set cover problem. In particular, the optimal objective function value of 
\myref{PCE-Ry} 
is equal to $d^* \in D$ if and only if there is a fractional set cover solution 
with radius $d^*$ that uses at most $p$ sets. It turns out that the following  
analogous result is also true for  \myref{APC2} for the \APCP.

\begin{theorem}
	\label{thm:convergence}
	Let $d^* \in D$. 
	Then the optimal objective function value of 
	\myref{APC2-Ry} is equal to $d^*$ 	
	if and only if 
	$d^*$ is the smallest possible value of $\delta$ such that
	there 
	is a feasible solution for $\myref{AFSC}_{\delta}$ 
	with objective function value at most $p$.
\end{theorem}
\begin{proof}
	 As \myref{APC2-Ry} requires the 
	$u$-variables to be binary and \eqref{pc3:ulink} has to hold, it is clear 
	that the optimal objective function value of \myref{APC2-Ry} is a value 
	from $D$. Furthermore, it is clear that the smallest possible value of 
	$\delta$ such that
	there 
	is a feasible solution for $\myref{AFSC}_{\delta}$ 
	with objective function value at most $p$ is a value 
	from $D$, because only for such values the problem
	$\myref{AFSC}_{\delta}$ changes. Thus, in order to prove the result it is 
	enough to show that for any $\delta \in D$
	there is a feasible solution for \myref{APC2-Ry} with objective function 
	value $\delta$ if and only if there 
	is a feasible solution for $\myref{AFSC}_{\delta}$ 
	with objective function value at most $p$.
	We will finish the proof by showing each side of this equivalence in a 
	separate part.
	
	\textbf{Part 1:}
	Let $(u^*,y^*)$ be a feasible solution of \myref{APC2-Ry} with objective 
	function value $\delta \in D$. We will finish this part of the proof by 
	showing 
	that $y^*$ is a feasible solution for $\myref{AFSC}_{\delta}$ 
	with objective function value at most $p$.
	
	Towards this end, let $\ell$ be such that $\delta = d_\ell$. 
	Then~\eqref{pc3:ulink} together with~\eqref{pcE:ubin} imply that $u^*_k = 
	1$ for all $k \in \{1, \dots, \ell\}$ and $u^*_k = 0$ for all $k \in 
	\{\ell+1, \dots, K\}$. 
	Now we fix some $i \in \pts$ and distinguish two cases.
	
	If there is an element in $D_i$ that is larger than $d_\ell$, then let 
	$d_{\ell'} = \min_{k \in \{1, \dots, K\}}\{d_k \in D_i: d_k > d_\ell\}$, 
	i.e., $d_{\ell'}$ is the smallest entry of $D_i$ that is larger than 
	$d_\ell$. By construction $u^*_{\ell'} = 0$ holds, so~\eqref{pc3:sumyu} for 
	$d_k = d_{\ell'}$ implies that
	\begin{alignat*}{3}
	\alpha(1- 
	y^*_i) 
	\leq   \sum_{j \in \pts \setminus \{i\} :  d_{ij} < d_{\ell'}} y^*_{j}  
	= \sum_{j \in \pts \setminus \{i\} :  d_{ij} \leq d_{\ell}} y^*_{j}
	= \sum_{j \in \pts \setminus \{i\} :  d_{ij} \leq \delta} y^*_{j},
\end{alignat*}
	so in this case~\eqref{eq:afsc1a} is satisfied for $i$.
	
	If there is no element in $D_i$ that is larger than $d_\ell$, then $d_{ij} 
	\leq d_{\ell}$ for all $j \in \pts \setminus \{i\}$ and 
	with~\eqref{pc3:sumy} \elli{and $1 \leq \alpha \leq p$} this
	implies that 
	\begin{alignat*}{3}
\sum_{j \in \pts \setminus \{i\} :  d_{ij} \leq \delta} y^*_{j}
= \sum_{j \in \pts \setminus \{i\} :  d_{ij} \leq d_{\ell}} y^*_{j}
= \sum_{j \in \pts \setminus \{i\}} y^*_{j}
= p - y^*_i  
\geq \alpha(1- 
y^*_i),
\end{alignat*}
so also in this case~\eqref{eq:afsc1a} is satisfied for $i$.

As a result, the inequality~\eqref{eq:afsc1a} is satisfied by $y^*$ in any 
case. 
Furthermore, $y^*$ fulfills~\eqref{eq:afsc2} because it satisfies the 
relaxation 
of~\eqref{pcE:ybin}. The objective function value~\eqref{eq:afsco} of $y^*$ is 
equal to $p$ 
because of~\eqref{pc3:sumy}. Thus, $y^*$ is feasible for 
$\myref{AFSC}_{\delta}$ with objective function value at most $p$. 
	
	\textbf{Part 2:}
	Assume $\delta \in D$ is such that there is a feasible solution $y\newOpt$  
	for $\myref{AFSC}_{\delta}$ 
	with objective function value at most $p$. 
	We will finish this part of the proof by constructing a
	feasible solution $(u^*,y^*)$ for 
	\myref{APC2-Ry} with objective function 
	value $\delta$.
	
	Towards this end, let $u^*_1 = 1$ if $d_k 
	\leq \delta$ and let $u^*_k = 0$ otherwise. 
	Furthermore, we construct $y^*$ from $y\newOpt$ in the same 
	fashion as in Part 2 of the proof of 
	Theorem~\ref{thm:convergenceDaskin}. In particular, let $\pNewOpt$ be the 
	objective function value~\eqref{eq:afsco} of $\yNewOpt$, so $\pNewOpt = 
	\sum_{j \in \pts} \yNewOpt_j$ and construct $y^*$ as $y^*_j = 
	\yNewOpt_j + 
	(1-\yNewOpt_j)\frac{p-\pNewOpt 
	}{|\pts| - \pNewOpt}$ for all $j \in \pts$. With the same arguments as in 
	the proof of Theorem~\ref{thm:convergenceDaskin} it follows that $y^*$ is 
	feasible for $\myref{AFSC}_{\delta}$ and has objective function value $p$.

	By construction, $u^*$ fulfills~\eqref{pc3:ulink} 
	and~\eqref{pcE:ubin}. Furthermore $y^*$ fulfills the relaxation 
	of \eqref{pcE:ybin} because of~\eqref{eq:afsc2}, and it 
	satisfies~\eqref{pc3:sumy} because it has  objective function value $p$ for 
	$\myref{AFSC}_{\delta}$. 
	Next we consider the inequalities~\eqref{pc3:sumyu} for some $i \in \pts$. 
	This inequality is clearly satisfied for any $d_k$ such that $u^*_k = 1$. 
	If $d_k \in D_i$ is such that $u^*_k = 0$, then by construction $d_k > 
	\delta$. This together with \eqref{eq:afsc1a} implies that
	\begin{alignat*}{3}
	\alpha(1-y^*_i) 
	\leq  \sum_{j \in \pts \setminus \{i\}:  d_{ij}  \leq \delta}y^*_{j}
	\leq  \sum_{j \in \pts \setminus \{i\}:  d_{ij}  < d_k}y^*_{j}
	= \sum_{j \in \pts \setminus \{i\}:  d_{ij}  < d_k}y^*_{j} + \alpha u^*_k, 
\end{alignat*}
	so~\eqref{pc3:sumyu} holds in any case.
	
	As a consequence, $(u^*,y^*)$ is feasible for \myref{APC2-Ry}. By 
	construction, and because $\delta \in D$, it follows that the objective 
	function value of $(u^*,y^*)$ for \myref{APC2-Ry} is $\delta$, which closes 
	this part of the proof.
\end{proof}

\subsection{Comparison of the best lower bounds}
\label{sec:comparison:comparede}

Finally, we compare the best lower bounds obtainable 
with the two formulations. For the  \PCP, \cite{gaar2022scaleable} 
proved that 
iteratively using the lower bound information for \myref{PC1} yields a bound, 
which coincides with the bounds obtained by the semi-relaxation \myref{PCE-Ry}. 

It turns out that this may not the case anymore for the \APCP. 
Towards this end, let $LB^*_\alpha$ be the optimal objective function of 
\myref{APC2-Ry}. Then we can deduce the following result.
\begin{theorem}
	It holds that $LB_\alpha^{\#} \geq {LB_\alpha^{\#}}' = LB^*_\alpha$. 
\end{theorem}
\begin{proof}
	This is a consequence of Corollary~\ref{cor:boundComparisonDaskin} and 
	Theorem~\ref{thm:convergenceDaskinRelaxation} and~\ref{thm:convergence}. 
\end{proof}

As a consequence, for the \APCP, when all our valid inequalities are included, 
the model \myref{APC1} 
produces as least as 
good bounds as the semi-relaxation of \myref{APC2}, and might produce better 
bounds.

\section{Implementation details \label{sec:implementation}}

Since both formulations are of polynomial size, they could be directly given to 
an integer programming solver for moderately-sized instances. However, we have 
implemented \BC approaches based on them which incorporate our valid 
inequalities and the lifted versions of it, our optimality-preserving 
inequalities, a starting heuristic and a primal heuristic, and variable 
fixing 
procedures. Moreover, for both formulations, we do not start the solution 
process with all the inequalities of the formulation, but add some of 
them 
on-the-fly when needed using separation procedures. 

We first describe the 
starting heuristic and the primal heuristic, which are used by both \BC 
algorithms in Section~\ref{sec:heu}. Then, we give a description 
of our \BC based on \myref{APC1} 
in Section \ref{sec:bc1}, and a description of our \BC based on 
\myref{APC2} in Section \ref{sec:bc2}. We evaluate the effects of the 
different 
ingredients of the \BC algorithms on the performance in Section 
\ref{sec:ingredients}.

\subsection{Starting heuristic and primal heuristic}
\label{sec:heu}

Our starting heuristic is a greedy heuristic. We initialize the (partial) 
solution $P$ by randomly picking a location $j \in \pts$ to open a 
facility. We then grow $P$ by iteratively adding additional locations to 
$P$ in a greedy fashion until $|P|=p$. As a greedy criterion to 
choose the location to add to $P$, we take the location $j \in \pts 
\setminus 
P$ which has the largest $\alpha$-distance to $|P|$. We note that 
if $|P|<\alpha$ this criterion is not well-defined, and thus in 
this 
case we use the $|P|$-distance. We run this heuristic \texttt{startHeur} 
times before we start with the \BC and initialize the \BC with the best 
solution found.

Our primal heuristic is a greedy heuristic driven by the values $y^*$ of the 
$y$-variables of the linear relaxation at the nodes of the \BC tree. The 
heuristic simply sorts the locations $j \in \pts$ in descending order according 
to $y_j^*$ and picks the $p$-largest ones as a solution. The primal heuristic 
is implemented within the \elli{\texttt{HeuristicCallback}} of CPLEX, which is 
the 
mixed-integer programming solver we are using.

\subsection{Implementation details of the branch-and-cut based on 
(APC1)
\label{sec:bc1}} 

\subsubsection{Variable fixing}

We use the solution value $UB$ from the solution obtained by the starting 
heuristic to fix the $x$-variables to zero as described in Theorem 
\ref{thm:optcuts} at initialization. During the \BC we continue with this 
variable fixing procedure by adding these fixings in the 
\blue{\texttt{UserCutCallback}} of CPLEX in case an improved primal solution 
found 
during the \BC allows additional fixings. This callback gets called by CPLEX 
whenever the solver encounters a fractional solution during the solution 
process.

\subsubsection{Overall separation scheme}

We separate the following inequalities in the branch-and-cut, where the 
order below gives the order in which we do the separation.

\begin{enumerate}
	\itemsep0em
	\item valid inequalities \eqref{cut:sumdijxij_alphaz}/their lifted version \eqref{cut:sumdijxij_alphazlifted}
	\item inequalities \eqref{pc1:dx}/their lifted version 
	\eqref{pc1:dxlifted}  from the original formulation
	\item inequalities \eqref{pc1:xy} from the original formulation
	\item optimality-based inequalities \eqref{eq:optcuts}
	\item optimality-based inequalities \eqref{cut:assignedToLaterSet}
\end{enumerate}

The inequalities listed above are separated within the 
\elli{\texttt{UserCutCallback}}. Inequalities \eqref{pc1:dx} and \eqref{pc1:xy} 
from 
the formulation, which are needed for the correctness of our algorithm, are 
also separated within the \texttt{LazyConstraintCallback}, which gets called by 
CPLEX for each integer solution (i.e., each potential new feasible solution). 
We perform at most \texttt{maxSepRoot} separation-rounds at the root-node and 
at most \texttt{maxSepTree} separation-rounds at the other nodes of the \BC 
tree. In the root-node, we add at most \texttt{maxIneqsRoot} violated 
inequalities in a separation-round and at the other nodes, we add at most 
\texttt{maxIneqsTree} violated inequalities. The parameter-values we used in 
our computations are given in Section \ref{sec:results}. Note that depending on 
the setting selected, in the computational study not all the inequalities above 
are actually used. For more details see Section \ref{sec:ingredients}.

\subsubsection{Details about the separation procedures}

All inequalities except \eqref{cut:assignedToLaterSet} are separated by 
enumeration. We note that the lifted inequalities 
\eqref{cut:sumdijxij_alphazlifted} and \eqref{pc1:dxlifted} depend on the 
current lower bound $LB$ and the inequalities \eqref{eq:optcuts} depend 
on the 
current upper bound $UB$. Thus, these inequalities can potentially be added 
again in a stronger version for fixed $i$ and $j$ or for a fixed $i$, 
when 
an improved bound becomes available. For this reason, we add them with the 
CPLEX-option \texttt{purgeable}, which allows CPLEX to remove added 
inequalities 
if they are deemed no longer useful by CPLEX. Moreover, during the \BC tree, we 
can use the local lower bounds from the nodes of the \BC tree as $LB$ for 
the inequalities \eqref{pc1:dxlifted} and 
\eqref{cut:sumdijxij_alphazlifted}. 
Naturally, the inequalities are then only valid for the subtree starting at 
this node. CPLEX allows to add such locally valid inequalities with the method 
\texttt{addLocal}\footnote{Unfortunately it is not possible to combine 
\texttt{addLocal} with \texttt{purgeable}. Thus, outside of the root-node, the 
linear programs solved within the \BC tree can contain redundant 
inequalities.}. When 
separating the inequalities \eqref{pc1:xy} and \eqref{pc1:dx}, for each 
point 
$i$ we add the ones corresponding to the \texttt{numInitAPC1} nearest locations 
$j$ at initialization.

The separation routine for the inequalities 
\eqref{cut:assignedToLaterSet} is a heuristic. For a given location $i \in 
\pts$, our goal is to find a set $X \subseteq \pts 
\setminus \{i\}$ 
and a set $N_\alpha \subseteq \pts \setminus (X \cup \{i\})$, with
$|N_\alpha|=\alpha$
and such that $N_\alpha$ contains only $j \in \pts \setminus (X \cup \{i\})$ 
with $d_{ij}<\min_{j' \in X} d_{ij'}$.
For any such $X$ and $N_\alpha$
\begin{align}\label{cut:assignedToLaterSetRelaxed}
\sum_{j \in N_\alpha} y_{j}+ \sum_{j \in X} x_{ij} \leq \alpha
\end{align}
is a relaxation of the valid inequality~\eqref{cut:assignedToLaterSet}, and 
hence also~\eqref{cut:assignedToLaterSetRelaxed} is a valid inequality.
Thus, we want to heuristically find such sets $X$ and $N_\alpha$
which maximize $\sum_{j \in N_\alpha} y^*_{j}+ \sum_{j \in X} x^*_{ij}$,  
where $(x^*,y^*)$ is the solution of the linear programming-relaxation at 
the nodes of the \BC tree. Then, if we have that $\sum_{j \in 
N_\alpha} y^*_{j}+ 
\sum_{j \in X} x^*_{ij}>\alpha$, we have obtained a violated 
inequality~\eqref{cut:assignedToLaterSetRelaxed} and thus also a violated 
inequality~\eqref{cut:assignedToLaterSet}. 

The heuristic proceeds as follows: Let 
$\pts_i$ be the locations $j \in \pts \setminus \{i\}$ sorted in 
descending order 
according to $d_{ij}$. We initialize $X$ with the first entry 
of $\pts_i$. Based on $X$, all potential candidates of $N_\alpha$ are 
all $j \in \pts \setminus (X \cup \{i\})$ with $d_{ij}<\min_{j' \in X} 
d_{ij'}$. To obtain $N_\alpha$, we sort all candidates $j$ according to 
their $y_j^*$-value in descending order, and take 
the $\alpha$ largest ones. If the 
inequality~\eqref{cut:assignedToLaterSetRelaxed}  
implied by $X$ and $N_\alpha$ is violated, we stop and 
add~\eqref{cut:assignedToLaterSetRelaxed} for $N_\alpha$ and this $X$, if 
not, we 
continue by adding the next entry from $\pts_i$ to $X$ and repeat the 
procedure.

\subsubsection{Branching priorities}

CPLEX allows to set branching priorities on the variables, which it then takes 
into account during the \BC. We set the priorities of the $y$-variables to 
100\footnote{Every nonegative value should 
	already be fine to give 
	higher 
	priority, the documentation is unfortunately not very clear about this, see 
	\url{https://www.ibm.com/docs/en/icos/20.1.0?topic=cm-setpriority-method}.} 
	and the priorities of the $x$-variables are 
left at the default value of zero in 
order to force CPLEX to branch on the $y$-variables first. This is done, as 
fixing $y$-variables is likely to have more structural impact on the 
linear programming relaxations compared to fixing $x$-variables. 

\subsection{Implementation details of the branch-and-cut based on 
(APC2)\label{sec:bc2}} 

\subsubsection{Variable fixings}

Similar to our approach for \myref{APC1}, we use the solution value $UB$ from 
the solution obtained by the starting heuristic for variable fixing, i.e., we 
fix the $u$-variables to zero as described in Theorem \ref{thm:optcutsElloumi} 
at initialization. Moreover, we also continue these fixings in the 
\elli{\texttt{UserCutCallback}} whenever an improved incumbent is found. 

Furthermore, we 
also fix the $u$-variables to one in the \elli{\texttt{UserCutCallback}} using 
the 
available (local) lower 
bound $LB$ at the current branch-and-cut node and the theory provided in 
Lemma \ref{thm:firstFractionalFixed}. Note that 
Lemma~\ref{thm:firstFractionalFixed} allows us to fix one fractional 
$u$-variable in each 
separation round. 
Thus, to speed-up the fixing, we first check if there are $k$ such 
that $u_k$ is fractional and $d_k \leq LB$, i.e., we check if there are  
$u$-variables that we can fix according to 
Theorem~\ref{thm:feascutsElloumi}. If yes, 
under all the $u$-variables fulfilling the conditions, we fix the one 
corresponding 
to the largest distance. By constraints \eqref{pc3:ulink} this setting will 
also set all variables corresponding to smaller distances to one. If there is 
no variable fulfilling this condition, then we use 
Lemma \ref{thm:firstFractionalFixed} for fixing. As we use the local lower 
bound for fixing, we add the fixing with the method \texttt{addLocal}.

\subsubsection{Details about the separation scheme}

We have implemented a separation routine for the inequalities 
\eqref{pc3:sumyu}. This allows us to dynamically add them when needed instead 
of adding all of them at initialization. This is an attractive option due to 
the structure of the formulation (in particular constraints \eqref{pc3:sumyu}) 
in combination with Lemma \ref{thm:firstFractionalFixed}. As this lemma 
provides results to fix $u$-variables to one, we may not need to add all 
inequalities \eqref{pc3:sumyu} to correctly measure the objective function 
value. 

Our separation routine is based on enumeration. However, we add at most 
one violated inequality \eqref{pc3:sumyu} per location $i \in \pts$ in each 
round of separation. In order to determine which inequality we add, if there is 
more than one inequality \eqref{pc3:sumyu} violated for a location $i$, we 
 compute 
$violation(u^*,y^*,i,k)=\alpha u^*_k +\sum_{j \in 
N\setminus\{i\}:d_{ij}<d_k}y^*_j-\alpha(1- y^*_i)$, where $(u^*,y^*)$ is the 
solution of 
the linear programming-relaxation at the nodes of the \BC tree. All 
inequalities with $violation(u^*,y^*,i,k)<0$ are violated.
Then we calculate the score $s=-violation(u^*,y^*,i,k) \cdot d_{ik}$. 
 With the score, we 
try to find a $k$ which gives a good balance between violation and effect on 
the objective function value.
When we apply the separation-approach, we initialize our \BC with all the 
inequalities \eqref{pc3:sumyu} corresponding to the \texttt{nInitAPC2} smallest 
distances of the instance. Since the inequalities \eqref{pc3:sumyu} are needed 
for correctness of the formulation, we call the separation routine both in the 
\elli{\texttt{UserCutCallback}} and also the 
\elli{\texttt{LazyConstraintCallback}}.

Regarding the number of separation rounds and the number of added violated 
 inequalities, we use the same strategy as described in Section \ref{sec:bc1}.

\subsubsection{Branching priorities}
Similar to the \BC for \myref{APC1}, we set the values of the branching 
priorities of the $y$-variables to 100, and the priorities of the $u$-variables 
are 
left at the 
default value of zero.

\section{Computational results} 
\label{sec:results}

We implemented our \BC algorithms in C++ using CPLEX 20.1. The runs were made on a single core of an Intel Xeon E5-2670v2 machine with 
2.5GHz and 6GB of RAM, and all CPLEX settings were 
left on their default values, except the branching priorities which we set as 
described in 
Section~\ref{sec:implementation}. We have set a time limit of 1800 seconds.

\subsection{Instances}

We considered two sets of instances from the literature in our computational study. The details of these sets are given below.

\begin{itemize}
	\itemsep0em
\item  \texttt{TSPLIB}: This  instance set is based on the TSP-library 
\citep{reinelt1991tsplib} and was used in \citet{sanchez2022grasp} with 
$\alpha=2,3$. In particular, the instances \texttt{att48}, \texttt{eil101}, 
\texttt{ch150}, \texttt{pr439}, \texttt{rat575}, 
\texttt{rat783}, \texttt{pr1002} and \texttt{rl1323} were used with $p \in 
\{10,20,\ldots,130,140\}$. The number in the instance-name gives the number of 
locations $|\pts|$. In these instances all locations are given as 
two-dimensional coordinates, and the Euclidean distance is used as a distance 
function. The instance set contains 154 instances.

We note that \citet{sanchez2022grasp} did not use all values of $p$ 
for all instances. In our computational study we considered the same 
combinations of instances and $p$ as \citet{sanchez2022grasp}. For the  
used values of $p$ for each instance see e.g., Tables~\ref{ta:tsplib1} 
and~\ref{ta:tsplib2}.

\item \texttt{pmedian}: This instance set is based on the OR-library 
\citep{beasley1990or}. It was used in \citet{mousavi2023exploiting} with 
$\alpha=2$. Each instance is given as a graph, and to obtain the distances 
between all the locations $\pts$ (nodes in the graph) an all-pair shortest-path 
computation needs to be done. In these instances, all the distances are 
integer. The number of locations $|\pts|$ is between $100$ and $900$, and $p$ 
is between $5$ and $200$. Each of these instances has a value of $p$ encoded in 
the instance. For the concrete values of $|\pts|$ and $p$ for each instance see 
Table~\ref{ta:pmed}. The instance set contains 40 instances.
\end{itemize}

\newcommand{\aaa}{\texttt{1}\xspace}
\newcommand{\ah}{\texttt{1H}\xspace}
\newcommand{\aha}{\texttt{1HS}\xspace}
\newcommand{\ahav}{\texttt{1HSV}\xspace}
\newcommand{\ahavl}{\texttt{1HSVL}\xspace}
\newcommand{\ahavlo}{\texttt{1HSVLO}\xspace}

\newcommand{\bbb}{\texttt{2}\xspace}
\newcommand{\bh}{\texttt{2H}\xspace}
\newcommand{\bhv}{\texttt{2HV}\xspace}
\newcommand{\bhvs}{\texttt{2HVS}\xspace}
\newcommand{\bhvsl}{\texttt{2HVSL}\xspace}

\subsection{Analysis of the ingredients of our branch-and-cut algorithms\label{sec:ingredients}}

To analyze the effect of the ingredients of our \BC algorithms, we performed a 
computational study on a subset of the instances, namely the instances 
\texttt{att48}, \texttt{eil101}, \texttt{ch150}.
We compare the following different settings for the \BC based on \myref{APC1}:

\begin{itemize}
	\itemsep0em
	\item \aaa: Directly solving \myref{APC1} without any additional ingredients
	\item \ah: Adding the starting heuristic, the primal heuristic and the 
	variable fixing based on the upper bound according to Theorem 
	\ref{thm:optcuts}
	\item \aha: Setting \ah with separation of the inequalities \eqref{pc1:xy} 
	and the inequalities \eqref{pc1:dx} (instead of adding them in the 
	beginning) as described in Section \ref{sec:bc1} 
	\item \ahav: Setting \aha together with the valid inequalities \eqref{cut:sumdijxij_alphaz} and \eqref{cut:yixij}
	\item \ahavl: Setting \ahav together with 
	the lifted version \eqref{pc1:dxlifted} of the inequalities \eqref{pc1:dx}
	and also 
	the lifted version \eqref{cut:sumdijxij_alphazlifted} of the inequalities 
	\eqref{cut:sumdijxij_alphaz}
	\item \ahavlo: Setting \ahavl together with the optimality-preserving inequalities \eqref{cut:assignedToLaterSet} and \eqref{eq:optcuts} 
\end{itemize}

For the \BC based on \myref{APC2} the following settings are considered:

\begin{itemize}
	\itemsep0em
	\item \bbb: Directly solving \myref{APC2} without any additional ingredients
	\item \bh: Adding the starting heuristic, the primal heuristic and the 
	variable fixing based on the upper bound according to Theorem 
	\ref{thm:optcutsElloumi}
	\item \bhv: Setting \bh with the valid inequalities \eqref{cut:ukyj} 
	replacing the corresponding inequalities \eqref{pc3:sumyu} according to 
	Observation~\ref{obs:elConstraintsReplacement}
	\item \bhvs: Setting \bhv with separation of the inequalities 
	\eqref{pc3:ulink} (instead of adding them in the beginning) as described in 
	Section \ref{sec:bc2}
	\item \bhvsl: Setting \bhvs with the variable fixing based on the lower bound according to Theorem \ref{thm:feascutsElloumi} and Lemma \ref{thm:firstFractionalFixed} as described in Section \ref{sec:bc2}
 	\end{itemize}

The following parameter 
values were used for the \BC algorithms, they were determined in preliminary 
computations: \texttt{startHeur}: 10, \texttt{maxIneqsRoot}: 50, 
\texttt{maxIneqsTree}: 20, \texttt{maxSepRoot}: 100, \texttt{maxSepTree}: 1, 
\texttt{numInitAPC1}: 10, \texttt{numInitAPC2}: 100 for the instance set 
\texttt{TSPLIB} and 10 for the instance set \texttt{pmedian}.
%
We have used a different parameter for \texttt{numIntiAPC2} depending on the 
instance set, as the distance-structure of the instances is quite different. In 
particular, for \texttt{TSPLIB} the distances are essentially unique (as they 
are Euclidean distances) while for \texttt{pmedian} many are similar (as they 
are shortest path distances on a graph). Thus, for \texttt{pmedian} we would 
often add all inequalities \eqref{pc3:sumyu} at initialization for a parameter 
value of 100, as there are usually less than 100 different distances in an 
instance.

In Figures \ref{fig:apc1} and \ref{fig:apc2} we show a plot of the runtimes. We 
see that for both formulations the largest positive effect is achieved by 
adding the heuristics with the associated variable fixing based on the upper 
bound. This can be explained by the fact that with the variable fixing the 
linear programs 
which are needed to be solved are getting much smaller. Moreover, the lifting 
procedures for \myref{APC1} and the variable fixing based on the lower bound 
for \myref{APC2} also have a discernible (incremental) effect. This is in line 
with both the computational results in \citet{gaar2022scaleable} for a 
similar lifting procedure for the \PCP, and the theoretical 
result 
provided in Section \ref{sec:polyhedral}. 

Starting to use separation of 
inequalities which are needed in the formulations (i.e., settings \aha  and 
\bhvs) instead of adding all of these inequalities at initialization has a 
rather neutral effect on the selected instances. This can be explained by the 
fact that these instances are quite small, for the larger instances in our 
sets, preliminary computations showed that we cannot even solve the 
root-relaxation (for both \myref{APC1} and \myref{APC2}) due to either running 
into the time limit or due to exceeding the available memory. 

The valid 
inequalities \eqref{cut:ukyj} also have no visible effect. A potential 
explanation of this is that modern mixed-integer programming solvers like CPLEX 
are quite effective in strengthening given inequalities and may already 
transform \eqref{pc3:ulink} into \eqref{cut:ukyj} automatically whenever it is 
possible.
Finally, adding the optimality-preserving inequalities for \myref{APC1} has a 
negative effect. This is consistent with 
Theorem~\ref{thm:convergenceWithOptimalityInequ}, which shows that at 
convergence the inequalities~\eqref{cut:assignedToLaterSet} 
and~\eqref{eq:optcuts} are not further 
improving the bound. 
 
\begin{figure}[tbh]
	\begin{center}

	\begin{subfigure}[b]{0.7\linewidth}
		\centering
		\includegraphics[width=\textwidth]{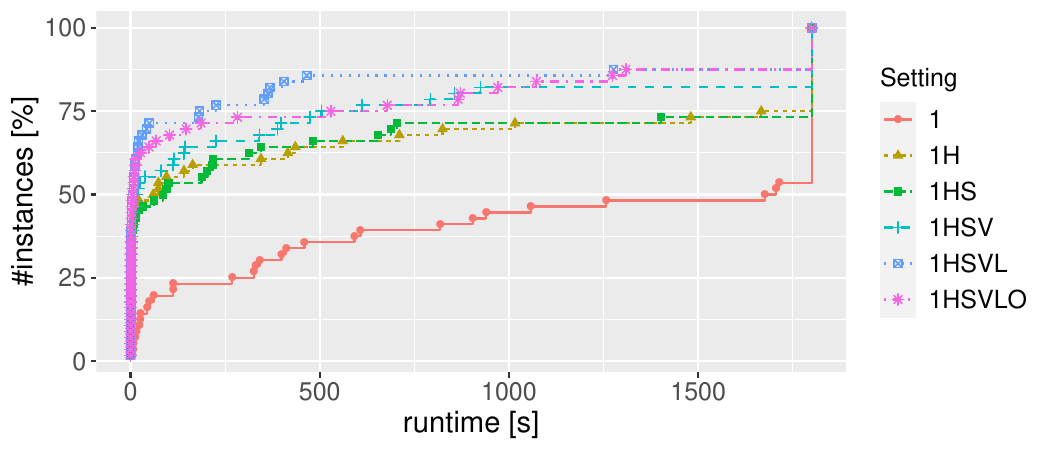}
		\caption{\myref{APC1}\label{fig:apc1}}
	\end{subfigure}

	\begin{subfigure}[b]{0.7\linewidth}
		\centering
		\includegraphics[width=\textwidth]{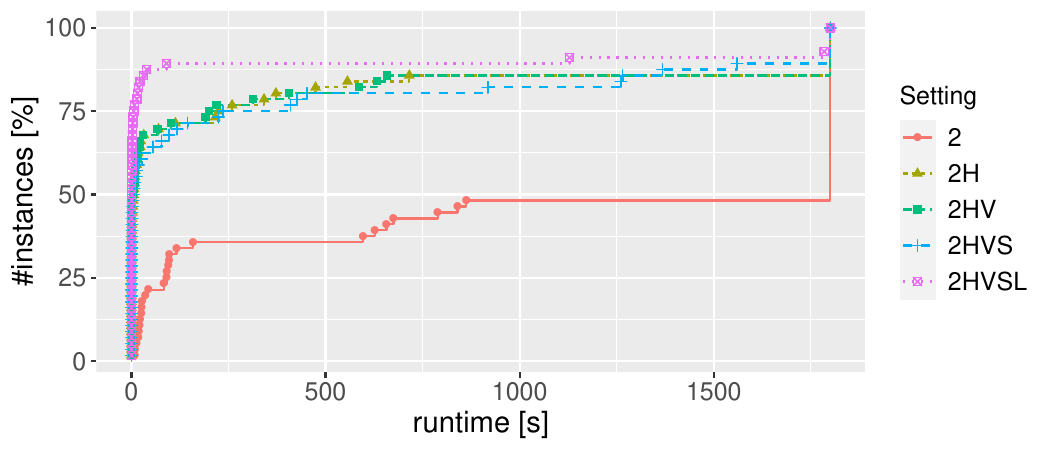}
		\caption{\myref{APC2}\label{fig:apc2}}
	\end{subfigure}
	\caption{Runtime for different settings of our \BC algorithms on a subset 
	of the instances.\label{fig:runtime}}
\end{center}
\end{figure}

\subsection{Comparison with approaches from the literature 
\label{sec:resultdetails}}

In this section we provide a detailed comparison with the existing approaches 
from literature, namely the GRASP of \citet{sanchez2022grasp} and the local 
search of \citet{mousavi2023exploiting} on the instances used in the respective 
works. We compare the existing approaches with the best settings for both of 
our \BC algorithms, i.e., \ahavl for the one based on \myref{APC1} and \bhvsl 
for the one based on \myref{APC2}.

In Tables \ref{ta:tsplib1}-\ref{ta:tsplib4} we give the comparison with 
\citet{sanchez2022grasp}. For our approaches we report the runtime (columns 
t[s] with entry TL 
indicating that the time limit of 1800 seconds was reached), the obtained upper 
bound (i.e., the objective function value of the best obtained solution, 
columns UB) and lower bound (columns LB) and the number of nodes in the \BC 
tree (columns nBC). Since the approach of \citet{sanchez2022grasp} is a 
heuristic, only 
upper bounds and runtime can be reported for their approach. We note that the 
runs in 
\citet{sanchez2022grasp} were made on a AMD Ryzen 5 3600 with 2.2 GHz and 16GB 
RAM. The best values for UB, LB and runtime are indicated in bold in the 
tables. For the runtime, we just consider our branch-and-cut approaches, while 
for the UB we consider all three approaches to determine these best values.

The tables show that for 114 out of 154 instances our approaches improve on the 
best solution value obtained in \citet{sanchez2022grasp} and additionally for 7 
instances, we match the best solution value. Our approaches manage to solve 76 
instances to proven optimality. For some of the instances, our approaches are 
more than two orders of magnitude faster than the GRASP (e.g., instance  
\texttt{pr439} with $p=30$ and $\alpha=2$). Comparing \ahavl with \bhvsl, we 
can see that \bhvsl performs better overall, in particular for larger 
instances. This can be explained by the fact that due to the structure of the 
formulations, the variable fixing procedures can fix much more variables when 
using \myref{APC2} compared to \myref{APC1}. We can also see that for 
$\alpha=3$ the problem is harder than for $\alpha=2$.  


\begin{table}[tbh]
	\centering
	\caption{Detailed results for instance set \texttt{TSPLIB} with $\alpha=2$, 
	part one. 
	\label{ta:tsplib1}} 
	\begingroup\footnotesize
	\begin{tabular}{lrr|rrrr|rrrr|rr}
		\toprule
		& & & \multicolumn{4}{|c|}{$\ahavl$} & \multicolumn{4}{|c|}{$\bhvsl$} & 
		\multicolumn{2}{|c}{\cite{sanchez2022grasp}} \\ name & $|N|$ & $p$ & UB 
		& LB & t[s] & nBC & UB & LB & t[s]  & nBC & UB & t[s] \\ \midrule
		att48 & 48 &  10 & \textbf{1592.12} & \textbf{1592.12} & 1.08 & 1 & 
		\textbf{1592.12} & \textbf{1592.12} & \textbf{0.26} & 0 & 
		\textbf{1592.12} &    5.14 \\ 
		att48 & 48 &  20 & \textbf{1061.69} & \textbf{1061.69} & 0.51 & 0 & 
		\textbf{1061.69} & \textbf{1061.69} & \textbf{0.08} & 0 & 1130.85 &    
		1.21 \\ 
		att48 & 48 &  30 & \textbf{729.90} & \textbf{729.90} & 0.34 & 0 & 
		\textbf{729.90} & \textbf{729.90} & \textbf{0.09} & 0 & 936.38 &    
		0.42 \\ 
		att48 & 48 &  40 & \textbf{485.06} & \textbf{485.06} & 0.07 & 0 & 
		\textbf{485.06} & \textbf{485.06} & \textbf{0.02} & 0 & 532.08 &    
		0.07 \\ 
		eil101 & 101 &  10 & \textbf{21.21} & \textbf{21.21} & 30.29 & 85 & 
		\textbf{21.21} & \textbf{21.21} & \textbf{1.79} & 13 & \textbf{21.21} 
		&   68.12 \\ 
		eil101 & 101 &  20 & \textbf{13.60} & \textbf{13.60} & 49.58 & 576 & 
		\textbf{13.60} & \textbf{13.60} & \textbf{2.07} & 35 & 14.14 &   28.27 
		\\ 
		eil101 & 101 &  30 & \textbf{11.05} & \textbf{11.05} & 20.66 & 364 & 
		\textbf{11.05} & \textbf{11.05} & \textbf{1.90} & 358 & 12.00 &   10.63 
		\\ 
		eil101 & 101 &  40 & \textbf{9.06} & \textbf{9.06} & 7.59 & 50 & 
		\textbf{9.06} & \textbf{9.06} & \textbf{0.54} & 68 & 9.43 &    6.19 \\ 
		eil101 & 101 &  50 & \textbf{8.06} & \textbf{8.06} & 8.48 & 340 & 
		\textbf{8.06} & \textbf{8.06} & \textbf{0.19} & 0 & 8.60 &    3.19 \\ 
		eil101 & 101 &  60 & \textbf{7.07} & \textbf{7.07} & 0.84 & 0 & 
		\textbf{7.07} & \textbf{7.07} & \textbf{0.14} & 0 & 8.25 &    1.94 \\ 
		eil101 & 101 &  70 & \textbf{6.32} & \textbf{6.32} & 0.78 & 0 & 
		\textbf{6.32} & \textbf{6.32} & \textbf{0.10} & 0 & 7.28 &    0.96 \\ 
		eil101 & 101 &  80 & \textbf{5.10} & \textbf{5.10} & 0.63 & 0 & 
		\textbf{5.10} & \textbf{5.10} & \textbf{0.08} & 0 & 6.32 &    0.43 \\ 
		eil101 & 101 &  90 & \textbf{4.12} & \textbf{4.12} & 0.67 & 1 & 
		\textbf{4.12} & \textbf{4.12} & \textbf{0.08} & 0 & 5.00 &    0.11 \\ 
		eil101 & 101 & 100 & \textbf{2.24} & \textbf{2.24} & 0.14 & 0 & 
		\textbf{2.24} & \textbf{2.24} & \textbf{0.12} & 0 & 2.83 &    0.05 \\ 
		ch150 & 150 &  10 & \textbf{205.66} & \textbf{205.66} & 1276.32 & 1364 
		& \textbf{205.66} & \textbf{205.66} & \textbf{4.27} & 0 & 
		\textbf{205.66} &  223.16 \\ 
		ch150 & 150 &  20 & \textbf{138.69} & \textbf{138.69} & 465.80 & 3235 & 
		\textbf{138.69} & \textbf{138.69} & \textbf{7.83} & 297 & 141.53 &   
		94.75 \\ 
		ch150 & 150 &  30 & \textbf{108.03} & \textbf{108.03} & 353.45 & 5441 & 
		\textbf{108.03} & \textbf{108.03} & \textbf{22.43} & 3614 & 112.51 &   
		55.58 \\ 
		ch150 & 150 &  40 & \textbf{92.67} & \textbf{92.67} & 226.43 & 5596 & 
		\textbf{92.67} & \textbf{92.67} & \textbf{6.02} & 932 & 96.42 &   31.74 
		\\ 
		ch150 & 150 &  50 & \textbf{82.11} & \textbf{82.11} & 42.61 & 935 & 
		\textbf{82.11} & \textbf{82.11} & \textbf{2.71} & 746 & 87.69 &   18.10 
		\\ 
		ch150 & 150 &  60 & \textbf{70.71} & \textbf{70.71} & 12.19 & 29 & 
		\textbf{70.71} & \textbf{70.71} & \textbf{0.73} & 0 & 78.42 &   12.24 
		\\ 
		ch150 & 150 &  70 & \textbf{64.45} & \textbf{64.45} & 3.90 & 1 & 
		\textbf{64.45} & \textbf{64.45} & \textbf{0.55} & 0 & 68.23 &    8.20 
		\\ 
		ch150 & 150 &  80 & \textbf{58.37} & \textbf{58.37} & 2.80 & 0 & 
		\textbf{58.37} & \textbf{58.37} & \textbf{0.50} & 0 & 64.45 &    5.57 
		\\ 
		ch150 & 150 &  90 & \textbf{51.50} & \textbf{51.50} & 2.08 & 0 & 
		\textbf{51.50} & \textbf{51.50} & \textbf{0.35} & 0 & 62.04 &    3.63 
		\\ 
		ch150 & 150 & 100 & \textbf{46.49} & \textbf{46.49} & 1.48 & 0 & 
		\textbf{46.49} & \textbf{46.49} & \textbf{0.30} & 0 & 53.21 &    2.35 
		\\ 
		ch150 & 150 & 110 & \textbf{43.77} & \textbf{43.77} & 1.29 & 0 & 
		\textbf{43.77} & \textbf{43.77} & \textbf{0.31} & 0 & 51.65 &    1.36 
		\\ 
		ch150 & 150 & 120 & \textbf{39.32} & \textbf{39.32} & 0.98 & 0 & 
		\textbf{39.32} & \textbf{39.32} & \textbf{0.27} & 0 & 50.30 &    0.72 
		\\ 
		ch150 & 150 & 130 & \textbf{36.02} & \textbf{36.02} & 0.52 & 0 & 
		\textbf{36.02} & \textbf{36.02} & \textbf{0.28} & 0 & 46.63 &    0.31 
		\\ 
		ch150 & 150 & 140 & \textbf{29.69} & \textbf{29.69} & 0.48 & 0 & 
		\textbf{29.69} & \textbf{29.69} & \textbf{0.29} & 0 & 42.30 &    0.14 
		\\ 
				\bottomrule
	\end{tabular}
\endgroup
\end{table}

		\begin{table}[tbh]
			\centering
			\caption{Detailed results for instance set \texttt{TSPLIB} with 
			$\alpha=2$, part two. 
				\label{ta:tsplib2}} 
			\begingroup\footnotesize
			\begin{tabular}{lrr|rrrr|rrrr|rr}
				\toprule
				& & & \multicolumn{4}{|c|}{$\ahavl$} & 
				\multicolumn{4}{|c|}{$\bhvsl$} & 
				\multicolumn{2}{|c}{\cite{sanchez2022grasp}} \\ name & $|N|$ & 
				$p$ & UB 
				& LB & t[s] & nBC & UB & LB & t[s]  & nBC & UB & t[s] \\ 
				\midrule
		pr439 & 439 &  10 & 4134.01 & 1944.76 & TL & 606 & \textbf{3146.63} & 
		\textbf{3146.63} & \textbf{12.29} & 0 & \textbf{3146.63} & TL \\ 
		pr439 & 439 &  20 & 2579.00 & 1337.89 & TL & 809 & \textbf{2177.44} & 
		\textbf{2177.44} & \textbf{27.75} & 136 & 2226.26 & TL \\ 
		pr439 & 439 &  30 & 1950.00 & 1020.79 & TL & 1082 & \textbf{1475.85} & 
		\textbf{1475.85} & \textbf{7.73} & 0 & 1500.21 & TL \\ 
		pr439 & 439 &  40 & 1614.00 & 865.38 & TL & 1097 & \textbf{1185.59} & 
		\textbf{1185.59} & \textbf{12.41} & 10 & 1253.99 & TL \\ 
		pr439 & 439 &  50 & 1308.63 & 722.83 & TL & 1182 & \textbf{984.89} & 
		\textbf{984.89} & \textbf{9.58} & 0 & 1068.00 & 1327.63 \\ 
		pr439 & 439 &  60 & 1116.08 & 653.73 & TL & 1201 & \textbf{883.88} & 
		\textbf{867.64} & TL & 32210 & 975.00 &  918.13 \\ 
		pr439 & 439 &  70 & 976.28 & 596.96 & TL & 1288 & \textbf{726.72} & 
		\textbf{726.72} & \textbf{1050.05} & 18676 & 905.54 &  639.50 \\ 
		pr439 & 439 &  80 & 855.13 & 554.50 & TL & 1629 & \textbf{637.38} & 
		\textbf{637.38} & \textbf{178.91} & 2513 & 731.86 &  509.87 \\ 
		pr439 & 439 &  90 & 742.04 & 516.31 & TL & 1121 & \textbf{583.10} & 
		\textbf{575.54} & TL & 43141 & 715.89 &  405.64 \\ 
		rat575 & 575 &  10 & 160.80 & 43.39 & TL & 293 & \textbf{116.10} & 
		\textbf{116.10} & \textbf{24.75} & 0 & 116.87 & 1773.45 \\ 
		rat575 & 575 &  20 & 97.45 & 37.73 & TL & 328 & \textbf{72.62} & 
		\textbf{72.62} & \textbf{265.55} & 1572 & 74.25 &  988.02 \\ 
		rat575 & 575 &  30 & 76.03 & 34.82 & TL & 433 & \textbf{59.14} & 
		\textbf{56.38} & TL & 8426 & 60.67 &  666.00 \\ 
		rat575 & 575 &  40 & 64.14 & 32.93 & TL & 396 & \textbf{50.25} & 
		\textbf{47.71} & TL & 9941 & 51.40 &  565.12 \\ 
		rat575 & 575 &  50 & 55.15 & 31.31 & TL & 491 & \textbf{45.88} & 
		\textbf{41.79} & TL & 9826 & 46.52 &  402.00 \\ 
		rat575 & 575 &  60 & 49.25 & 30.40 & TL & 468 & \textbf{41.15} & 
		\textbf{37.48} & TL & 12053 & 41.59 &  290.62 \\ 
		rat575 & 575 &  70 & 44.55 & 29.24 & TL & 376 & \textbf{37.48} & 
		\textbf{34.41} & TL & 13237 & 37.70 &  268.17 \\ 
		rat575 & 575 &  80 & 41.01 & 28.45 & TL & 354 & \textbf{34.99} & 
		\textbf{32.02} & TL & 19747 & 35.90 &  221.25 \\ 
		rat575 & 575 &  90 & 37.59 & 27.54 & TL & 357 & \textbf{32.45} & 
		\textbf{29.70} & TL & 23787 & 33.60 &  158.91 \\ 
		rat575 & 575 & 100 & 36.00 & 26.68 & TL & 300 & \textbf{30.00} & 
		\textbf{27.86} & TL & 27393 & 31.38 &  122.60 \\ 
		rat783 & 783 &  10 & 193.26 & 41.83 & TL & 45 & \textbf{135.25} & 
		\textbf{135.25} & \textbf{34.92} & 0 & 138.60 & TL \\ 
		rat783 & 783 &  20 & 109.42 & 38.17 & TL & 141 & \textbf{83.10} & 
		\textbf{83.10} & \textbf{25.68} & 0 & 86.38 & TL \\ 
		rat783 & 783 &  30 & 92.05 & 34.79 & TL & 134 & \textbf{67.88} & 
		\textbf{66.21} & TL & 2819 & 70.84 & 1717.02 \\ 
		rat783 & 783 &  40 & 75.72 & 32.85 & TL & 160 & \textbf{57.43} & 
		\textbf{55.60} & TL & 4464 & 60.14 & 1695.86 \\ 
		rat783 & 783 &  50 & 65.19 & 31.76 & TL & 100 & 55.04 & \textbf{49.20} 
		& TL & 2802 & \textbf{52.80} & 1212.41 \\ 
		rat783 & 783 &  60 & 56.59 & 31.17 & TL & 151 & 49.04 & \textbf{44.05} 
		& TL & 3872 & \textbf{48.75} & 1044.99 \\ 
		rat783 & 783 &  70 & 53.23 & 30.48 & TL & 90 & \textbf{44.20} & 
		\textbf{40.31} & TL & 4542 & 44.41 & 1038.25 \\ 
		rat783 & 783 &  80 & 49.65 & 29.97 & TL & 107 & \textbf{41.68} & 
		\textbf{37.36} & TL & 5300 & 42.43 &  748.93 \\ 
		rat783 & 783 &  90 & 46.39 & 29.26 & TL & 99 & 40.36 & \textbf{35.00} & 
		TL & 6154 & \textbf{39.20} &  722.81 \\ 
		rat783 & 783 & 100 & 42.64 & 28.76 & TL & 99 & 37.64 & \textbf{32.98} & 
		TL & 7309 & \textbf{37.48} &  536.23 \\ 
		pr1002 & 1002 &  10 & 5481.79 & 1223.80 & TL & 0 & \textbf{3853.89} & 
		\textbf{3853.89} & \textbf{48.07} & 0 & \textbf{3853.89} & TL \\ 
		pr1002 & 1002 &  20 & 3479.22 & 1130.48 & TL & 0 & \textbf{2593.26} & 
		\textbf{2543.32} & TL & 3359 & 2710.17 & TL \\ 
		pr1002 & 1002 &  30 & 2731.30 & 1009.07 & TL & 0 & \textbf{2059.73} & 
		\textbf{2008.76} & TL & 4002 & 2150.58 & TL \\ 
		pr1002 & 1002 &  40 & 2214.16 & 992.40 & TL & 18 & \textbf{1746.42} & 
		\textbf{1702.84} & TL & 7920 & 1811.77 & TL \\ 
		pr1002 & 1002 &  50 & 1990.60 & 949.95 & TL & 3 & \textbf{1523.15} & 
		\textbf{1478.03} & TL & 3555 & 1619.41 & TL \\ 
		pr1002 & 1002 &  60 & 1733.49 & 914.11 & TL & 0 & \textbf{1403.57} & 
		\textbf{1315.17} & TL & 6229 & 1431.78 & TL \\ 
		pr1002 & 1002 &  70 & 1555.63 & 851.41 & TL & 0 & 1372.95 & 
		\textbf{1204.16} & TL & 5818 & \textbf{1346.29} & TL \\ 
		pr1002 & 1002 &  80 & 1443.09 & 825.63 & TL & 3 & 1253.99 & 
		\textbf{1104.54} & TL & 6891 & \textbf{1253.00} & TL \\ 
		pr1002 & 1002 &  90 & 1365.65 & 803.99 & TL & 22 & \textbf{1131.37} & 
		\textbf{1033.32} & TL & 9010 & 1170.47 & 1696.72 \\ 
		pr1002 & 1002 & 100 & 1270.99 & 777.90 & TL & 0 & \textbf{1070.05} & 
		\textbf{982.98} & TL & 15670 & 1079.35 & 1337.80 \\ 
		rl1323 & 1323 &  10 & 6657.78 & 1138.05 & TL & 0 & \textbf{4554.09} & 
		\textbf{4554.09} & \textbf{660.59} & 0 & 4694.15 & TL \\ 
		rl1323 & 1323 &  20 & 4025.01 & 1020.45 & TL & 0 & \textbf{3055.56} & 
		\textbf{3016.84} & TL & 252 & 3227.00 & TL \\ 
		rl1323 & 1323 &  30 & 3209.07 & 737.33 & TL & 0 & 2913.42 & 
		\textbf{2372.10} & TL & 224 & \textbf{2563.30} & TL \\ 
		rl1323 & 1323 &  40 & 2592.84 & 916.48 & TL & 0 & \textbf{2039.56} & 
		\textbf{1972.47} & TL & 482 & 2166.96 & TL \\ 
		rl1323 & 1323 &  50 & 2248.99 & 888.87 & TL & 0 & 1958.61 & 
		\textbf{1745.58} & TL & 1390 & \textbf{1907.69} & TL \\ 
		rl1323 & 1323 &  60 & 2027.10 & 861.87 & TL & 0 & \textbf{1710.60} & 
		\textbf{1566.59} & TL & 867 & 1735.40 & TL \\ 
		rl1323 & 1323 &  70 & 1868.99 & 823.70 & TL & 0 & 1647.07 & 
		\textbf{1415.82} & TL & 1425 & \textbf{1595.20} & TL \\ 
		rl1323 & 1323 &  80 & 1702.98 & 802.17 & TL & 0 & 1536.00 & 
		\textbf{1302.31} & TL & 1300 & \textbf{1440.89} & TL \\ 
		rl1323 & 1323 &  90 & 1576.08 & 782.66 & TL & 0 & \textbf{1329.66} & 
		\textbf{1210.04} & TL & 1833 & 1374.72 & TL \\ 
		rl1323 & 1323 & 100 & 1468.95 & 768.94 & TL & 0 & \textbf{1278.10} & 
		\textbf{1126.16} & TL & 2060 & 1293.63 & TL \\ 
		\bottomrule
	\end{tabular}
	\endgroup
\end{table}

\begin{table}[tbh]
	\centering
	\caption{Detailed results for instance set \texttt{TSPLIB} with $\alpha=3$, 
	part one. \label{ta:tsplib3}} 
	\begingroup\footnotesize
	\begin{tabular}{lrr|rrrr|rrrr|rr}
		\toprule
		& & & \multicolumn{4}{|c|}{$\ahavl$} & \multicolumn{4}{|c|}{$\bhvsl$} & 
		\multicolumn{2}{|c}{\cite{sanchez2022grasp}} \\ name & $|N|$ & $p$ & UB 
		& LB & t[s] & nBC & UB & LB & t[s]  & nBC & UB & t[s] \\ \midrule
		att48 & 48 &  10 & \textbf{2081.57} & \textbf{2081.57} & 10.41 & 31 & 
		\textbf{2081.57} & \textbf{2081.57} & \textbf{0.89} & 103 & 2186.31 
		&    6.72 \\ 
		att48 & 48 &  20 & \textbf{1283.35} & \textbf{1283.35} & 3.47 & 14 & 
		\textbf{1283.35} & \textbf{1283.35} & \textbf{0.53} & 48 & 1374.48 &    
		1.61 \\ 
		att48 & 48 &  30 & \textbf{949.29} & \textbf{949.29} & 1.53 & 0 & 
		\textbf{949.29} & \textbf{949.29} & \textbf{0.10} & 0 & 1011.66 &    
		0.54 \\ 
		att48 & 48 &  40 & \textbf{645.88} & \textbf{645.88} & 0.13 & 0 & 
		\textbf{645.88} & \textbf{645.88} & \textbf{0.06} & 0 & 675.00 &    
		0.08 \\ 
		eil101 & 101 &  10 & \textbf{29.43} & \textbf{29.43} & 368.66 & 1338 & 
		\textbf{29.43} & \textbf{29.43} & \textbf{15.66} & 1092 & 
		\textbf{29.43} &   92.44 \\ 
		eil101 & 101 &  20 & \textbf{17.80} & \textbf{17.80} & 363.10 & 3745 & 
		\textbf{17.80} & \textbf{17.80} & \textbf{18.45} & 2253 & 18.03 &   
		43.73 \\ 
		eil101 & 101 &  30 & \textbf{13.15} & \textbf{13.15} & 405.18 & 9499 & 
		\textbf{13.15} & \textbf{13.15} & \textbf{14.06} & 1567 & 14.14 &   
		19.37 \\ 
		eil101 & 101 &  40 & \textbf{11.18} & \textbf{11.18} & 179.44 & 6421 & 
		\textbf{11.18} & \textbf{11.18} & \textbf{3.39} & 800 & 12.04 &    9.71 
		\\ 
		eil101 & 101 &  50 & \textbf{9.43} & \textbf{9.43} & 20.20 & 424 & 
		\textbf{9.43} & \textbf{9.43} & \textbf{1.15} & 449 & 10.63 &    4.74 
		\\ 
		eil101 & 101 &  60 & \textbf{8.06} & \textbf{8.06} & 6.01 & 156 & 
		\textbf{8.06} & \textbf{8.06} & \textbf{0.45} & 220 & 9.06 &    2.22 \\ 
		eil101 & 101 &  70 & \textbf{7.28} & \textbf{7.28} & 2.28 & 0 & 
		\textbf{7.28} & \textbf{7.28} & \textbf{0.22} & 25 & 8.54 &    1.06 \\ 
		eil101 & 101 &  80 & \textbf{6.40} & \textbf{6.40} & 0.84 & 0 & 
		\textbf{6.40} & \textbf{6.40} & \textbf{0.12} & 0 & 7.28 &    0.44 \\ 
		eil101 & 101 &  90 & \textbf{5.00} & \textbf{5.00} & 0.43 & 0 & 
		\textbf{5.00} & \textbf{5.00} & \textbf{0.09} & 0 & 6.08 &    0.11 \\ 
		eil101 & 101 & 100 & \textbf{2.83} & \textbf{2.83} & 0.14 & 0 & 
		\textbf{2.83} & \textbf{2.83} & \textbf{0.07} & 0 & \textbf{2.83} &    
		0.05 \\ 
		ch150 & 150 &  10 & \textbf{297.96} & 205.79 & TL & 1556 & 
		\textbf{297.96} & \textbf{297.96} & \textbf{30.88} & 828 & 298.56 &  
		398.03 \\ 
		ch150 & 150 &  20 & 178.21 & 143.05 & TL & 3732 & \textbf{176.47} & 
		\textbf{176.47} & \textbf{90.44} & 4446 & 179.71 &  150.94 \\ 
		ch150 & 150 &  30 & 140.06 & 121.29 & TL & 11797 & \textbf{137.46} & 
		\textbf{137.46} & \textbf{1128.29} & 75456 & 146.41 &   78.08 \\ 
		ch150 & 150 &  40 & 114.58 & 108.32 & TL & 30854 & \textbf{114.47} & 
		\textbf{111.55} & TL & 157239 & 119.22 &   52.10 \\ 
		ch150 & 150 &  50 & \textbf{100.34} & 96.71 & TL & 38009 & 100.47 & 
		\textbf{98.04} & TL & 244637 & 108.03 &   26.70 \\ 
		ch150 & 150 &  60 & \textbf{90.58} & 86.79 & TL & 55403 & 
		\textbf{90.58} & \textbf{89.09} & TL & 379556 & 97.46 &   17.78 \\ 
		ch150 & 150 &  70 & \textbf{83.19} & 79.80 & TL & 78518 & 83.33 & 
		\textbf{81.91} & TL & 475149 & 92.82 &   13.10 \\ 
		ch150 & 150 &  80 & \textbf{74.93} & \textbf{74.93} & \textbf{182.44} & 
		6308 & \textbf{74.93} & \textbf{74.93} & 1784.19 & 621412 & 83.38 &    
		8.34 \\ 
		ch150 & 150 &  90 & \textbf{67.73} & \textbf{67.73} & \textbf{18.17} & 
		487 & \textbf{67.73} & \textbf{67.73} & 39.35 & 27266 & 79.81 &    4.75 
		\\ 
		ch150 & 150 & 100 & \textbf{63.42} & \textbf{63.42} & 8.67 & 19 & 
		\textbf{63.42} & \textbf{63.42} & \textbf{1.34} & 517 & 69.35 &    3.23 
		\\ 
		ch150 & 150 & 110 & \textbf{59.04} & \textbf{59.04} & 13.41 & 424 & 
		\textbf{59.04} & \textbf{59.04} & \textbf{2.39} & 1753 & 67.22 &    
		1.85 \\ 
		ch150 & 150 & 120 & \textbf{52.97} & \textbf{52.97} & 2.33 & 0 & 
		\textbf{52.97} & \textbf{52.97} & \textbf{0.55} & 0 & 61.29 &    0.95 
		\\ 
		ch150 & 150 & 130 & \textbf{44.46} & \textbf{44.46} & 1.09 & 0 & 
		\textbf{44.46} & \textbf{44.46} & \textbf{0.33} & 0 & 57.50 &    0.41 
		\\ 
		ch150 & 150 & 140 & \textbf{38.56} & \textbf{38.56} & 0.55 & 0 & 
		\textbf{38.56} & \textbf{38.56} & \textbf{0.30} & 0 & 52.20 &    0.16 
		\\ 
		\bottomrule
\end{tabular}
\endgroup
\end{table}
	
		\begin{table}[tbh]
			\centering
			\caption{Detailed results for instance set \texttt{TSPLIB} with 
			$\alpha=3$, 
				part two. \label{ta:tsplib4}} 
			\begingroup\footnotesize
			\begin{tabular}{lrr|rrrr|rrrr|rr}
				\toprule
				& & & \multicolumn{4}{|c|}{$\ahavl$} & 
				\multicolumn{4}{|c|}{$\bhvsl$} & 
				\multicolumn{2}{|c}{\cite{sanchez2022grasp}} \\ name & $|N|$ & 
				$p$ & UB 
				& LB & t[s] & nBC & UB & LB & t[s]  & nBC & UB & t[s] \\ 
				\midrule
		pr439 & 439 &  10 & 5997.08 & 2410.47 & TL & 582 & \textbf{4050.31} & 
		\textbf{4050.31} & \textbf{53.02} & 179 & 4076.23 & TL \\ 
		pr439 & 439 &  20 & 3505.89 & 1548.18 & TL & 769 & \textbf{2683.28} & 
		\textbf{2683.28} & \textbf{79.62} & 786 & 2726.03 & TL \\ 
		pr439 & 439 &  30 & 2520.17 & 1209.18 & TL & 1111 & \textbf{2065.49} & 
		\textbf{2065.49} & \textbf{1082.90} & 7280 & 2231.73 & TL \\ 
		pr439 & 439 &  40 & 2102.38 & 1014.28 & TL & 897 & \textbf{1600.78} & 
		\textbf{1600.78} & \textbf{1304.52} & 24066 & 1644.88 & TL \\ 
		pr439 & 439 &  50 & 1760.86 & 899.70 & TL & 1108 & \textbf{1350.00} & 
		\textbf{1350.00} & \textbf{302.98} & 4413 & 1467.35 & TL \\ 
		pr439 & 439 &  60 & 1550.00 & 753.42 & TL & 1040 & \textbf{1150.27} & 
		\textbf{1120.79} & TL & 19888 & 1340.01 & TL \\ 
		pr439 & 439 &  70 & 1308.63 & 707.85 & TL & 985 & \textbf{1006.23} & 
		\textbf{982.66} & TL & 20694 & 1231.11 & 1316.50 \\ 
		pr439 & 439 &  80 & 1129.71 & 651.34 & TL & 1396 & \textbf{915.49} & 
		\textbf{873.76} & TL & 18587 & 1217.58 &  955.74 \\ 
		pr439 & 439 &  90 & 1025.91 & 603.73 & TL & 1421 & \textbf{813.94} & 
		\textbf{752.26} & TL & 17207 & 986.47 &  723.38 \\ 
		rat575 & 575 &  10 & 220.06 & 44.70 & TL & 242 & \textbf{138.85} & 
		\textbf{138.85} & \textbf{19.41} & 0 & 140.52 & TL \\ 
		rat575 & 575 &  20 & 137.20 & 39.93 & TL & 300 & \textbf{93.43} & 
		\textbf{93.43} & \textbf{705.69} & 2335 & 94.64 & TL \\ 
		rat575 & 575 &  30 & 94.94 & 36.76 & TL & 396 & \textbf{72.09} & 
		\textbf{71.31} & TL & 8725 & 74.52 & 1101.33 \\ 
		rat575 & 575 &  40 & 82.04 & 34.21 & TL & 387 & 66.61 & \textbf{59.49} 
		& TL & 4803 & \textbf{64.88} &  950.51 \\ 
		rat575 & 575 &  50 & 73.00 & 32.58 & TL & 490 & 57.25 & \textbf{51.90} 
		& TL & 6189 & \textbf{56.94} &  717.39 \\ 
		rat575 & 575 &  60 & 63.29 & 31.91 & TL & 489 & 53.74 & \textbf{46.82} 
		& TL & 4661 & \textbf{51.35} &  595.10 \\ 
		rat575 & 575 &  70 & 55.15 & 30.95 & TL & 597 & \textbf{47.42} & 
		\textbf{42.44} & TL & 7099 & 47.85 &  494.17 \\ 
		rat575 & 575 &  80 & 50.96 & 30.29 & TL & 309 & 45.28 & \textbf{39.05} 
		& TL & 7689 & \textbf{44.29} &  448.23 \\ 
		rat575 & 575 &  90 & 48.30 & 29.73 & TL & 493 & 44.69 & \textbf{36.25} 
		& TL & 9866 & \textbf{41.11} &  319.13 \\ 
		rat575 & 575 & 100 & 44.27 & 29.16 & TL & 409 & 40.79 & \textbf{34.13} 
		& TL & 10831 & \textbf{38.63} &  247.96 \\ 
		rat783 & 783 &  10 & 254.92 & 42.13 & TL & 0 & \textbf{163.68} & 
		\textbf{163.68} & \textbf{54.58} & 0 & 166.23 & TL \\ 
		rat783 & 783 &  20 & 162.75 & 40.00 & TL & 46 & \textbf{109.57} & 
		\textbf{109.57} & \textbf{841.93} & 1016 & 112.70 & TL \\ 
		rat783 & 783 &  30 & 111.57 & 37.33 & TL & 83 & \textbf{83.55} & 
		\textbf{83.49} & TL & 3938 & 88.57 & TL \\ 
		rat783 & 783 &  40 & 97.00 & 35.01 & TL & 77 & 76.90 & \textbf{70.18} & 
		TL & 2296 & \textbf{76.03} & TL \\ 
		rat783 & 783 &  50 & 86.58 & 33.37 & TL & 67 & 68.66 & \textbf{60.76} & 
		TL & 2219 & \textbf{66.10} & TL \\ 
		rat783 & 783 &  60 & 74.33 & 32.64 & TL & 161 & 61.40 & \textbf{54.92} 
		& TL & 2864 & \textbf{60.02} & 1617.55 \\ 
		rat783 & 783 &  70 & 65.37 & 31.94 & TL & 70 & 59.03 & \textbf{50.25} & 
		TL & 2403 & \textbf{55.44} & 1642.05 \\ 
		rat783 & 783 &  80 & 60.61 & 31.34 & TL & 91 & 56.14 & \textbf{46.10} & 
		TL & 3000 & \textbf{51.66} & 1420.24 \\ 
		rat783 & 783 &  90 & 56.04 & 30.95 & TL & 70 & 50.49 & \textbf{43.09} & 
		TL & 4618 & \textbf{48.47} & 1211.55 \\ 
		rat783 & 783 & 100 & 53.14 & 30.47 & TL & 65 & 47.76 & \textbf{40.36} & 
		TL & 4175 & \textbf{45.88} & 1019.60 \\ 
		pr1002 & 1002 &  10 & 6435.06 & 1251.44 & TL & 0 & \textbf{5202.16} & 
		\textbf{5202.16} & \textbf{107.49} & 0 & 5331.28 & TL \\ 
		pr1002 & 1002 &  20 & 4606.52 & 1169.60 & TL & 0 & \textbf{3170.57} & 
		\textbf{3170.57} & \textbf{135.59} & 44 & 3290.14 & TL \\ 
		pr1002 & 1002 &  30 & 3431.11 & 1073.20 & TL & 0 & \textbf{2631.54} & 
		\textbf{2502.77} & TL & 2470 & 2644.33 & TL \\ 
		pr1002 & 1002 &  40 & 2983.29 & 1042.25 & TL & 0 & \textbf{2210.20} & 
		\textbf{2140.22} & TL & 3850 & 2304.89 & TL \\ 
		pr1002 & 1002 &  50 & 2562.23 & 996.93 & TL & 11 & 2015.56 & 
		\textbf{1841.90} & TL & 3006 & \textbf{2013.08} & TL \\ 
		pr1002 & 1002 &  60 & 2241.09 & 952.64 & TL & 0 & 1874.17 & 
		\textbf{1681.42} & TL & 3965 & \textbf{1838.48} & TL \\ 
		pr1002 & 1002 &  70 & 2015.56 & 935.16 & TL & 15 & 1732.77 & 
		\textbf{1507.48} & TL & 2889 & \textbf{1710.26} & TL \\ 
		pr1002 & 1002 &  80 & 1860.78 & 899.33 & TL & 3 & 1565.25 & 
		\textbf{1391.70} & TL & 3657 & \textbf{1518.22} & TL \\ 
		pr1002 & 1002 &  90 & 1718.28 & 862.03 & TL & 9 & \textbf{1431.36} & 
		\textbf{1283.52} & TL & 5200 & 1442.22 & TL \\ 
		pr1002 & 1002 & 100 & 1569.24 & 842.84 & TL & 9 & 1414.21 & 
		\textbf{1208.88} & TL & 5803 & \textbf{1353.70} & TL \\ 
		rl1323 & 1323 &  10 & 8524.65 & 1470.07 & TL & 0 & \textbf{6229.60} & 
		\textbf{6193.80} & TL & 0 & 6313.82 & TL \\ 
		rl1323 & 1323 &  20 & 5699.21 & 1041.99 & TL & 0 & \textbf{3845.66} & 
		\textbf{3832.86} & TL & 440 & 4032.83 & TL \\ 
		rl1323 & 1323 &  30 & 3992.96 & 973.42 & TL & 0 & 3906.16 & 
		\textbf{2984.30} & TL & 111 & \textbf{3204.16} & TL \\ 
		rl1323 & 1323 &  40 & 3375.04 & 923.26 & TL & 0 & \textbf{2652.14} & 
		\textbf{2502.04} & TL & 388 & 2774.72 & TL \\ 
		rl1323 & 1323 &  50 & 2963.63 & 913.79 & TL & 0 & \textbf{2308.32} & 
		\textbf{2198.20} & TL & 766 & 2430.27 & TL \\ 
		rl1323 & 1323 &  60 & 2505.42 & 848.05 & TL & 0 & 2495.02 & 
		\textbf{1947.60} & TL & 407 & \textbf{2149.14} & TL \\ 
		rl1323 & 1323 &  70 & 2317.57 & 876.96 & TL & 0 & \textbf{1918.35} & 
		\textbf{1778.52} & TL & 1047 & 1997.22 & TL \\ 
		rl1323 & 1323 &  80 & 2144.00 & 863.91 & TL & 0 & 1973.72 & 
		\textbf{1646.13} & TL & 1105 & \textbf{1842.10} & TL \\ 
		rl1323 & 1323 &  90 & 2025.01 & 835.69 & TL & 0 & 1751.21 & 
		\textbf{1530.92} & TL & 1031 & \textbf{1745.58} & TL \\ 
		rl1323 & 1323 & 100 & 1890.05 & 811.18 & TL & 0 & 1624.22 & 
		\textbf{1429.61} & TL & 1185 & \textbf{1620.92} & TL \\ 
		\bottomrule
	\end{tabular}
	\endgroup
\end{table}

In Table \ref{ta:pmed} we provide a comparison with the local search of 
\cite{mousavi2023exploiting}. The runs in \cite{mousavi2023exploiting} were 
made on an Intel Core i5-6200 with 2.3 GHz CPU and 8 GB of RAM. We note that 
\cite{mousavi2023exploiting} presents runtimes for different version of their 
developed heuristics, in the table we show the fastest runtime and the best 
objective function value found by the heuristics.
The results show that \bhvsl can solve all instances to optimality under one 
minute, while for two of the instances the heuristics of 
\cite{mousavi2023exploiting} do not manage to find the optimal solution. 
Similar to the instance set \texttt{TSPLIB}, the setting \ahavl performs worse 
than \bhvsl.

\begin{table}[tbh]
	\centering
	\caption{Detailed results for instance set \texttt{pmed} \blue{with 
	$\alpha=2$}. \label{ta:pmed}} 
	\begingroup\footnotesize
	\begin{tabular}{lrr|rrrr|rrrr|rr}
		\toprule
		& & & \multicolumn{4}{|c|}{$\ahavl$} & \multicolumn{4}{|c|}{$\bhvsl$} & 
		\multicolumn{2}{|c}{\cite{mousavi2023exploiting}} \\ name & $|N|$ & $p$ 
		& UB & LB & t[s] & nBC & UB & LB & t[s] & nBC & UB & t[s] \\ \midrule
		pmed1 & 100 &   5 & \textbf{150} & \textbf{150} & 24.16 & 103 & 
		\textbf{150} & \textbf{150} & \textbf{0.34} & 0 & \textbf{150} & 0.01 
		\\ 
		pmed2 & 100 &  10 & \textbf{121} & \textbf{121} & 20.73 & 12 & 
		\textbf{121} & \textbf{121} & \textbf{0.31} & 0 & \textbf{121} & 0.20 
		\\ 
		pmed3 & 100 &  10 & \textbf{121} & \textbf{121} & 55.80 & 325 & 
		\textbf{121} & \textbf{121} & \textbf{0.54} & 32 & \textbf{121} & 0.26 
		\\ 
		pmed4 & 100 &  20 & \textbf{ 97} & \textbf{ 97} & 51.19 & 545 & 
		\textbf{ 97} & \textbf{ 97} & \textbf{0.91} & 345 & \textbf{ 97} & 8.19 
		\\ 
		pmed5 & 100 &  33 & \textbf{ 63} & \textbf{ 63} & 4.83 & 0 & \textbf{ 
		63} & \textbf{ 63} & \textbf{0.23} & 0 & \textbf{ 63} & 0.02 \\ 
		pmed6 & 200 &   5 & \textbf{ 99} & \textbf{ 99} & 1009.08 & 2451 & 
		\textbf{ 99} & \textbf{ 99} & \textbf{0.37} & 0 & \textbf{ 99} & 0.03 
		\\ 
		pmed7 & 200 &  10 &  85 &  74 & TL & 4207 & \textbf{ 80} & \textbf{ 80} 
		& \textbf{0.91} & 7 & \textbf{ 80} & 0.09 \\ 
		pmed8 & 200 &  20 & \textbf{ 70} &  66 & TL & 3839 & \textbf{ 70} & 
		\textbf{ 70} & \textbf{1.10} & 63 & \textbf{ 70} & 0.03 \\ 
		pmed9 & 200 &  40 & \textbf{ 49} & \textbf{ 49} & 1725.43 & 3205 & 
		\textbf{ 49} & \textbf{ 49} & \textbf{0.89} & 54 & \textbf{ 49} & 0.73 
		\\ 
		pmed10 & 200 &  67 & \textbf{ 28} & \textbf{ 28} & 22.72 & 7 & \textbf{ 
		28} & \textbf{ 28} & \textbf{0.51} & 0 & \textbf{ 28} & 0.62 \\ 
		pmed11 & 300 &   5 &  73 &  55 & TL & 1793 & \textbf{ 68} & \textbf{ 
		68} & \textbf{0.59} & 0 & \textbf{ 68} & 0.00 \\ 
		pmed12 & 300 &  10 &  72 &  53 & TL & 1851 & \textbf{ 60} & \textbf{ 
		60} & \textbf{1.16} & 0 & \textbf{ 60} & 0.27 \\ 
		pmed13 & 300 &  30 &  47 &  41 & TL & 1760 & \textbf{ 43} & \textbf{ 
		43} & \textbf{2.38} & 70 & \textbf{ 43} & 2.07 \\ 
		pmed14 & 300 &  60 &  38 &  33 & TL & 3177 & \textbf{ 34} & \textbf{ 
		34} & \textbf{2.99} & 148 & \textbf{ 34} & 0.93 \\ 
		pmed15 & 300 & 100 &  24 & \textbf{ 23} & TL & 4287 & \textbf{ 23} & 
		\textbf{ 23} & \textbf{1.85} & 156 & \textbf{ 23} & 6.86 \\ 
		pmed16 & 400 &   5 &  56 &  46 & TL & 894 & \textbf{ 52} & \textbf{ 52} 
		& \textbf{0.94} & 0 & \textbf{ 52} & 0.24 \\ 
		pmed17 & 400 &  10 &  56 &  39 & TL & 499 & \textbf{ 45} & \textbf{ 45} 
		& \textbf{3.00} & 48 & \textbf{ 45} & 0.04 \\ 
		pmed18 & 400 &  40 &  44 &  32 & TL & 903 & \textbf{ 34} & \textbf{ 34} 
		& \textbf{4.25} & 75 & \textbf{ 34} & 17.76 \\ 
		pmed19 & 400 &  80 &  29 &  23 & TL & 1242 & \textbf{ 24} & \textbf{ 
		24} & \textbf{12.88} & 836 &  25 & 0.17 \\ 
		pmed20 & 400 & 133 &  22 &  18 & TL & 1941 & \textbf{ 19} & \textbf{ 
		19} & \textbf{3.88} & 273 & \textbf{ 19} & 1.24 \\ 
		pmed21 & 500 &   5 &  59 &  34 & TL & 310 & \textbf{ 45} & \textbf{ 45} 
		& \textbf{1.96} & 4 & \textbf{ 45} & 1.20 \\ 
		pmed22 & 500 &  10 &  52 &  34 & TL & 247 & \textbf{ 44} & \textbf{ 44} 
		& \textbf{4.22} & 10 & \textbf{ 44} & 0.42 \\ 
		pmed23 & 500 &  50 &  36 &  25 & TL & 399 & \textbf{ 27} & \textbf{ 27} 
		& \textbf{7.08} & 52 & \textbf{ 27} & 11.18 \\ 
		pmed24 & 500 & 100 &  23 & \textbf{ 19} & TL & 511 & \textbf{ 19} & 
		\textbf{ 19} & \textbf{27.88} & 1081 &  20 & 0.54 \\ 
		pmed25 & 500 & 167 &  19 & \textbf{ 15} & TL & 559 & \textbf{ 15} & 
		\textbf{ 15} & \textbf{15.84} & 1501 & \textbf{ 15} & 33.68 \\ 
		pmed26 & 600 &   5 &  57 &  35 & TL & 202 & \textbf{ 43} & \textbf{ 43} 
		& \textbf{2.21} & 0 & \textbf{ 43} & 0.24 \\ 
		pmed27 & 600 &  10 &  44 &  29 & TL & 198 & \textbf{ 36} & \textbf{ 36} 
		& \textbf{3.42} & 0 & \textbf{ 36} & 0.09 \\ 
		pmed28 & 600 &  60 &  28 &  21 & TL & 199 & \textbf{ 22} & \textbf{ 22} 
		& \textbf{7.46} & 30 & \textbf{ 22} & 0.59 \\ 
		pmed29 & 600 & 120 &  22 &  16 & TL & 301 & \textbf{ 17} & \textbf{ 17} 
		& \textbf{11.25} & 78 & \textbf{ 17} & 0.32 \\ 
		pmed30 & 600 & 200 &  17 & \textbf{ 13} & TL & 494 & \textbf{ 13} & 
		\textbf{ 13} & \textbf{11.08} & 500 & \textbf{ 13} & 2.89 \\ 
		pmed31 & 700 &   5 &  47 &  28 & TL & 0 & \textbf{ 34} & \textbf{ 34} & 
		\textbf{3.26} & 0 & \textbf{ 34} & 0.05 \\ 
		pmed32 & 700 &  10 &  46 &  26 & TL & 0 & \textbf{ 33} & \textbf{ 33} & 
		\textbf{5.59} & 3 & \textbf{ 33} & 0.21 \\ 
		pmed33 & 700 &  70 &  24 &  17 & TL & 131 & \textbf{ 19} & \textbf{ 19} 
		& \textbf{13.94} & 40 & \textbf{ 19} & 10.28 \\ 
		pmed34 & 700 & 140 &  18 &  13 & TL & 158 & \textbf{ 14} & \textbf{ 14} 
		& \textbf{54.78} & 981 & \textbf{ 14} & 97.77 \\ 
		pmed35 & 800 &   5 &  43 &  26 & TL & 0 & \textbf{ 34} & \textbf{ 34} & 
		\textbf{4.37} & 0 & \textbf{ 34} & 0.54 \\ 
		pmed36 & 800 &  10 &  49 &   0 & TL & 0 & \textbf{ 31} & \textbf{ 31} & 
		\textbf{9.74} & 3 & \textbf{ 31} & 0.25 \\ 
		pmed37 & 800 &  80 &  24 &  17 & TL & 25 & \textbf{ 18} & \textbf{ 18} 
		& \textbf{35.20} & 210 &  19 & 0.12 \\ 
		pmed38 & 900 &   5 &  54 &  23 & TL & 0 & \textbf{ 33} & \textbf{ 33} & 
		\textbf{7.89} & 0 & \textbf{ 33} & 0.09 \\ 
		pmed39 & 900 &  10 &  39 &  21 & TL & 0 & \textbf{ 26} & \textbf{ 26} & 
		\textbf{11.40} & 13 & \textbf{ 26} & 0.18 \\ 
		pmed40 & 900 &  90 &  22 &  14 & TL & 0 & \textbf{ 16} & \textbf{ 16} & 
		\textbf{19.91} & 44 & \textbf{ 16} & 3.04 \\ 
		\bottomrule
	\end{tabular}
	\endgroup
\end{table}
\section{Conclusions} 
In this work, we present two integer programming formulations for the discrete 
version of the $\alpha$-neighbor $p$-center problem (\APCP), which is an 
emerging variant of the classical discrete $p$-center problem (\PCP), which 
recently 
got attention in literature. We also present lifting procedures for 
inequalities in the formulations, valid 
inequalities, 
optimality-preserving inequalities and variable fixing procedures. We provide 
theoretical results on the strength of the formulations and convergence results 
for the lower bounds obtained after applying the lifting procedures or the 
variable fixing procedures in an iterative fashion. These results extend 
results obtained by \citet{elloumi2004} and \citet{gaar2022scaleable} for the 
\PCP. 
Based on these results we provide two branch-and-cut algorithms, namely one 
based on each of the two formulations. 

We assess the efficacy of our branch-and-cut algorithms in a computational 
study on instances from the literature. The results show that our exact 
algorithms 
outperforms existing algorithms for the \APCP. These existing algorithms are 
heuristics, 
namely a GRASP by \citet{sanchez2022grasp} and a local search by 
\citet{mousavi2023exploiting}. Our algorithms manage to solve 116 of 194 
instances from literature to proven optimality within a time limit of 1800 
seconds, in fact many of them are solved to optimality within 60 seconds. They 
also provide improved best solution values for 116 instances from literature. 
Note that these 116 instances are not the same instances as the instances where 
optimality is proven, as for some of the latter instances the existing 
heuristics already manage to find the optimal solution (but of course can not 
prove optimality, as they are heuristics).

There are various directions for further work. One direction could be to try to 
derive 
further valid inequalities. In particular it could be interesting to 
investigate if there are inequalities which ensure that the best possible 
bounds 
of both formulations coincide, i.e., if the second formulation can be 
further strengthened, as our current results show that the best bound of the 
first formulation could be better for some instances. 
Another interesting avenue could be the development of a 
projection-based approach similar to the one of \citet{gaar2022scaleable} for 
the \PCP, in which a lower number of variables suffices to model the problem 
and which is therefore better suited for large scale instances.  

Furthermore, trying to extend the approaches 
including the lifting schemes to other variants of the \PCP such as robust 
versions (see, e.g., \citep{lu2013robust}), capacitated versions (see, e.g., 
\citep{scaparra2004large}) or the 
$p$-next center problem (see, 
e.g., 
\citep{lopez2019grasp}) could be fruitful. Moreover, while we managed to 
improve 
many of the 
best known solution values for the instances from literature, there are also 
some instances where the existing heuristic work better. Thus further 
developments of heuristics can also be interesting, including matheuristics 
such as local branching (see, e.g., \citep{fischetti2003local}) which could 
exploit our 
formulations.

\label{sec:conclusion}

\section*{Acknowledgments}

This research was funded in whole, or in part, by the Austrian Science Fund 
(FWF)
[P 35160-N]. For the purpose of open access, the author has applied a CC BY 
public 
copyright licence to any Author Accepted Manuscript version arising from this submission. It is also supported by the Johannes Kepler University Linz, Linz Institute
of Technology (Project LIT-2021-10-YOU-216).

\clearpage

\bibliographystyle{elsarticle-harv}
\bibliography{biblio}

\appendix
\newpage
\blue{
\section{Formulations for the \PCP from the literature \label{sec:app1}}

A formal definition of the \PCP is as follows. Given an integer $p$, a 
set of customer 
demand points $\cus$ with cardinality $|\cus|=n$, a set of potential facility 
locations $\loc$ of cardinality $|\loc|=m \geq p$ and a distance $d_{ij}$ from 
a 
customer 
demand point $i$ to the potential facility location $j$ for every $i \in \cus$ 
and 
$j \in \loc$,
find a subset $S \subseteq \loc$ with cardinality $|S|=p$ of facilities to 
\emph{open} such that the maximum 
distance between a customer demand point and its closest 
open facility is minimized, i.e., such that 
$\max_{i \in \cus} \min_{j \in S} \{d_{ij} \}$ is minimized. \green{We note 
that in the \APCP, we have $N=I=J$ by definition of the problem. This is 
necessary as the set of demand points (i.e., customers) in the \APCP depends on 
a given feasible solution and is defined 
as all points where no facility is opened in the solution. Due to this 
difference, slightly 
modified definitions of $D$ and $D_i$ \orange{are necessary for the \PCP below 
as compared to the ones of \APCP above.}}

Let the binary 
variables $y_j$ for all $j\in \loc$ indicate whether a facility 
is 
opened at location $j$. Let the binary variables $x_{ij}$ for all $i\in \cus, j 
\in \loc$ indicate whether the 
customer $i \in \cus$ 
is assigned to the open facility $j$. Let the continuous variables $z$ measure 
the 
distance in the objective function.
The classical textbook formulation \elli{of the \PCP} (see for 
example~\citet{daskin2013network}) is as follows.
\begin{subequations}
	\begin{alignat}{3}
	\mytag{PC1} \qquad
	& \min & \obj \phantom{iiiii} \label{xxx:pc1:z}  \\ 
	& \st~ &  \sum_{j \in \loc} y_j &= p \label{xxx:pc1:sumy} \\       
	&& \sum_{j \in \loc} x_{ij} & =  1 && \forall i \in \cus 
	\label{xxx:pc1:sumx} \\
	&& x_{ij} &\leq y_j && \forall i \in \cus, \forall j \in 
	\loc\label{xxx:pc1:xy}\\   
	&& \sum_{j \in \loc} d_{ij} x_{ij} & \leq  \obj && \forall i \in \cus 
	\label{pc1:sumdx}\\
	&& x_{ij} &\in  \{0,1\} \qquad&& \forall i \in \cus, \forall j \in \loc  
	\label{xxx:pc1:xbin}\\
	&& y_{j} &\in  \{0,1\} && \forall j \in \loc  \label{xxx:pc1:ybin}\\
	&& \obj & \in \mathbb{R}
	\end{alignat}
\end{subequations}


In \cite{elloumi2004} another 
formulation was introduced: 
let $D=\{d_{ij}: i \in 
\cus, j \in \loc\}$  denote the set of all possible distances 
and let $d_1$, $\ldots$, $d_{K}$ be the values contained in $D$, so 
$D=\{d_1, \ldots, d_{K}\}$. 
Furthermore there is a binary variable for each 
value in $D$ that indicates whether the optimal value of 
\elli{the} \PCP is less or equal than this value. 
Towards this end let  $u_k = 0$ if all customers have an open facility with 
distance at most $d_{k-1}$, otherwise $u_k = 1$ for all 
$k \in \{2, \dots, K\}$. \elli{Then the} formulation reads as follows.
\begin{subequations}
	\begin{alignat}{3}
	\mytag{PCE} \qquad
	& \min & d_1 + \sum_{k=2}^{K} (d_k &- d_{k-1})u_k  
	\label{xxx:pcE:ojb}  
	\\ 
	& \st~ &  \sum_{j \in \loc} y_j &\leq p \label{xxx:pcE:sumyp} \\       
	&&\sum_{j \in \loc} y_j &\geq 1 \label{xxx:pcE:sumy1} \\  
	&& u_k + \sum_{j\colon d_{ij} < d_k} y_{j} &\geq  1 && \forall i \in \cus, 
	\forall k \in \{2, \dots, K\} \label{xxx:pcE:sumyu} \\
	&& u_{k} &\in  \{0,1\} \qquad&& \forall  k \in \{2, \dots, 
	K\}\label{xxx:pcE:ubin}\\
	&& y_{j} &\in  \{0,1\} && \forall j \in \loc  
	\end{alignat}
\end{subequations}

\green{Let $D_i = \{d_{ij}: j \in \loc \}\setminus \{d_1\}$ for $i \in \cus$.}
The modified variant of \myref{PCE} proposed by \citet{ales2018} is as
follows.
\begin{subequations}
	\begin{alignat}{3}
\mytag{PCA} \qquad
& \min & d_1 + \sum_{k=2}^{K} (d_k &- d_{k-1})u_k  \label{pcA:ojb}  \\ 
& \st~ &  \sum_{j \in \loc} y_j &\leq p \label{pcA:sumyp} \\       
&&\sum_{j \in \loc} y_j &\geq 1 \label{pcA:sumy1} \\ 
&& u_{k-1} &\geq u_{k} && \forall k \in \{3, \dots, K\} \label{pcA:ulink} 
\\  
&& u_k + \sum_{j \in \loc :  d_{ij} < d_k} 
y_{j} &\geq 1\qquad && \forall i \in \cus, 
\forall d_k \in D_i \elli{\cup \{K\}} \label{pcA:sumyu} \\
&& u_{k} &\in  \{0,1\} \qquad&& \forall  k \in \{2, \dots, K\} 
\label{pcA:ubin}	\\
&& y_{j} &\in  \{0,1\} && \forall j \in \loc
\label{pcA:ybin}
\end{alignat}
\end{subequations}
}

\end{document}